\documentclass[12pt,a4paper]{article}
\usepackage[utf8]{inputenc}
\usepackage{amsmath}
\usepackage{amsfonts}
\usepackage{amssymb}
\usepackage{amsthm}
\usepackage[colorlinks=true,linkcolor=blue]{hyperref}
\hypersetup{
  citecolor={blue}
} 
\usepackage[left=2cm,right=2cm,top=2cm,bottom=2cm]{geometry}
\usepackage[dvipsnames]{xcolor}
\usepackage[all,cmtip]{xy}
\usepackage{tikz}
\usetikzlibrary{decorations.markings}
\usepackage{enumerate}
\usepackage{enumitem}
\setlist[enumerate]{label = (\roman*)}
\usepackage{array}
\usepackage{float}
\graphicspath{{bilder/}}
\newcolumntype{C}[1]{>{\centering\arraybackslash}m{#1}}

\newcommand{\U}{\mathrm{U}}
\newcommand{\SU}{\mathrm{SU}}

\newcommand{\cT}{\mathcal{T}}

\newcommand{\RR}{\mathbb{R}}
\newcommand{\CC}{\mathbb{C}}
\newcommand{\QQ}{\mathbb{Q}}
\newcommand{\ZZ}{\mathbb{Z}}

\newcommand{\PP}{\mathbb{P}}
\newcommand{\Id}{\mathrm{Id}}
\newcommand{\Ann}{\mathrm{Ann}}
\newcommand{\depth}{\mathrm{depth}}
\newcommand{\Th}{\mathrm{Th}}
\newcommand{\im}{\mathrm{im}}

\newtheorem{thm}{Theorem}[section]
\newtheorem{prop}[thm]{Proposition}
\newtheorem{cor}[thm]{Corollary}
\newtheorem{lem}[thm]{Lemma}

\theoremstyle{definition}
\newtheorem{rem}[thm]{Remark}
\newtheorem{defn}[thm]{Definition}
\newtheorem{ex}[thm]{Example}

\begin{document}
\title{The GKM correspondence in dimension $6$}
\author{Oliver Goertsches\footnote{Philipps-Universit\"at Marburg, email:
goertsch@mathematik.uni-marburg.de}, Panagiotis Konstantis\footnote{Philipps-Universit\"at Marburg,
email: pako@mathematik.uni-marburg.de}, and Leopold
Zoller\footnote{Ludwig-Maximilians-Universit\"at M\"unchen, email: leopold.zoller@mathematik.uni-muenchen.de}}

\maketitle

\begin{abstract}
It follows from the GKM description of equivariant cohomology that the GKM graph of a GKM manifold has free equivariant graph cohomology, and satisfies a Poincar\'e duality condition. We prove that these conditions are sufficient for an abstract $3$-valent $T^2$-GKM graph to be realizable by a simply-connected $6$-dimensional GKM manifold. Our realization has the property that any closed stratum of a finite isotropy group contains a fixed point. Furthermore, we argue that in case there exists a fixed point in whose vicinity there occur at most two distinct finite nontrivial isotropy groups such a realization is unique up to equivariant homeomorphism, thus establishing a complexity one GKM correspondence in dimension $6$. We show that the statement on equivariant uniqueness is false without the two conditions on the finite isotropies by providing counterexamples in presence of a fixed point with three distinct neighbouring finite isotropy groups, as well as an example of a simply-connected integer GKM manifold with a closed stratum of a finite isotropy group which does not contain any fixed point.
\end{abstract}
{
\hypersetup{linkcolor=Blue}
\tableofcontents
}

\section{Introduction}

The idea to understand torus actions via associated combinatorial invariants has a long history. In any of the categories one traditionally considers -- Hamiltonian torus actions on symplectic manifolds, algebraic torus actions on complex algebraic varieties, or smooth actions on differentiable manifolds -- the outcome of the theory is that for actions of tori of maximal rank the situation is as clean as possible: the action is uniquely determined by the invariants up to equivariant isomorphism, and one can associate an action to any prescribed admissible set of invariants. In the symplectic setting, this amounts to the Delzant correspondence associating to any toric symplectic manifold its momentum polytope \cite{Delzant}; in the complex algebraic setting one considers the fan of a toric variety, see for example \cite[Chapter 5]{ToricTopology}. In the smooth setting there are various topological analogues of these theories like those of quasitoric \cite{DavisJ}, torus \cite{MasudaUnitary, MasudaPanov}, or topological toric manifolds \cite{IshidaFM}.

More recently, torus actions of complexity one, i.e., for which the dimension of the acting torus is one less than maximal, have attracted attention. Among the properties that were investigated are the orbit spaces of such actions, see \cite{KarshonTolmanQuotients} for the Hamiltonian case and \cite{AyzenbergMasuda, AyzenbergCherepanov} for the topological case, or their equivariant cohomology algebras, see \cite{HolmKessler}. The subclass of tall (symplectic) complexity one spaces was classified in \cite{KarshonTolman}.

In order to deal with actions of tori of arbitrary dimension, another theory has emerged in recent years. While in the toric setting the combinatorial object associated to the action is a polytope or similar, in GKM theory \cite{GKM} one relaxes the assumptions insofar as to retain only a labelled graph, the so-called GKM graph. In the case of a toric symplectic manifold this graph coincides with the one-skeleton of the momentum polytope, and in the topological setting with the one-skeleton of the orbit space. This theory allows for much greater flexibility: there are many spaces that do not allow for a toric action, but admit an action of GKM type, such as homogeneous spaces of equal rank \cite{GHZhom} or certain biquotients \cite{GKZsympbiquot, GKZdim6}. The main appeal of the GKM setting is that, despite the added flexibility, the GKM graph does still encode a lot of the topology of the manifold. In particular it encodes the entire equivariant as well as the nonequivariant cohomology in terms of combinatorial formulas (cf. \cite{GKM}). Note that we consider GKM theory in the general setting of smooth orientable manifolds; some authors additionally assume an invariant almost complex structure, see Remark \ref{rem:signedgraph} below. Independently from the geometric setup, the notion of an abstract GKM graph was introduced in \cite{GuilleminZaraII} and studied from a purely combinatorial perspective.

Two natural questions about the GKM correspondence
\[
\{ \textrm{GKM actions} \} \longrightarrow \{ \textrm{abstract GKM graphs} \}
\]
that sends a GKM action to its GKM graph arise: the rigidity question asks in how far a GKM action is determined by its GKM graph; the realization question asks whether an abstract GKM graph is realized by a GKM action.

Concerning the rigidity question, we showed in \cite{GKZrigid} that in dimensions $8$ and higher, the GKM graph does not determine the homotopy type of the manifold acted on: we explicitly constructed two GKM actions (of complexity one) with the same GKM graph on the total spaces of the two $S^2$-bundles over $S^6$, whose fifth homotopy group differs by a $\mathbb{Z}_2$ factor. Using the diffeomorphism classification in dimension $6$ by Wall \cite{Wall}, Jupp \cite{Jupp}, and \v{Z}ubr \cite{Zubr} we were able to show, however, that in dimension $6$ the GKM graph determines even the (nonequivariant) diffeomorphism type \cite{GKZdim6}. As a corollary, Tolman's \cite{Tolman} and Woodward's \cite{Woodward} examples of Hamiltonian non-K\"ahler manifolds were seen to be diffeomorphic to Eschenburg's twisted flag manifold, and hence admit a nonequivariant K\"ahler structure.

Concerning the realization question, not much is known beyond the
classical correspondences mentioned above. It was observed in \cite[Section
3.1]{GuilleminZaraII} that a given (signed) GKM graph can be thickened to an
open manifold on which a torus acts with the graph as orbit space of the
one-skeleton. In \cite{Ayzenberg} certain characteristic data is realized as a topological space on which a torus acts with complexity one and connected stabilizers, see \cite[Construction 3.5]{Ayzenberg}. In the realm of closed smooth manifolds however, as far as we are aware,
the only known result is that certain $3$-valent GKM fiber bundles
\cite{GKMFiberBundles} can be realized as projectivizations of complex rank $2$
vector bundles over quasitoric $4$-manifolds or $S^4$, see \cite{GKZ}.

In this paper, we focus on the general GKM realization and rigidity problem for complexity one actions in dimension $6$. When aiming for a correspondence between geometry and combinatorics it is reasonable to restrict to simply-connected spaces on the geometric side as otherwise no rigidity can be expected.
Our main result is (cf. Theorems \ref{thm:rationalGKM}, \ref{thm:mainintegerGKM}, and \ref{thm:rigidity} as well as Remark \ref{rem:condition(a)}):
\begin{thm}\label{thm:mainthm}
\begin{enumerate}
\item An abstract effective 3-valent $T^2$-GKM graph $(\Gamma,\alpha)$ is realizable geometrically by a $6$-dimensional simply-connected rational GKM manifold if and only if the graph cohomology $H^*(\Gamma,\alpha;\mathbb{Q})$ satisfies Poincaré duality. It is realizable by a simply-connected integer GKM manifold if and only if additionally the equivariant graph cohomology $H^*_{T^2}(\Gamma,\alpha;\mathbb{Z})$ is a free module over $H^*(BT^2;\mathbb{Z})$. In both cases there is such a realization satisfying condition
\begin{enumerate}
\item[(a)] every closed stratum of a finite isotropy group contains a $T^2$-fixed point.
\end{enumerate}
\item If two simply-connected integer $T^2$-GKM manifolds have the same GKM graph and satisfy condition (a) as well as
\begin{enumerate}
\item[(b)] there exists a $T^2$-fixed point in whose vicinity there occur at most two distinct finite nontrivial isotropy groups,
\end{enumerate} then they are equivariantly homeomorphic.
\end{enumerate}
\end{thm}

By a closed stratum of a finite isotropy group in condition (a) we mean a connected component $N$ of the fixed point set of a finite subgroup $H\subset T$ such that $H$ is the principal isotropy group on $N$. Condition (b) can be reformulated as the condition that at one (and hence any) $T^2$-fixed point there exist two weights of the isotropy representation that form an integer basis (cf.\ Remark \ref{rem:condition(b)}). In particular, unlike condition (a), condition (b) manifests directly in the GKM graph.
 The theorem implies that one has a bijective correspondence
\begin{center}
\begin{tikzpicture}
\node (a) at (0,0) {$\left\{\begin{array}{c}
\text{6-dim.\ simply-connected}\\
\text{integer $T^2$-GKM manifolds whose}\\
\text{finite isotropies satisfy (a) and (b)}\end{array}\right\}$};

\node (b) at (8.7,0) {$\left\{\begin{array}{c}
\text{3-valent } T^2\text{-GKM graphs with Poincaré}\\
\text{duality, free equivariant cohomology,}\\
\text{s.t.\ some adjacent weights form a basis}\end{array}\right\}$};

\draw[<->] (a) -- (b);
\end{tikzpicture}
\end{center}
where the left hand side is up to equivariant homeomorphism and the right hand side is up to graph isomorphisms preserving the labels.

We stress again that the injectivity of this correspondence is a low-dimensional phenomenon, by the main result of \cite{GKZrigid}. The fact that the left hand side is up to equivariant homeomorphism as opposed to equivariant diffeomorphism is due to the fact that, on the regular stratum, $T^2$-compatible smooth structures correspond to smooth structures on the orbit space. Hence rigidity up to equivariant diffeomorphism is connected to uniqueness of smooth structures in dimension $4$ (cf.\ Remark \ref{rem:smoothuniqueness}). 

Roughly speaking, the proofs of both our realization and rigidity result make use of an equivariant handle decomposition that starts with a thickening of the one-skeleton. We would like to emphasize that in case all isotropy groups are connected, the rigidity part of the statement, i.e., the first sentence of (ii), can also be proven using the results of \cite{Ayzenberg}, as outlined in the beginning of Section \ref{sec:rigidity} below.

As an immediate consequence of the above correspondence we answer the open question whether Tolman's and Woodward's examples mentioned above are equivariantly homeomorphic, see Corollary \ref{cor:tolwood}: \begin{cor}
Tolman's and Woodward's examples are equivariantly homeomorphic to Eschen\-burg's twisted flag manifold.
\end{cor}

Furthermore we prove that the rigidity in the above correspondence fails if
either one of conditions (a) or (b) are violated (cf.\ Theorem
\ref{thm:nonrigidity}):

\begin{thm}
There are two simply-connected $6$-dimensional integer GKM $T^2$-manifolds which are not equivariantly homeomorphic but have the same GKM graph. The examples can be chosen such that both satisfy condition (a) or both satisfy condition (b).
\end{thm}

The reasons for non-rigidity in the two cases are quite different in flavour. Condition (a) ensures that the global behaviour of the finite isotropy strata is determined by the GKM graph, which is not in general a consequence of (integer) equivariant formality. The reason is that equivariant cohomology, even with integer coefficients, can not detect certain interactions between $p$-torsion isotropy and $q$-torsion isotropy for different primes $p,q$. More precisely our counterexample relies on the fact that while $\mathbb{Z}_p$- and $\mathbb{Z}_q$-strata are in a sense encoded in the equivariant cohomology, the latter can not necessarily detect whether two such strata intersect in a $\mathbb{Z}_{pq}$-orbit. On the other hand, condition (b) ensures that the orbit space of all nonregular orbits is a surface with non-empty boundary. Without this condition non-homeomorphic examples can be produced via knotted closed surfaces in $S^4$.

In order to prove the realization part of the main theorem, we develop a combinatorial orientability property for GKM graphs (Definition \ref{defn:graphorientable}) which we believe to be of independent interest outside of the low dimensional setting. It manifests in forcing certain signs in the congruence relations among the labels of the graph. Using the built in orientability condition in the definition of a GKM manifold we prove (cf.\ Corollary \ref{cor:orientable})

\begin{prop}
The GKM graph of a GKM manifold is orientable.
\end{prop}

It is not hard to find examples of abstract GKM graphs for which the
orientability condition fails and which are therefore not realizable by GKM
manifolds. However they might of course come from
nonorientable  analogues of GKM
manifolds, see Example \ref{ex:nonorientable}. The reason for orientability not
showing up in our main theorem is that we prove it to be implied by the Poincaré
duality hypothesis, see Corollary \ref{cor:PDorientable}. Our proof is however
specifically for dimension $6$ and does not immediately generalize.

Section \ref{sec:prelims} begins with a recollection of the essential terminology. Furthermore the notion of orientability of a GKM graph is introduced and some results on graph cohomology are proved.
In Section \ref{sec:realization} we provide a construction to realize a given abstract GKM graph as the GKM graph of a GKM action. Our strategy to achieve this result is outlined in detail the beginning of said section. The proof that the constructed realization is in fact a GKM manifold is the purpose of Section \ref{sec:cohom}. The final Section \ref{sec:rigidity} contains the results on the rigidity problem.\\

Throughout the paper, cohomology is taken with integer coefficients, unless indicated otherwise. All GKM manifolds as well as all GKM graphs will be assumed to be effective (cf.\ Definition \ref{defn:effectivegraph}).\\

\noindent {\bf Acknowledgements.} We thank Anton Ayzenberg, Daniel Kasprowski and Catalin Zara for some valuable explanations, as well as Nikolas Wardenski for helpful literature references. This work is part of a project funded by the Deutsche Forschungsgemeinschaft (DFG, German Research Foundation) - 452427095. The research that led to this paper was partly carried out during a stay of the third author at the Max Planck Institute for Mathematics to which he is grateful for its hospitality and financial support.

\section{GKM graphs and equivariant cohomology}\label{sec:prelims}

\subsection{GKM actions and graphs}

The GKM correspondence, named after Goresky--Kottwitz--MacPherson \cite{GKM}, associates to certain actions of compact tori on compact, connected, orientable manifolds, so-called GKM actions, a combinatorial object, namely an (abstract) GKM graph \cite{GuilleminZaraII}. From this graph one can read off various topological information, such as the equivariant cohomology and equivariant characteristic classes of the action.

\begin{defn} \label{defn:gkmaction}
We consider a torus $T = T^k = S^1 \times \ldots \times S^1$ acting on a compact, connected, orientable $2n$-dimensional manifold $M$ with $H^{odd}(M)=0$. If 
\begin{enumerate}
\item the fixed point set $M^T = \{p\in M\mid Tp=\{p\}\}$ is finite and
\item the one-skeleton $M_1 = \{p\in M\mid \dim Tp\leq 1\}$ is a finite union of $T$-invariant $2$-spheres,
\end{enumerate}
then we say that the action is an (integer) GKM action. If only the rational odd cohomology vanishes, we speak of a (rational) GKM action. The manifold $M$ together with the action is called an (integer/rational) GKM $T$-manifold.
\end{defn}
Throughout the paper we will furthermore assume all GKM actions are effective and repress effectivity from the notation. In the above situation, the orbit space of the one-skeleton $M_1/T$ is homeomorphic to a graph $\Gamma$, whose vertices $V(\Gamma)$ are in one-to-one correspondence to the fixed point set $M^T$, and whose set of (unoriented) edges $E(\Gamma)$ contains precisely one edge connecting $p$ and $q$ for each $T$-invariant $2$-sphere containing $p$ and $q$. As a space $M_1$ is fully encoded in $\Gamma$. The $T$-action on $M_1$ can be encoded combinatorially by assigning labels to edges of $\Gamma$ via a function $\alpha\colon E(\Gamma)\rightarrow \hom (T^k,S^1)/\pm$ where $\hom(T^k,S^1)$ is understood as an additive Abelian group. The function $\alpha$ is defined as follows: the $T^k$-action on the $2$-sphere associated to an edge is equivariantly diffeomorphic to the pullback of the standard unit speed rotation action of $S^1$ on $S^2$ along a homomorphism $\varphi\colon T^k\rightarrow S^1$. Note however that $\varphi$ is only well defined up to sign as $\varphi$ and $-\varphi$ give rise to equivariantly homeomorphic actions on $S^2$. The labelled graph $(\Gamma,\alpha)$ is called the \emph{GKM graph} of the $T^k$-space $M$.

We will often identify $\hom(T^k,S^1)\cong\mathbb{Z}^k$ such that a tuple $(a_1,\ldots,a_k)\in \mathbb{Z}^k$ encodes the action $(t_1,\ldots,t_k)\cdot(z,h)=(t_1^{a_1}\ldots t_k^{a_k}z,h)$ on $S^2$ for $(t_1,\ldots,t_k)\in T^k$, $(z,h)\in S^2\subset \mathbb{C}\oplus \mathbb{R}$. Another interpretation of the labels is obtained via the observation that the $2$-spheres meeting at a fixed point $p$ correspond bijectively to the $2$-dimensional irreducible subrepresentations of $T_p M$. The associated elements in $\hom(T^k,S^1)$ are precisely the characters of these subrepresentations hence the information is equivalent to the weights of the subrepresentations in the dual Lie algebra $\mathfrak{t}^*$ of the torus (up to sign). We will therefore also refer to the labels as weights.

As the aim of this paper is a correspondence between geometry and algebra we are also interested in GKM graphs as abstract combinatorial objects. Let $\Gamma$ be an $n$-valent graph with finite sets of vertices $V(\Gamma)$ and edges $E(\Gamma)$. Multiple edges between vertices are allowed, but no loops. For the moment, edges are considered unoriented, but soon we will, for notational purposes, fix an auxiliary arbitrary orientation on each edge. Edges attached to a vertex $v\in V(\Gamma)$ will be collected in the $n$-element set $E(\Gamma)_v$.  The following definition goes back to \cite{GuilleminZaraI}.
\begin{defn}
A connection on $\Gamma$ is a collection $\nabla$ of bijections
$\nabla_{e,v,w}:E(\Gamma)_v\to E(\Gamma)_w$, for each edge $e$ connecting
vertices $v$ and $w$, such that
\begin{enumerate}
\item $\nabla_{e,v,w}(e) = e$ and
\item $\nabla_{e,w,v} = \nabla_{e,v,w}^{-1}$ for all vertices $v$ and $w$ and edges $e$ connecting $v$ and $w$.
\end{enumerate}
When $e$ is an oriented edge from $v$ to $w$ we also write $\nabla_e$ for $\nabla_{e,v,w}$.
\end{defn}
We picture a connection as a method to slide edges along edges. 
\begin{figure}[h]
	\centering
	\includegraphics{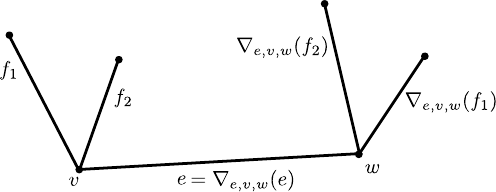}
	\caption{Picture of a connection.}
	\label{fig: connection} 
\end{figure}

\begin{defn}
Let $k\geq 1$. An (abstract) GKM graph $(\Gamma,\alpha)$ consists of a connected $n$-valent graph $\Gamma$ and a labelling of the edges $\alpha:E(\Gamma)\to {\mathbb{Z}}^k/{\pm 1}$, called axial function, such that the following conditions are satisfied:
\begin{enumerate}
\item For any $v\in V(\Gamma)$ and distinct $e,f\in E(\Gamma)_v$ the labels $\alpha(e)$ and $\alpha(f)$ are linearly independent.
\item There exists a compatible connection $\nabla$ on $\Gamma$, i.e., for any $v\in V(\Gamma)$ and $e,f\in E(\Gamma)_v$ there exists $c\in {\mathbb{Z}}$ such that
\begin{equation}\label{eq:gkmcond}
\alpha(\nabla_{e,v,w}(f)) = \pm  \alpha(f) + c\alpha(e).
\end{equation}
\end{enumerate}
\end{defn}
Although we usually fix a compatible connection when working with GKM graphs, they are in general not unique and not part of the canonical combinatorial data arising from geometry.
\begin{rem}\label{rem:signedgraph}
The original definition of abstract GKM graphs was given in \cite{GuilleminZaraII} and builds upon the above definition by an additional compatible choice of orientations and signs. We will refer to it as a \emph{signed} GKM graph, as it slightly differs from the above definition. This is due to the fact that it is tailored towards the geometric situation of almost complex manifolds, while the above definition is designed without additional geometric structures in mind. In order to separate the notion of an abstract GKM graph as defined above from the original definition we also like to refer to it as an \emph{unsigned} GKM graph. In this paper however we repress this from the notation as signed graphs will not play a role.
\end{rem}

The fact that a GKM graph $(\Gamma,\alpha)$ associated to a GKM manifold does admit a compatible connection and is thus indeed a GKM graph in the abstract sense is proven in \cite[Proposition 2.3]{GoertschesWiemelerNonNeg}, or \cite{GuilleminZaraII}.

\begin{defn} \label{defn:effectivegraph} We say that an abstract GKM graph is \emph{effective} if for one, and hence all vertices of the graph the common kernel of the labels of the adjacent edges, considered as homomorphisms $T\to S^1$, is trivial. 
\end{defn}
For a GKM graph associatied to a GKM $T$-action, this is equivalent to the effectivity of the action. Just as on the geometric side, all occurring GKM graphs will assumed to be effective.

\subsection{Cohomology}
Let $T$ be a torus and $T\rightarrow ET\rightarrow BT$ the universal $T$-bundle. For a $T$-space $X$ one defines the Borel construction as
\[X_T:=(ET\times X)/T\]
where the $T$-action is the diagonal one.
The (Borel) equivariant cohomology $H_T^*(X)$ is defined as the singular cohomology $H^*(X_T)$. We recall our convention that when not specified otherwise, the coefficient ring is the integers, but of course we may also consider equivariant cohomology with arbitrary coefficient rings. The Borel construction is natural with respect to maps of $T$-spaces. Hence such a map induces a map on equivariant cohomology. Usual cohomological techniques such as e.g.\ the Mayer-Vietoris sequence carry over directly to the equivariant setting.

An important special case is the single point $T$-space $*$ for which we obtain
$H_T^*(*)=H^*(ET/T)=H^*(BT)$. We will denote this graded ring by $R$. It is
isomorphic to the polynomial ring $\mathbb{Z}[x_1,\ldots,x_k]$ where the $x_i$
are of degree $2$ and $k$ is the rank of $T$. Analogously we define
$R_\mathbb{Q}$ to be $H^*(BT;\mathbb{Q})\cong \mathbb{Q}[x_1,\ldots,x_{
k}]$.
Considering the $T$-equivariant map $X\rightarrow *$ endows $H_T^*(X)$ with a
natural $R$-algebra structure.

Another feature which we will commonly use, in particular in Section \ref{sec:cohom}, is how equivariant cohomology relates to the cohomology of the orbit space. The projection $X_T\rightarrow X/T$ induces a map $H^*(X/T)\rightarrow H_T^*(X)$. It is an isomorphism if the action is free. If all stabilizers are finite, then one still gets an isomorphism $H^*(X/T;\mathbb{Q})\rightarrow H^*_T(X;\mathbb{Q})$. Note that this comparison homomorphism is natural with respect to equivariant maps.

Finally we recall the concept of equivariant formality as defined in \cite{GKM}. 
\begin{defn}
We say a $T$-space $X$ is (integer) equivariantly formal if $H_T^*(X)$ is free as an $R$-module. If $H_T^*(X;\mathbb{Q})$ is free over $R_\mathbb{Q}$ we say it is rationally equivariantly formal.
\end{defn}
This very restrictive condition is central to the theory of GKM manifolds and is what allows one to compute the nonequivariant cohomology from the equivariant one. The following observation sheds some light on the condition of vanishing odd cohomology in the definition of GKM manifolds. The rational version is a rather direct consequence of the Borel localization theorem and
\cite[Corollary 4.2.3]{AlldayPuppe}. The integral version was observed in \cite[Lemma 2.1]{MasudaPanov}.
\begin{prop}
Let $X$ be a $T$-manifold such that $X^T$ is finite. Then $X$ is equivariantly formal (resp.\ rationally equivariantly formal) if and only if $H^{odd}(X)$ (resp.\ $H^{odd}(X;\mathbb{Q})$) vanishes.
\end{prop}

The special appeal of the GKM conditions is that equivariant formality allows
the computation of topological invariants via equivariant cohomology and that
the latter is encoded combinatorially in the GKM graph via the notion of graph
cohomology. To define it we observe that there is yet another interpretation of
the labels of a GKM graph. A homomorphism $\varphi\in \hom(T,S^1)$ gives rise
to a map $BT\rightarrow BS^1$ of classifying spaces which in turn gives a map
$H^2(BS^1)\rightarrow H^2(BT)$. Associating to $\varphi$ the image of a fixed
generator of $H^2(BS^1)\cong \mathbb{Z}$ in $H^2(BT)$ gives an isomorphism
$\hom(T,S^1)\cong H^2(BT)$. Hence labels can be interpreted with values in
$H^2(BT)$, up to sign.

\begin{defn}
The equivariant cohomology $H^*_T(\Gamma,\alpha)$ of a GKM graph $(\Gamma,\alpha)$ is given by
\[
H^*_T(\Gamma,\alpha) = \left\{\left.(f(u))_u\in \bigoplus_{u\in V(\Gamma)} R \,\,\,\right|  \, f(v) - f(w) \equiv 0 \text{ mod }\alpha(e) \text{ for all edges }e \text{ from } v \text{ to } w\right\}.
\]
\end{defn}
We consider $H^*_T(\Gamma,\alpha)$ as an $R$-algebra. Replacing $R$ with $R_\QQ$ in the definition we obtain the rational equivariant cohomology $H^*_T(\Gamma,\alpha;{\mathbb{Q}})$.

\begin{defn}
The cohomology $H^*(\Gamma,\alpha)$ of a GKM graph $(\Gamma,\alpha)$ is defined as
\[
H^*_T(\Gamma,\alpha) := H^*_T(\Gamma,\alpha)/R^+\cdot H^*_T(\Gamma,\alpha),
\]
where $R^+$ consists of all elements in $R$ with positive
degree. The rational cohomology $H^*(\Gamma,\alpha;\QQ)$ is defined
analogously.
\end{defn}
\begin{rem}
Note that our notation is slightly different from the one used in \cite{GuilleminZaraII}: They denote the equivariant cohomology of the GKM graph $(\Gamma,\alpha)$ by $H^*(\Gamma,\alpha)$.
\end{rem}
The following proposition was shown in \cite[Theorem 7.2]{GKM}, as a consequence of the Chang-Skjelbred Lemma \cite[Lemma 2.3]{ChangSkjelbred}.

\begin{prop}\label{prop:gkmprop} For any integer (resp.\ rational) GKM $T$-manifold $M$, its
equivariant cohomology $H^*_T(M)$ (resp.\ $H^*_T(M;{\mathbb{Q}})$) is isomorphic
as an $R$-algebra to $H^*_T(\Gamma,\alpha)$ (resp.\
$H^*_T(\Gamma,\alpha;{\mathbb{Q}})$). \end{prop}

For each edge we now fix an arbitrary sign of the associated weight. In other words, we consider a lift of the axial function to a map $\alpha:E(\Gamma)\to {\mathbb{Z}}^k$.

\begin{rem}
This choice is not to be confused with the notion of a \emph{signed} GKM graph, where Condition \eqref{eq:gkmcond} reads differently, with $+$ instead of $\pm$; cf.\ Remark \ref{rem:signedgraph}.
\end{rem}

\begin{defn}[\cite{GuilleminZaraII}]
The Thom class $\Th_v$ of a vertex $v\in V(\Gamma)$ is the element $(f(u))_u\in
H^{2n}_T(\Gamma, \alpha)$ defined by
\[
f(u) = \begin{cases} \prod_{e\in E(\Gamma)_v} \alpha(e) & u = v\\ 0 & u \neq v.\end{cases}
\]
\end{defn}
If we had not chosen signs of weights, the Thom class would be well-defined only up to sign. 
\begin{lem}\label{lem:rankgraphcohom}
The rank of $H^*_T(\Gamma,\alpha)$ as an $R$-module is equal to the number of vertices of $\Gamma$.
\end{lem}
\begin{proof}
We show that the cokernel of the inclusion
\[
H^*_T(\Gamma,\alpha) \to \bigoplus_{v\in V(\Gamma)} R
\]
is a torsion $R$-module. To this end, let $f\in \bigoplus_{v\in V(\Gamma)} R$ be arbitrary. Then,
\[
\prod_{e\in E(\Gamma)} \alpha(e)\cdot f
\] 
is in the image of the inclusion, as it can be expressed as an $R$-linear combination of the Thom classes of the vertices of $\Gamma$.
\end{proof}
We can also define the Thom class of an oriented edge $e$ going from a vertex $v$ to a vertex $w$: let $\nabla$ be a compatible connection, and $e_i$ be the other edges at $v$. Having fixed signs for the weights we may write
\[
\alpha(\nabla_{e,v,w}(e_i)) = \varepsilon_i \alpha(e_i) + k_i \alpha(e),
\]
where $\varepsilon_i\in \{\pm 1\}$ and $k_i\in {\mathbb{Z}}$. We set $\varepsilon_1=1$, $k_1=0$ and note that the remaining numbers are uniquely determined. 
\begin{defn}
The Thom class $\Th_e\in H^{2n-2}_T(\Gamma,\alpha)$ of the oriented edge $e$ is the element $(f(u))_u$ defined by
\[
f(u) = \begin{cases} \prod_{f\in E(\Gamma)_v\setminus\{e\}} \alpha(f) & u=v \\
\prod_i \varepsilon_i \prod_{f\in E(\Gamma)_w\setminus\{e\}} \alpha(f) & u=w \\ 0 & u \neq v,w\end{cases}
\]
\end{defn}
A short combinatorial argument shows that $\Th_e$ does not depend on the choice of the connection $\nabla$.

\begin{lem}
The Thom class of an oriented edge $e$ defines an element of $H^{2n-2}_T(\Gamma,\alpha)$.
\end{lem}

\begin{proof}
The divisibility condition is trivial for any edge which is attached to at most one of the vertices $v$ and $w$. Let us check the condition for the edge $e$: write $\Th_e=(f(u))_u$; when multiplying out the expression
\[
f(w) - f(v) = \prod_i (\alpha(e_i) + \varepsilon_i k_i \alpha(e)) - \prod_i \alpha(e_i)
\]
we see that the two summands $\prod_i \alpha(e_i)$ cancel, and any of the remaining summands is divisible by $\alpha(e)$. If there is another edge $e_j$ that connects $v$ and $w$, then we see that both products in the expression above contain a factor $\alpha(e_j)$, hence $f(w)-f(v)$ is divisible by $\alpha(e_j)$.
\end{proof}

\begin{lem}\label{lem:thomverticescohomologous}
The Thom classes of any two vertices of $\Gamma$ define, up to sign, the same element in $H^*(\Gamma,\alpha)$.
\end{lem}
\begin{proof}
As $\Gamma$ is connected, it suffices to show that Thom classes of two vertices $v$ and $w$ joined by an edge $e$ define, up to sign, the same element. But this follows immediately, as the product of the Thom class of $e$ with $\alpha(e)$ is either the sum or the difference of the Thom classes of $v$ and $w$.
\end{proof}

For later reference we also recall the following definition.
\begin{defn}
Let $n\geq 0$. We say that a graded $\mathbb{Q}$-algebra $H$ satisfies ($n$-dimensional) Poincaré duality if $H^n\cong \mathbb{Q}$ and the paring $H^k\otimes H^{n-k}\rightarrow H^n\cong \mathbb{Q}$ is nondegenerate.
\end{defn}

\subsection{Orientable GKM graphs}

Let us fix an GKM graph $(\Gamma,\alpha)$, as well as a compatible connection $\nabla$. We choose signs for the weights, i.e.\ we consider a lift $\alpha\colon E(\Gamma)\rightarrow \mathbb{Z}^k$ of the axial function. For edges $e_1,\ldots,e_n$ meeting at a vertex, recall that transporting along say $e_1$ we obtain numbers $k_i\in \mathbb{Z}$, $\varepsilon_i\in \{\pm 1\}$ such that
\[ \alpha(\nabla_{e_1}(e_i))=\varepsilon_i \alpha(e_i)+k_i\alpha(e_1),\]
where we declare $\varepsilon_1=1$, $k_1=0$ and the remaining numbers are uniquely determined. We define $\eta(e_1)=-\varepsilon_2\cdot\ldots\cdot\varepsilon_n$.

\begin{defn}\label{defn:graphorientable}
A GKM graph is called \emph{orientable} if there is a choice of signs for the edge weights, such that for each closed edge path $e_1,\ldots,e_n$ we have
\[\prod_{i=1}^n \eta(e_i)=1.\]

\end{defn}

\begin{rem}
A geometric justification for this choice of terminology is given by Corollary \ref{cor:orientable} as well as Proposition \ref{prop:orientability}.
In \cite{oddGKM}, orientability of a GKM graph $(\Gamma,\alpha)$ was defined ad hoc via the nonvanishing of the top cohomology $H^*(\Gamma,\alpha)$. Furthermore there is a definition of orientability for so-called torus graphs, i.e.\ in the special case of an $n$-valent graph with labels in $\mathbb{Z}^n$, see \cite[Def.\ 7.9.16]{ToricTopology}. We are not aware of the precise relation between these combinatorial terms.
\end{rem}

\begin{rem}
Consider two vertices $v,w$ connected by an edge $e$ and with the other adjacent edges labelled $e_2,\ldots,e_n$ at $v$ and $e_2',\ldots,e_n'$ at $w$. Given a choice of sign for every edge, we observe that changing the sign of $\alpha(e)$ changes the sign of the $\eta(e_i)$ and $\eta(e_i')$, while leaving $\eta(-)$ invariant for all other edges. In particular the product over $\eta(-)$ remains unchanged along every closed edge path. We conclude that a GKM graph is orientable if and only if any choice of sign for the edge weights fulfils the condition in the definition of orientability.
\end{rem}

\begin{ex}\label{ex:nonorientable}
One can consider GKM theory for actions on nonorientable manifolds $M$, see
\cite{GoertschesMare}. In this setting, the one-skeleton $M_1$ consists of a
union of $T$-invariant two-spheres and real projective planes. If there are no
real projective planes occurring, then the orbit space of $M_1$, labelled as
usual with the isotropy weights, defines an abstract GKM graph $\Gamma$ such
that $H^*(\Gamma,\alpha;{\mathbb{Q}})$ does not satisfy Poincar\'e duality. As a
concrete example, consider a product $T^2$-action on $S^2\times S^2 \times S^2$
given by three pairwise linearly independent weights $\alpha,\beta$ and
$\gamma$. This action commutes with the ${\mathbb{Z}}_2$-action given by the
componentwise antipodal map, hence the $T^2$-action descends to the nonorientable manifold $M:=(S^2\times S^2\times S^2)/{\mathbb{Z}}_{2}$ with the following nonorientable GKM graph.
 \begin{center}
\begin{tikzpicture}

\draw[very thick] (-2,0) -- ++(3,0) -- ++(0,-3) -- ++(-3,0) -- ++(0,3) -- ++(3,-3);
\draw[very thick] (-2,-3) -- ++(3,3);

  \node at (-2,0)[circle,fill,inner sep=2pt]{};
  \node at (-2.3,-1.5){$\alpha$};

  \node at (1,0)[circle,fill,inner sep=2pt]{};
    \node at (1.3,-1.5){$\alpha$};
    \node at (-.5,.3){$\beta$};
    \node at (-.5,-3.3){$\beta$};
        \node at (-.5,-1.1){$\gamma$};
        
    \node at (-.5,-1.95){$\gamma$};

  \node at (1,-3)[circle,fill,inner sep=2pt]{};

  \node at (-2,-3)[circle,fill,inner sep=2pt]{};

\end{tikzpicture}
\end{center}
The connection can be defined with $\varepsilon_i=1$, $k_i=0$ for all edges. Hence $\eta(e)=-1$ for every edge $e$ and consequently any path of odd length violates the orientability condition.
\end{ex}

\begin{prop}\label{prop:thomclasslem}
If $(\Gamma,\alpha)$ is not orientable, then for every vertex $v\in V(\Gamma)$ the Thom class $\Th(v)$ vanishes in $H^{2n}(\Gamma,\alpha;\mathbb{Q})$ .
\end{prop}
\begin{proof}
We fix some choice of sign for the weights. As $(\Gamma,\alpha)$ is not orientable, there exists a closed edge path $e_1,\ldots,e_n$ starting at a vertex $v_1$, for which $\prod_{i=1}^n \eta(e_i) = -1$.
By definition of the Thom classes, we obtain $\Th_{v_1} - \alpha(e_1)\Th_{e_1} = \eta(e_1)\Th_{v_2}$.  Following the closed path, we conclude that $\Th_{v_1} = \prod_{i=1}^n \eta(e_i) \cdot \Th_{v_1}= -\Th_{v_1}$ in $H^{2n}(\Gamma,\alpha;{\mathbb{Q}})$.
\end{proof}

\begin{cor}\label{cor:orientable}
If $M$ is a GKM manifold then its GKM graph is orientable.
\end{cor}
\begin{proof}
Recall e.g.\ from \cite[Chapter 10]{supersymmetry} that for a $T$-invariant  submanifold there is an equivariant Thom class in $H_T^*(M)$ which restricts to the nonequivariant Thom class in $H^*(M)$, see Section 10.4 in the above reference. In particular, specializing to the case of a fixed point $p\in M^T$, we obtain a class $x\in H_T^*(M;\mathbb{Q})$ which is localized around $p$ in the sense that $x$ restricts to $0$ at all other fixed points and which restricts to a fundamental class in $H^*(M)$ (note that from the reference we get this for real coefficients but then such a class automatically also exists over $\mathbb{Q}$). We denote by $(\Gamma,\alpha)$ the GKM graph of $M$. Under the isomorphism $H_T^*(\Gamma,\alpha)\cong H_T^*(M)$ the class corresponding to $x$ is located at a single vertex and hence a multiple of the (combinatorial) Thom class at this vertex. If $(\Gamma,\alpha)$ were not orientable, then the Proposition \ref{prop:thomclasslem} would yield $x=0$, which is a contradiction.
\end{proof}

\subsection{Results on $3$-valent GKM graphs}

Starting in this section, we restrict to $n=3$ and $k=2$, i.e., $3$-valent GKM graphs with labels in ${\mathbb{Z}}^2$. As at the end of the last section, we fix (arbitrary) signs of the labels.

\begin{prop}\label{prop:eqcohomfreeoverQ}
The equivariant cohomology $H^*_T(\Gamma,\alpha;{\mathbb{Q}})$ is a free $R_\QQ$-module.
\end{prop}
\begin{proof}
We fix, for every edge $e$, an arbitrary orientation, i.e., we have chosen its initial and terminal vertex $i(e)$ and $t(e)$. This determines an exact sequence
\[
0\to H^*_T(\Gamma,\alpha;{\mathbb{Q}})\to \bigoplus_{u\in V(\Gamma)} R_\QQ \to \bigoplus_{e\in E(\Gamma)} R_\mathbb{Q}/(\alpha(e)),
\]
where the last map is $\varphi:(f(v))_{v\in V(\Gamma)}\mapsto (f(i(e))-f(t(e))+R_\mathbb{Q}\cdot \alpha(e))_{e\in E(\Gamma)}$. 
Then, the image of $\varphi$ has depth $\geq 1$ since multiplication with a generic element of $\mathbb{R}_\mathbb{Q}^2$ is injective -- see e.g.\ \cite{GoertschesToeben} for the notion of depth and Cohen-Macaulay modules in equivariant cohomology. Thus it follows from the short exact sequence
\[
0\to H^*_T(\Gamma,\alpha;{\mathbb{Q}})\to \bigoplus_{u\in V(\Gamma)} R_\QQ \to \im\, \varphi\to 0
\]
that $H^*_T(\Gamma,\alpha;\QQ)$ is of depth $\geq 2$. As $R_\mathbb{Q}$ has Krull dimension $2$ this implies freeness.
\end{proof}

\begin{prop}\label{prop:h4thom}
$H^4_T(\Gamma,\alpha;{\mathbb{Q}})$ is additively spanned by the Thom classes of the edges of $\Gamma$ and $R_\QQ^4\cdot 1$.
\end{prop}
\begin{proof}
Let $(f(v))_v\in H^4_T(\Gamma,\alpha;\QQ)$. By adding a suitable multiple of $1\in H^4_T(\Gamma,\alpha;\QQ)$ we may assume that $f(v)=0$ for some vertex $v$.
We show that by adding appropriate multiples of other Thom classes of edges, we can increase the number of vertices at which $f$ vanishes. Inductively, this will  imply the claim. To this end, consider any edge $e_1$ connecting a vertex $v$ with $f(v)=0$ with a vertex $w$ with $f(w)\neq 0$. Then $f(w) = \alpha(e_1)g$ for a linear form $g$. Let $e_2$ and $e_3$ be the other two edges at $w$. As $\alpha(e_2)$ and $\alpha(e_3)$ are a ${\mathbb{Q}}$-basis of $R_\mathbb{Q}^2\cong\mathbb{Q}^2$ we can write $g=a_2\alpha(e_2) + a_3\alpha(e_3)$ for some $a_2,a_3\in {\mathbb{Q}}$. If one of the coefficients, say $a_3$, vanishes, then $\alpha(e_3)$ does not divide $f(w)$, hence $e_3$ connects $w$ to a vertex $u$ with $f(u)\neq 0$, and we can subtract from $f$ a multiple of  $\Th_{e_3}$ (with respect to some orientation of $e_3$) to obtain a class that vanishes at more vertices. If none of the coefficients $a_2$ and $a_3$ vanishes, then neither of $\alpha(e_2)$ and $\alpha(e_3)$ divide $f(w)$, hence both $e_2$ and $e_3$ connect $w$ to vertices at which $f$ does not vanish. Then, we can subtract the corresponding linear combination of  $\Th_{e_2}$ and $\Th_{e_3}$ to arrive at a class that vanishes at more vertices.
\end{proof}

\begin{prop}\label{prop:thomverticesnontriv} If $H^*(\Gamma,\alpha;{\mathbb{Q}})$ satisfies Poincaré duality, then the Thom classes of the vertices define nontrivial elements in $H^{6}(\Gamma,\alpha;{\mathbb{Q}})$.
\end{prop}
\begin{proof}
If $H^4(\Gamma,\alpha;{\mathbb{Q}})= 0$, then by Poincaré duality also $H^2(\Gamma,\alpha;{\mathbb{Q}})=0$. In this situation, if a Thom class of a vertex was trivial in $H^6(\Gamma,\alpha;{\mathbb{Q}})$, it would already be in $R_\QQ^6\cdot 1$. But this is impossible as $\Gamma$ has more than one vertex, and the Thom classes of vertices are localized at a single vertex.

If $H^4(\Gamma,\alpha;{\mathbb{Q}})\neq 0$, then by Proposition \ref{prop:h4thom} we find an edge $e$ whose Thom class defines a nontrivial element in $H^4(\Gamma,\alpha;{\mathbb{Q}})$. By Poincaré duality, we can  multiply to it an element in $H^2(\Gamma,\alpha;{\mathbb{Q}})$ to obtain a nonzero element $\omega\in H^6(\Gamma,\alpha;{\mathbb{Q}})$ (i.e., the product in $H^6_T(\Gamma,\alpha;{\mathbb{Q}})$ is not divisible by an element in $R_\QQ^2$). By definition of $\Th_e$, the class $\omega$ is represented by an element in $H^6_T(\Gamma,\alpha;{\mathbb{Q}})$ that is localized at at most two vertices. By adding an appropriate $R_\QQ^2$-multiple of $\Th_e$, this class is localized at a single vertex, hence a multiple of the Thom class of that vertex.
\end{proof}

Propositions \ref{prop:thomclasslem} and \ref{prop:thomverticesnontriv} combine to yield

\begin{cor}\label{cor:PDorientable}
If $H^*(\Gamma,\alpha;{\mathbb{Q}})$ satisfies Poincaré duality, then $\Gamma$ is orientable.
\end{cor}

\section{Realization}
\label{sec:realization}

\subsection{The construction}

Given a $3$-valent abstract GKM graph $(\Gamma,\alpha)$ we want to build a $6$-dimensional closed GKM manifold realizing it. From a certain point in the construction $(\Gamma,\alpha)$ will be assumed to be orientable as defined in Definition \ref{defn:graphorientable}. The construction proceeds by gluing $T^2$-equivariant handles, i.e.\ $6$-dimensional $T^2$-spaces whose orbit spaces are given by the standard nonequivariant notion of a $4$-dimensional handle. We give an overview of the construction procedure:
\begin{enumerate}
\item We start by using $0$-handles and $1$-handles -- corresponding bijectively to the vertices and edges of the graph -- to construct a $6$-manifold $M$ with boundary which contains the $1$-skeleton of the action and deformation retracts onto the graph of $2$-spheres encoded by $(\Gamma,\alpha)$. The boundary $\partial M$ is almost free and potentially contains orbits with nontrivial discrete isotropy. We note that the fact that the abstract one-skeleton defined by an abstract signed GKM graph can be thickened to an open manifold has already been observed in \cite[Section 3.1]{GuilleminZaraII}.
\item 
	
	We now glue $2$-handles to the thickened graph
	$M$, reducing the fundamental group in
	the process. As $M/T$ deformation retracts onto the graph $\Gamma$ this
	process corresponds (on the orbit space level and up to homotopy) to gluing in
	$2$-discs into $\Gamma$ along edge paths. The connection defines canonical
	candidates for these closed edge paths in form of the so-called connection
	paths. Gluing in $2$-handles along these connection paths (which can be
	thought of as the paths of potentially nonregular orbits remaining in the
	boundary) we arrive at a manifold with boundary $N$ in which all nonregular
	isotropy lies on the inside and whose orbit space deformation retracts onto
	the closed surface $F$ which arises from $\Gamma$ by gluing in $2$-disks
	along all connection paths.
\item The surface $F$ depends on the connection and can in general be an arbitrary, even nonorientable, closed surface. In order for our $T^2$-manifold to satisfy the topological conditions of $H^{odd}=0$ and of simply-connectedness we want $F$ to be simply-connected, i.e.\ $S^2$. This is achieved by gluing additional $2$-handles to $N$ which, as opposed to all previous building blocks, do not directly come from combinatorial data provided by the graph. However, at this point we are working in the regular stratum of the manifold which reduces the problem mostly to non-equivariant surgery and hence does not rely on the compatibility of isotropies which were previously ensured by the combinatorics of the graph. We arrive at a simply-connected $T^2$ manifold $N_2$ with $N_2/T^2=S^2\times D^2$ and $\partial N_2=S^2\times S^1\times T^2$. (for completeness we note that $N_2$ does not directly arise out of the original $N$ but potentially one of the connection handles from the previous step has to be removed and reglued with a gluing map modified by an equivariant Dehn twist).
\item We now obtain a closed manifold by equivariantly gluing $D^3\times
	S^1\times T^2$ to $\partial N_2$. More generally one can cap off $N_2$ by
	embedding $N_2/T$ into $S^4$ and  gluing appropriately with $Z\times T^2$,
	where $Z$ is the complement of the embedding. The final orbit space will be
	$S^4$, as predicted by results of \cite{AyzenbergMasuda} for connected
	isotropies. 
\item In Section \ref{sec:cohom} we will prove that that for any choice of embedding of $S^2 \times D^2$ in $S^4$ as in (iv) the resulting simply-connected manifold $X$ is GKM, provided the abstract GKM graph $(\Gamma,\alpha)$ satisfies the necessary conditions. That is, if $H^*(\Gamma,\alpha;{\mathbb{Q}})$ satisfies $6$-dimensional Poincar\'e duality (which implies orientability by Corollary \ref{cor:PDorientable}), then $X$ is a rational GKM manifold, see Theorem \ref{thm:rationalGKM}. If additionally $H^*_T(\Gamma,\alpha)$ is a free $R$-module, then $X$ is an integer GKM manifold, see Theorem \ref{thm:mainintegerGKM}. Note that in the rational case, by Proposition \ref{prop:eqcohomfreeoverQ}, we do not need any assumption on freeness of the equivariant cohomology of the graph.
\end{enumerate}
We briefly comment on some notation and technicalities in the construction: in order to write down explicit building blocks we will at several points fix certain data like orders of e.g.\ edges or weights which is not part of the combinatorial data of $(\Gamma,\alpha)$. These are highly non-canonical and merely a way to fix the language. In order to compare them we will use appropriate permutations. The notation is such that, e.g. when passing from the building block associated to an edge $e$ to the one associated to a vertex $v$ one uses the permutation $\sigma_{e,v}$. Similarly when passing from a vertex $v$ to an adjacent vertex $w$ by using an edge $e$ between them, the corresponding permutation will be given the name $\sigma_{v,e,w}$. The same logic applies throughout the text to all kinds of transition or gluing maps.

We want to point the reader towards to the fact that, in particular during steps (i) and (ii) of the above procedure, there are three layers to the construction: the first two are the $6$-dimensional $T^2$-spaces and their $4$-dimensional orbit spaces. However there is a third $3$-dimensional layer underlying the two, which we will track throughout the construction. It comes from the fact, that locally our $T^2$-action will be described as the restriction of a surrounding $T^3$-action. The aforementioned $3$-dimensional layer is that of the corresponding $T^3$-orbit spaces which are naturally smooth manifolds with corners. Although the $T^3$-action does glue globally, the orbit spaces do, providing a very helpful description of the $4$-dimensional $T^2$-orbit space over the $3$-dimensional layer.

\begin{rem}
We would like to relate our construction to another realization procedure in \cite{Ayzenberg}, see Construction 3.5 therein. In this paper Ayzenberg produces, starting from a topological manifold $Q$ as well as certain abstract compatible characteristic data, a complexity one topological space with orbit space $Q$. In general, his construction does not produce a manifold and is limited to connected stabilizers. However, having realized an abstract $3$-valent GKM graph in which any pair of adjacent weights is a basis with our method by a GKM manifold with connected stabilizers, a posteriori this realization is equivariantly homeomorphic to Ayzenberg's construction, see \cite[Proposition 3.7]{Ayzenberg}.
\end{rem}

\subsection{0-handles and 1-handles}

We set $T=T^2$ and fix a $3$-valent GKM graph $(\Gamma,\alpha)$ with labels in $\mathbb{Z}^2/{\pm 1}$ as well as a compatible connection on it.

For each vertex we fix an order of the three outgoing edges. For each edge we choose an orientation and a sign of the associated weight. Note that the choice of sign and orientation are independent and arbitrary. Now for a vertex $v$ let $\alpha_1,\alpha_2,\alpha_3$ be the weights of the adjacent edges in the chosen order and with the chosen signs. Set $B_v=D^6_{(\alpha_1,\alpha_2,\alpha_3)}$ where the latter is $D^6\subset \mathbb{C}^3$ together with the $T$-action defined by the set of weights $(\alpha_1,\alpha_2,\alpha_3)$.

For an edge $e$ set $K_e=I\times S^1_{\alpha_1}\times D^4_{(\alpha_2,\alpha_3)}$, where $I=[0,1]$ is the unit interval, $\alpha_1,\alpha_2,\alpha_3$ are the weights with the chosen signs at the initial vertex of $e$ and the order is chosen arbitrarily with the condition that $\alpha_1$ is the weight associated to $e$. Now each of the building blocks $B_v$, $K_e$ comes with the following additional structure:

\begin{itemize}
\item There is a canonical $T^3$ action on $B_v$ (resp. $K_e$). The weights define a unique homomorphism $T\rightarrow T^3$ such that the $T$-action is the pullback of the $T^3$-action. The image of this homomorphism is denoted $T_v\subset T^3$ (resp. $T_e$).
\item The orbit space $\overline{B_v}:= B_v/T^3$ can be canonically identified with $\mathbb{R}_{\geq 0}^3\cap D^3$ such that the map $s_v\colon \mathbb{R}_{\geq 0}^3\cap D^3\rightarrow B_v$ induced by the obvious inclusion $\mathbb{R}^3\rightarrow \mathbb{C}^3$ is a section of the orbit map.
Analogously $\overline{K_e}:=K_e/T^3\cong I\times (\mathbb{R}_{\geq 0}^2\cap D^2)$ and $s_e\colon I\times (\mathbb{R}_{\geq 0}^2\cap D^2)\rightarrow K_e$, $(t,x,y)\mapsto (t,1,x,y)$ is a section of the orbit map.
\end{itemize}

Analogous to \cite[Lemma 2.11]{Ayzenberg} we have the following description of $T$-orbit spaces. Here, $\mu$ denotes the multiplication maps of the $T^3$-actions on $B_v$, $K_e$, and $F_v=(\partial \mathbb{R}_{\geq 0}^3)\cap D^3$, $F_e= I\times ((\partial \mathbb{R}_{\geq 0}^2)\cap D^2)$.

\begin{lem} \label{lem:orbitspaceisS1bundle}
\begin{enumerate}
\item The map $\mu\circ (s_v\times\Id_{T^3})$ induces a homeomorphism
\[ \left(\overline{B_v}\times T^3/T_v\right)/_\sim\cong B_v/T_v\]
where $\sim$ collapses the $T^3/T_v$ factor over $F_v$.
\item The map $\mu\circ (s_e\times\Id_{T^3})$ induces a homeomorphism
\[ \left(\overline{K_e}\times T^3/T_e\right)/_\sim\cong K_e/T_e\]
where $\sim$ collapses the $T^3/T_e$ factor over $F_e$.
\end{enumerate}
\end{lem}

\begin{proof}
We prove $(i)$. The proof of $(ii)$ is identical. The map $\overline{B_v}\times T^3\rightarrow B_v$ is $T^3$-equivariant. Thus we obtain a $T^3/T_v$-equivariant map $\overline{B_v}\times T^3/T_v\rightarrow B_v/T_v$ which is a homeomorphism on $T^3/T_v$-orbit spaces. Hence it becomes a homeomorphism if we divide every orbit $\{p\}\times T^3/T_v$ by the $T^3/T_v$-stabilizer of its image orbit $T_v\cdot s_v(p)\subset B_v/T_v$. Said stabilizer is just the image of the stabilizer $T^3_{s_v(p)}$ in $T^3/T_v$. Note that since the $T^3$-action has only connected stabilizers, this image is a connected subgroup of the circle $T^3/T_v$. So the only question is whether it is the trivial subgroup or all of $T/T_v$, i.e.\ whether $T^3_{s_v(p)}\subset T_v$ or not. Using that the weights at $v$ are pairwise linearly independent (i.e.\ $T_v$ is in general position) one quickly verifies that this is the case if and only if $p\notin F_v$.
\end{proof}

Consider an edge $e$ going from $v$ to $w$. Denote by $\alpha_1,\alpha_2,\alpha_3$ the weights (with chosen signs) at $v$, corresponding to the edges $e_1=e, e_2, e_3$ in the order in which they appear in $K_e$. Let $\beta_1,\beta_2,\beta_3$ denote the weights at $w$ in the fixed order chosen at $w$. Then the connection $\nabla_e$ gives a bijective correspondence $\{\alpha_i\}\leftrightarrow \{\beta_i\}$. Let $\sigma_{e,w}$ be the permutation such that $\alpha(\nabla_e (e_i))=\beta_{\sigma_{e,w}(i)}$. Using the congruence relations of the connection, let $\epsilon_i\in\{\pm1\}$ and $k_i\in \mathbb{Z}$ such that $\alpha(\nabla_{e_1}(e_i))=\varepsilon_i\alpha_i+k_i\alpha_1$. These numbers are unique except for $i=1$, where we choose $\varepsilon_1=1$ and $k_1=0$.
Let $P_{{e,w}}$ be the permutation matrix which swaps the entries of column vectors according to $\sigma_{e,w}$, i.e.\ $P_{e,w}\cdot b_i=b_{\sigma_{e,w}(i)}$ for the standard basis $b_i$. Then the automorphism $\varphi_{e,w}\colon T^3\rightarrow T^3$ defined by the matrix
\[P_{e,w}\cdot\begin{pmatrix}
1 & 0 & 0\\ k_2 & \varepsilon_2 & 0\\ k_3& 0& \varepsilon_3
\end{pmatrix}\]
has the property that the diagram
\begin{equation}\label{eq1}
\xymatrix{
& T\ar[dl]_{(\alpha_1,\alpha_2,\alpha_3)}\ar[dr]^{(\beta_1,\beta_2,\beta_3)} &\\
T^3\ar[rr]^{\varphi_{e,w}} & &T^3
}
\end{equation}
commutes. In particular $\varphi_{e,w}$ sends $T_e$ to $T_w$. Define $\sigma_{e,v}$, $\varphi_{e,v}$ analogously, where the permutation transforms the order of the weights at $e$ into the one at $v$ and for $\varphi_{e,v}$ we just need the permutation part, i.e. $k_i=0$, $\varepsilon_i=1$.

Now fix some $\delta>0$ and denote by $c_1$ and $c_{-1}$ the identity and conjugation map on $\mathbb{C}$. We glue ${K_e}$ to ${B_w}$ along the map
\[\psi_{e,w}\colon \partial_1 K_e=\{1\}\times S^1\times D^4\rightarrow S^5\subset B_w\]
which is the composition of \[(v,w,z)\mapsto (\sqrt{1-\delta^2\|(w,z)\|^2}\cdot v,\delta \cdot c_{\varepsilon_2}(w)\cdot v^{k_2},\delta\cdot c_{\varepsilon_3}(z)\cdot v^{k_3})\]
with the coordinate switch map on $\mathbb{C}^3$ defined by multiplication with $P_{e,w}$. We collect some important observations in the following

\begin{lem}\label{lem:orbitstuff}
\begin{enumerate}
\item For $p\in \partial_1K_e$, $g\in T^3$ we have
\[\psi_{e,w}(g\cdot p)=\varphi_{e,w}(g)\cdot\psi_{e,w}(p).\]
\item $\psi_{e,w}$ is $T$-equivariant.
\item $\psi_{e,w}$ induces a map on $T^3$-orbit spaces which is given by the composition of
\begin{align*}
\overline{\psi}_{e,w}\colon \partial_1\overline{K_e}= \{1\}\times (\mathbb{R}^2_{\geq 0}\cap D^2)
&\longrightarrow \mathbb{R}^3_{\geq 0}\cap D^3= \overline{B_w}\\
(x,y)&\longmapsto (\sqrt{1-\delta^2\|(x,y)\|^2},\delta x,\delta y)
\end{align*}
with the coordinate switch map which is multiplication with $P_{e,w}$.
\item The isomorphisms from Lemma \ref{lem:orbitspaceisS1bundle} are compatible with these maps in the sense that the diagram
\[\xymatrix{
(\overline{K_e}\times T^3/T_e)/_\sim\ar[r] & K_e/T_e\\
(\partial_1\overline{K_e}\times T^3/T_e)/_\sim\ar[d]_{\overline{\psi}_{e,w}\times\overline{\varphi}_{e,w}} \ar[u]\ar[r]&\partial_1K_e/T_e\ar[d]\ar[u]\\
(\overline{B_w}\times T^3/T_w)/_\sim \ar[r]& B_w/T_w
}\]
commutes, where $\overline{\varphi}_{e,w}\colon T^3/T_e\rightarrow T^3/T_w$ is induced by $\varphi_{e,w}$.
\end{enumerate}
\end{lem}

\begin{proof}
Part $(i)$ is a straight forward verification. Part $(ii)$ follows from $(i)$ and the commutative triangle \eqref{eq1}. Regarding part $(iii)$, the existence of $\overline{\psi}_{e,w}$ is a consequence of $(i)$. Its explicit formula is a straight forward verification. In the diagram of $(iv)$ only the commutativity of the lower square needs verification. It is induced as a quotient of the outer square of the diagram
\[\xymatrix{
\partial_1\overline{K_e}\times T^3\ar[r]^{s_e\times\Id}\ar[d]_{\overline{\psi}_{e,w}\times\varphi_{e,w}}& \partial_1 K_e\times T^3\ar[r]^\mu \ar[d]_{\psi_{e,w}\times \varphi_{e,w}} & \partial_1 K_e\ar[d]_{\psi_{e,w}}\\
\overline{B_w}\times T^3\ar[r]^{s_w\times\Id} & B_w\times T^3\ar[r]^\mu & B_w
}\]
in which commutativity of the left hand square is easily checked and commutativity of the right hand square is $(i)$.
\end{proof}

The analogous definitions can be used to define the gluing map $\psi_{e,v}\colon \partial_0 K_e=\{0\}\times S^1\times D^4\rightarrow B_v$ from the starting point of an edge to its initial vertex. The definitions and properties extend in the obvious way. Now we define

\[M:=\left(\bigsqcup_{v\in V(\Gamma)}B_v\sqcup \bigsqcup_{e\in E(\Gamma)} K_e\right)/_\sim\]
with the equivalence relation induced by all the gluing maps $\psi(e,v)$ for $e\in E(\Gamma)$, $v\in V(\Gamma)$.

\begin{figure}[H]
	\centering
	\includegraphics{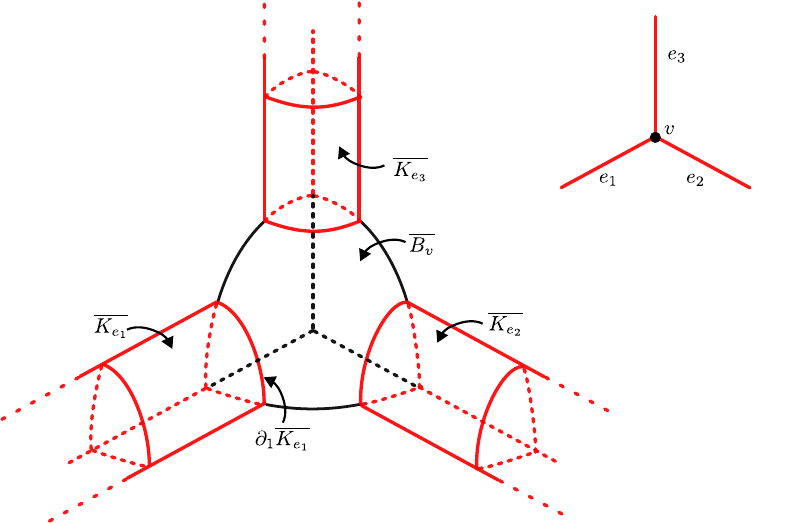}
	\caption{A local picture of the $T^3$ orbit of $M$ around a vertex $v$.}
	\label{fig: connection path}  
\end{figure}

\begin{rem}\label{rem:bookkeeping} We set $\sigma_{v,e,w}=\sigma_{e,w}\circ\sigma_{e,v}^{-1}$ and $ \varphi_{v,e,w}=\varphi_{e,w}\circ\varphi_{e,v}^{-1}$ and sum up the essentials for bookkeeping when passing from one vertex to an adjacent one:
Let $v$ be a vertex with edges $e_1,e_2,e_3$ in the chosen order with weights $\alpha_i$. Let $w$ be an adjacent vertex connected by $e\in \{e_1,e_2,e_3\}$. Denote by $e_1',e_2',e_3'$ the outgoing edges at $w$ in the chosen order with associated weights $\beta_i$. Then $\nabla_e\colon \{e_1,e_2,e_3\}\rightarrow \{e_1',e_2',e_3'\}$ defines a bijection. One checks that, independently of the enumeration of weights chosen in $K_e$, the permutation $\sigma_{v,e,w}$ is defined by the property that \[\nabla_e(e_i)=e_{\sigma_{v,e,w}(i)}'.\]
The homomorphism $\varphi_{v,e,w}$ is has the property that
\[\xymatrix{
& T\ar[dl]_{(\alpha_1,\alpha_2,\alpha_3)}\ar[dr]^{(\beta_1,\beta_2,\beta_3)} &\\
T^3\ar[rr]^{\varphi_{v,e,w}} & &T^3
}\]
commutes. To determine the matrix representation $(a_{ij})$ of $\varphi_{v,e,w}$ one proceeds as follows: we have $\alpha(\nabla_e(e_i))= \varepsilon_i \alpha(e_i)+k_i\alpha(e)$, for some $\varepsilon_i=\pm 1$, $k_i\in \mathbb{Z}$, setting $\varepsilon_m=1$, $k_m=0$ for the $m\in\{1,2,3\}$ with $e=e_m$. Now $a_{ij}=\varepsilon_j$ for $i=\sigma_{v,e,w}(j)$, $a_{ij}=0$ for $j\neq m$, $i\neq \sigma_{v,e,w}(j)$, and $a_{im}=k_{\sigma_{v,e,w}^{-1}(i)}$ for $i\neq \sigma_{v,e,w}(m)$.

We note for later use that if we change the order of the edges at $v,w$ via permutations $\sigma,\tau$ to be $e_{\sigma(1)},e_{\sigma(2)},e_{\sigma(3)}$ and $e'_{\tau(1)},e'_{\tau(2)},e'_{\tau(3)}$, then $\varphi_{v,e,w}$ changes to $P_\tau\varphi_{v,e,w} P_\sigma^{-1}$, where $P_\sigma,P_\tau$ are the permutation automorphisms commuting the coordinates according to $\sigma,\tau$.
\end{rem}

\subsection{Orientability}\label{sec:orientability}

Recall that an orientation on a manifold with boundary induces a unique orientation on its boundary by defining a local frame to be positive if an inward pointing tangent vector followed by said frame is positive in the whole manifold.
Now we fix orientations on every $B_v$ via the orientation on $\mathbb{C}$ defined by the standard complex structure. We also orient the $K_e$ by orienting $S^1$ counterclockwise and $I$ from $0$ to $1$. In view of orienting the global object $M$ we remark that the chosen orientations on $B_v$, $K_e$ extend to the glued object $B_v\cup K_e$ (assuming $v$ is adjacent to $e$) if and only if the the gluing map $\psi_{e,v}$ is orientation reversing w.r.t.\ the orientations on the boundary.

\begin{lem}\label{lem:orientation}
Let $e$ be an edge which is oriented from $v$ to $w$.
\begin{enumerate}
\item The gluing map $\psi_{e,v}$ is always orientation reversing.
\item The gluing map $\psi_{e,w}$ is orientation reversing if and only if $\eta(e)=1$.
\end{enumerate}
\end{lem}

\begin{proof}
The map on $B_v$ which switches complex coordinates is orientation preserving. Hence the question whether $\psi_{e,v}$ is orientation preserving depends not on the order of the edges which we chose in every vertex and edge. Thus for part $(i)$ we only need to show that the map
\begin{align*}
S^1\times D^4&\longrightarrow S^5\subset D^6\\
(s,v,w)&\mapsto \left( \sqrt{1-\delta^2\| (v,w)\|} s,\delta v,\delta w\right)
\end{align*}
reverses the chosen orientations. We recall that $S^1\times D^4$ is oriented as $\{0\}\times S^1\times D^4$, hence $(i)$ holds. By the same reasoning, for $(ii)$ we need to investigate whether the map

\begin{align*}
S^1\times D^4&\longrightarrow S^5\subset D^6\\
(s,v,w)&\mapsto \left( \sqrt{1-\delta^2\| (v,w)\|} s,\delta c_{\varepsilon_2}(v)s^{k_2},\delta c_{\varepsilon_3}(w)s^{k_3}\right)
\end{align*}
preserves orientation where $c_1$ is the identity and $c_{-1}$ is complex conjugation, $\varepsilon_i\in \{\pm 1\}$, $k_i\in \mathbb{Z}$. Note that this time the orientation on $S^1\times D^4$ is the one coming from $\{1\}\times S^1\times D^4$ and thus different from the one in the proof of part $(i)$. We claim that the question whether the map preserves orientations does not depend on the $k_i$. Indeed we consider the covering $\mathbb{R}\times D^4\rightarrow S^1\times D^4$ and compose it with the above map. The isotopy class (through local diffeomorphisms) of this composition does not depend on the value of the $k_i$ as these can be now varied continuously. Consequently they all behave identically with respect to orientations. We may thus assume $k_1=0=k_2$.
Since conjugation reverses orientation the whole map reverses orientation if and only if $\varepsilon_1\varepsilon_2=-1$. This is equivalent to $\eta(e)=1$.
\end{proof}

\begin{prop}\label{prop:orientability}
The manifold $M$ is orientable if and only if $(\Gamma,\alpha)$ is orientable.
\end{prop}

\begin{proof}
Given an orientation on $(\Gamma,\alpha)$, we fix an orientation on a single $B_v$ and spread its orientation through the graph by giving each piece $K_e$, $B_w$ the orientation such that its gluing map is orientation reversing. The orientability condition together with Lemma \ref{lem:orientation} ensures that this gives a well defined orientation on $M$. Conversely, similar arguments show that $M$ is nonorientable in case there is a closed edge path $e_1,\ldots,e_n$ such that $\prod_{i=1}^n\eta(e_i)=-1.$
\end{proof}

Recall that for any edge $e$ adjacent to vertices $v,w$ we also have the associated permutation $\sigma_{v,e,w}\in S_3$ and the homomorphism $\varphi_{v,e,w}\colon T^3\rightarrow T^3$. We identify the latter with its corresponding matrix in $\mathrm{GL}(3,\mathbb{Z})$.

\begin{lem}\label{lem:signformula}
One has $\eta(e)=-\mathrm{sign}(\sigma_{v,e,w})\cdot\det(\varphi_{v,e,w})$.
\end{lem}

\begin{proof}
By construction we have
\[
\det(\varphi_{v,e,w})=\det P_{e,v}^{-1}\cdot \det P_{e,w}\cdot \det
\begin{pmatrix}
1 &0 &0\\ k_2& \varepsilon_2 & 0\\ k_3& 0 &\varepsilon_3
\end{pmatrix}
=\mathrm{sign}(\sigma_{v,e,w})\cdot \varepsilon_1\varepsilon_2
\]
which yields the desired formula.
\end{proof}

\subsection{2-handles}\label{sec:2-handles}

Given two edges $e,e'$ at a vertex $v$, the connection defines an edge path $c$ through $e_1,\ldots,e_n$ with $e_1=e$, $e_n=e'$, and $\nabla_{e_i}(e_{i-1})=e_{i+1}$ where we set $e_0=e', e_{n+1}=e_1$. Thinking of $e_1$ as leaving the vertex $v$ and $e_n$ entering $v$ we obtain an orientation on the edge path and we define $v_i$ as the starting point of $e_i$. We call $c$ a \emph{connection path}. The construction depends only on the two initial edges and their order (where changing the latter reverses orientation of the connection path). Up to orientation and starting point of the connection path, every edge is contained in exactly two of them.

\begin{figure}[H]
	\centering
	\includegraphics{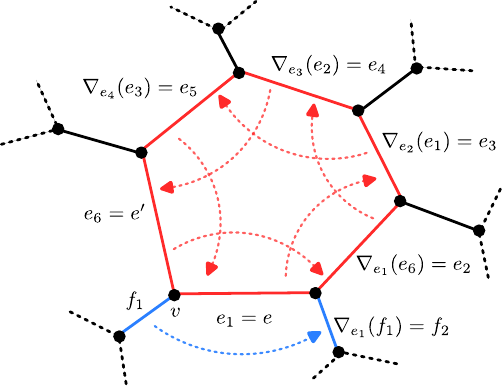}
	\caption{A connection path (red).}
	\label{fig: connection path} 
\end{figure}

\begin{lem} \label{lem:circlehomom}
If $(\Gamma,\alpha)$ is orientable then $\varphi_c=\varphi_{v_n,e_n,v_{n+1}}\circ\ldots\circ\varphi_{v_1,e_1,v_2}=\Id_{T^3}$.
\end{lem}

\begin{proof}
Let $f$ denote the third edge adjacent to $v_1$ besides $e,e'$. We note first that by the last paragraph of Remark \ref{rem:bookkeeping}, the statement does not depend on the order of the edges chosen at every vertex. Thus we assume that the fixed order at $v_1$ is $e,e',f$. The composition $\nabla_{e_n}\circ\ldots\circ\nabla_{e_1}$ around the edge path is the identity in case $n$ is even and swaps $e,e'$ while keeping $f$ fixed in case $n$ is odd. Thus for \[\sigma=\sigma_{v_n,e_n,v_1}\circ\ldots\circ\sigma_{v_1,e_1,v_2}\] we have $\mathrm{sign}(\sigma)=(-1)^n$. Furthermore by Lemma \ref{lem:signformula} and orientability we have
\[1=\prod_{i=1}^n \eta(e)=(-1)^n \mathrm{sign}(\sigma)\prod_{i=1}^n \det \varphi_{v_i,e_i,v_{i+1}}=\det \varphi_c.\]
Now we prove inductively that the third column of $\varphi_{v_{k-1},e_{k-1},v_{k}}\circ\ldots\circ\varphi_{v_1,e_1,v_2}$ has only a single nonzero entry $\pm 1$ which is located in row number $\sigma_{v_{k-1},e_{k-1},v_{k}}\circ\ldots\circ\sigma_{v_1,e_1,v_2}(3)$. Assuming that this holds for some $k$ we observe that we never transport along the edge coming from $f$, i.e.\ $e_k\neq \nabla_{e_{k-1}}\circ\ldots\circ\nabla_{e_1}(f)$, where the right hand side is the edge with number $j=\sigma_{v_{k-1},e_{k-1},v_{k}}\circ\ldots\circ\sigma_{v_1,e_1,v_2}(3)$ at the vertex $v_k$. Thus the description of $\varphi_{v_k,e_k,v_{k+1}}$ from Remark \ref{rem:bookkeeping} yields that the $j$th column has a single nonzero entry $\pm 1$ in the row number $\sigma_{v_k,e_k,v_{k+1}}$. This finishes the induction. In particular, since $\sigma(3)=3$, after a full turn we have

\[\varphi_c=\begin{pmatrix}
a_{11}&a_{12}&0\\a_{21}&a_{22}&0\\ a_{31}& a_{32} & \varepsilon\end{pmatrix}
\]
for some $a_{ij}\in \mathbb{Z},\varepsilon\in \{\pm 1\}$. Because of the commutative diagram (\ref{eq1}) we also have
\[\begin{pmatrix}
a_{11}&a_{12}&0\\a_{21}&a_{22}&0\\ a_{31}& a_{32} & \varepsilon\end{pmatrix} \cdot \begin{pmatrix}
\alpha_1 (t)\\ \alpha_2(t) \\ \alpha_3(t)
\end{pmatrix}=\begin{pmatrix}
\alpha_1 (t)\\ \alpha_2(t) \\ \alpha_3(t)
\end{pmatrix}
\]
where $\alpha_1,\alpha_2,\alpha_3$ are the weights of $e,e',k$ and $t\in T$. Now linear independence of the weights forces
\[
\begin{pmatrix}
a_{11}&a_{12}&0\\a_{21}&a_{22}&0\\ a_{31}& a_{32} & \varepsilon\end{pmatrix} =
\begin{pmatrix}
1&0&0\\0& 1&0\\ a_{31}& a_{32} & \varepsilon\end{pmatrix}.\]
Then $\det \varphi_c=1$ implies $\varepsilon=1$ and finally the equation 
$a_{31}\alpha_1+a_{32}\alpha_2+\alpha_3=\alpha_3$
yields $a_{31},a_{32}=0$.
\end{proof}

From now on we assume $(\Gamma,\alpha)$ to be orientable as a standing assumption. We define $D_c=D^2\times S^1_{\alpha_1}\times S^1_{\alpha_2}\times D^2_{\alpha_3}$ where $\alpha_1,\alpha_2,\alpha_3$ are the weights of the edges $e_1,e_n,f$ at $v_1$ in this order. Again we view the weights as a homomorphism $T\rightarrow T^3$ and denote the image by $T_c$. Furthermore let $\varphi_{c,v_1}\colon T^3\rightarrow T^3$ be the coordinate switching automorphism which permutes the weights according to the weight order in $v_1$, thus in particular mapping $T_c$ onto $T_{v_1}$. Then by Lemma \ref{lem:circlehomom} there is a well defined (and inherently unique) way to extend this definition to transition homomorphisms $\varphi_{c,e_i},\varphi_{c,v_i}$ such that the diagrams

\[\xymatrix{& & T^3\ar[dd]^{\varphi_{e_i,v}}\\
T \ar@/^1.3pc/[urr]\ar@/_1.3pc/[drr]\ar[r]& T^3\ar[ur]^{\varphi_{c,e_i}}\ar[dr]_{\varphi_{c,v}} &\\
 & & T^3
}\]
commute, where $v=v_i,v_{i+1}$, $i=1,\ldots,n$. Thus we obtain

\begin{lem}\label{lem:T3action}
There is a well-defined $T^3$-action on $M_c:=\bigcup_{i=1}^n (B_{v_i}\cup K_{e_i})$ which extends the $T$-action and acts on $K_{e_i}$ (resp.\ $B_{v_i}$) as the pullback of the standard $T^3$-action along $\varphi_{c,e_i}$ (resp. $\varphi_{c,v_i}$).
\end{lem}

\begin{rem}
The $T^3$-action being well defined along a connection path $c$ is a special occurrence. It does in general not extend over arbitrary edge paths or even globally. Note however that the $T^3$-orbit spaces of the local $T^3$-actions do in fact glue globally and will help us derive a global description of the $T^2$-orbit spaces.
\end{rem}

We find the same structure on the $D_c$ as on the $B_v$, $K_e$.
\begin{itemize}
\item We have $\overline{D_c}:=D_c/T^3\cong D^2\times I$. Denote by $F_c\subset \overline{D_c}$ the subspace $D^2\times 0$.
\item The map $s_c\colon D^2\times I\rightarrow D^2\times S^1\times S^1\times D^2$, $(p,t)\mapsto (p,1,1,t)$ is a section of the orbit map.
\end{itemize}

We also obtain a description of the $T$-orbit space of $D_c$ analogous to Lemmas \ref{lem:orbitspaceisS1bundle}. The proof is analogous to before.

\begin{lem}\label{lem:orbitspaceisS1bundle2}
The section of the $T^3$-orbit map and $T^3$-multiplication induce a homeomorphism
\[(\overline{D_c}\times T^3/T_c)/_\sim\rightarrow D_c/T_c\]
where $\sim$ collapses $T^3/T_c$ over $F_c$.
\end{lem}

Denote by $\overline{M}$ the gluing of all $\overline{B_v}$ and $\overline{K_e}$. Also set $\overline{M_c}:=M_c/T^3$. Note that the $F_e$ get glued to to the $F_v$ and the union of all of them gives rise to a topological surface with boundary $F_M\subset \overline{M}$.

\begin{lem}\label{lem:FM}
We have $\overline{M}\cong F_M\times I$, where $F_M$ gets mapped homeomorphically onto $F_M\times\{0\}$ and $F_e$ (resp. $F_v$) corresponds to $F_e\times\{0\}\subset F_e\times I$ (resp. $F_v\times\{0\}\subset F_v\times I)$.
\end{lem}
\begin{proof}
Topologically we have $\overline{B}_v=\mathbb{R}^3_{\geq 0}\cap D^3\cong (\partial \mathbb{R}^3_{\geq 0})\times I$ and $K_e=I\times \mathbb{R}^2_{\geq 0}\cap D^2\cong (I\times \partial \mathbb{R}^2_{\geq 0})\times I$. The homeomorphisms can be chosen compatible with the gluing maps.
\end{proof}

		\begin{figure}[h]
			\centering
			\includegraphics{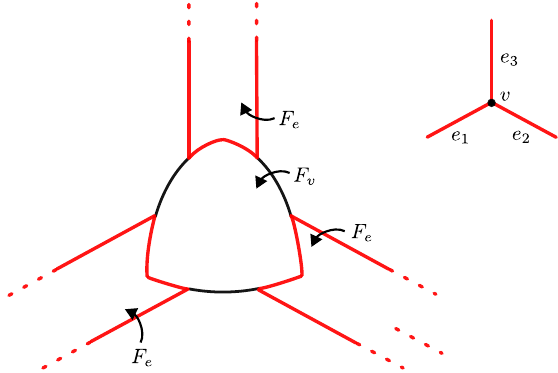}
			\caption{Local picture of $F_M$}
			\label{fig: surface locally} 
  	\end{figure}

Now with $e,e',k$ as above we consider $F_e=I\times ((\partial \mathbb{R}^2_{\geq 0})\cap D^2)=I\times (I\times\{0\}\cup \{0\}\times I)$. The intersection $F_e\cap \partial F_M$ consists of the two segments $s_1= I\times \{(1,0)\}$ and $s_2=I\times \{(0,1)\}$. Recall that the complex coordinates of $D^4$ in $K_e=I\times S^1\times D^4$ correspond to the edges $e',f$ and assume that $e'$ corresponds to the first (resp.\ second) coordinate of $D^4$. Then following the boundary component of $s_1$ (resp.\ $s_2$), which is a topological circle, in the direction of $v_2$, we obtain a closed path in $\partial F_M$ which is combinatorially exactly the edge path $c$. We denote the corresponding boundary component by $S_{e',e}\subset \partial F_M\cap \overline{M_c}$ and fix an identification $S_{e',e}\cong S^1$. Now, beginning on the level of $T^3$-orbit spaces, we glue $\overline{D_c}=D^2\times I$ to $\overline{M}$ along the map \[\overline{\psi_c}\colon S^1\times I\cong S_{e,e'}\times I\subset F_M\times I\cong \overline{M}.\]
Although it is not quite the topological boundary we also write $\partial\overline{D_c}=S^1\times I\subset \overline{D_c}$ and $\partial D_c=S^1\times S^1\times S^1\times D^2\subset D_c$. Now observe that there is a unique $T^3$-equivariant extension of this i.e.\ a unique map $\psi_{c}$ fitting into the commutative diagram

\[\xymatrix{
\partial \overline{D_c}\times T^3\ar[d]_{\overline{\psi_c}\times \Id_{T^3}}\ar[r] & \partial D_c\ar[d]_{\psi_c}\ar[r] &\overline{D_c}\ar[d]_{\overline{\psi_c}}\\
\overline{M_c}\times T^3\ar[r] & M_c \ar[r] & \overline{M_c}
}\]
where the left hand vertical maps are given by the respective sections of the orbit space maps composed with the respective $T^3$-actions (see Lemma \ref{lem:T3action}).

We view two connection paths as equivalent if they agree up to starting point and orientation. This is the case precisely if they agree on two consecutive edges, up to orientation of the path. Choose a representative $c$ for each connection path and glue the equivariant $2$-handle $D_c$ to $M$ via the procedure described above. We call the resulting manifold $N$. Let $F$ denote the union of $F_M$ with the $F_c$.

		\begin{figure}[H]
			\centering
			\includegraphics{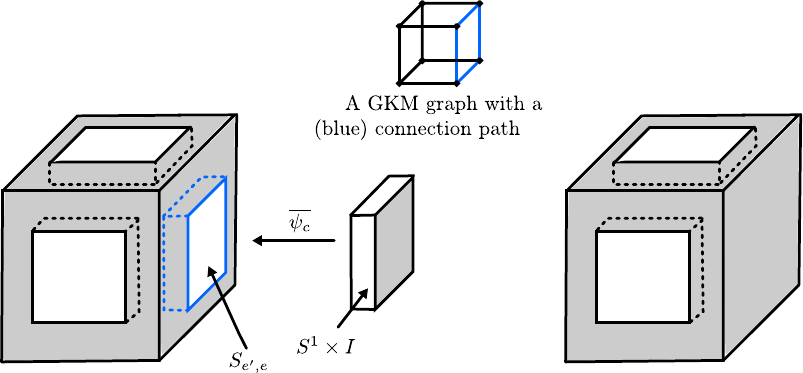}
			\caption{Before and after attaching a 2-handle to $\overline M$.}
			\label{fig: attaching two-handles} 
  	\end{figure}
  	
\begin{ex}\label{ex:toric surface}
		\begin{enumerate}[label=(\alph*)]
			\item 
Let $X$ be a $6$-dimensional quasitoric manifold and $T^2$ a subtorus of $T^3$ such that $T^2$ acts in a GKM fashion on $X$ (which holds true for generic subtori). The	$1$-skeleta of the $T^2$- and the $T^3$-action coincide. Furthermore the $T^3$-action	is also GKM and possesses a natural connection such that the connection paths are given by the faces of the orbit space, considered as a $3$-dimensional polytope. For the $T^2$-action we choose the same connection, thus obtaining an	abstract $3$-valent GKM graph. Gluing as above the $2$-discs along the	connection paths with respect to the chosen connection, we recover as surface $F$ the boundary of the $3$-dimensional polytope, which is a $2$-sphere.
			\item Let $T^2$ be the standard maximal torus of $\mathrm{SU}(3)$ and	consider the $T^2$-action by left-multiplication on the flag manifold	$\mathrm{SU}(3)/T^2$. This action is GKM \cite{GHZhom} and respects the projective	bundle structure 
	\[
	\CC\PP^{1}=\mathrm{S}(\mathrm{U}(1) \times \mathrm{U}(2))/T^2 \longrightarrow \mathrm{SU}(3)/T^2 \longrightarrow	\mathrm{SU}(3)/\mathrm{S}(\mathrm{U}(1) \times \mathrm{U}(2)) = \CC\PP^{2},
	\]
resulting in the following GKM fibration (a notion introduced in \cite{GKMFiberBundles}), see \cite[Section 2.1]{GKMFiberBundles}. The label of each edge is given by its primitive slope vector.
\begin{center}
\begin{tikzpicture}
\draw[step=1, dotted, gray] (-2.5,-3.5) grid (7.5,1.5);
\draw[very thick] (1,1) -- ++(-3,0) -- ++(4,-4) -- ++(0,3);
\draw[very thick] (2,0) -- ++(-1,1);
\draw[very thick] (2,-3)--++(-1,0);
\draw[very thick] (1,-3)--++(0,4);
\draw[very thick] (1,-3) -- ++(-3,3) -- ++ (4,0);
\draw[very thick] (-2,0) -- ++(0,1);
 \node at (1,1)[circle,fill,inner sep=2pt]{};
 \node at (-2,1)[circle,fill,inner sep=2pt]{};
 \node at (1,-3)[circle,fill,inner sep=2pt]{};
\node at (-2,0)[circle,fill,inner sep=2pt]{};
 \node at (2,0)[circle,fill,inner sep=2pt]{};
\node at (2,-3)[circle,fill,inner sep=2pt]{};
\draw[->, very thick] (3,-1) -- ++(1.5,0);
\draw[very thick] (5.0,0) -- ++(2,0);
\draw[very thick] (7.0,0) -- ++(0,-2);
\draw[very thick] (7.0,-2) -- ++(-2,2);
 \node at (5.0,0)[circle,fill,inner sep=2pt]{};
\node at (7.0,0)[circle,fill,inner sep=2pt]{};
\node at (7.0,-2)[circle,fill,inner sep=2pt]{};
\end{tikzpicture}
\end{center}
Concretely, this fibration collapses the three short (called vertical) edges, and sends pairs of parallel long  (called horizontal) edges to a single edge of the same slope in the GKM graph of $\CC P^2$.
There is a canonical connection on the GKM graph of ${\mathrm{SU}}(3)/T^2$ compatible with the GKM fibration respecting vertical and horizontal edges, and for which the trapezoids lying over the edges of the base graph are connection paths. For this connection, there is a unique fourth connection path consisting of $6$ horizontal edges which wind around the base triangle twice. Gluing the $2$-discs along the first type of connection paths results in a Möbius band (this corresponds to the fact that the GKM fibration is of \emph{twisted type} in the sense of \cite[Definition 4.1]{GKZ}). Gluing a $2$-disk along the remaining connection path gives as surface $F$ the real projective plane.
\item  A similar structure can be found on Eschenburg's twisted flag manifold $\mathrm{SU}(3)//T^2$ which is also a $\CC\PP^{1}$-bundle over $\CC\PP^{2}$. The corresponding GKM fibration is graph-theoretically identical to the one in (b), but with different labeling, see \cite[Example 4.8]{GKZ}. Thus, in this case the surface $F$ is again the real projective plane.
		\end{enumerate}
\end{ex}

\begin{thm}
\begin{enumerate}\label{thm:diskbudle}
\item $N$ is a topological $T$-manifold with boundary. The action is free on $\partial N$.
\item $N/T$ is a $D^2$-bundle over the closed topological surface $F$. More precisely it is of the form $E\times I/ \sim$ where $E\rightarrow F$ is a $S^1$-bundle with orientable total space and a global section, and $\sim$ collapses the fibers of $E$ in $E\times\{0\}$. 
\end{enumerate}
\end{thm}

\begin{proof}
Clearly $N$ is a manifold with boundary. The $T^3$-orbit maps piece together to a map $N\rightarrow F\times I$ (see Lemma \ref{lem:FM}) which maps $\partial N$ onto $F\times \{1\}$. In each of the building blocks, the $T^3$ orbits over $F\times\{1\}$ are free. The standing assumption on $(\Gamma,\alpha)$ being effective ensures that for any set of weights $(\alpha_1,\alpha_2,\alpha_3)$ meeting at a vertex the corresponding map $T\rightarrow T^3$ is injective. Hence the $T$-action is free as well.

Regarding $(ii)$ the Lemmas \ref{lem:orbitspaceisS1bundle}, \ref{lem:orbitspaceisS1bundle2} give a description of $N/T$ as an $S^1$-bundle over $F\times I$ with collapsed fibers over $F\times \{0\}$. This gives the desired description of $N/T$ as a $D^2$-bundle. The total space of the corresponding boundary circle bundle $E\rightarrow F$ is exactly $\partial N/T$ which is orientable since the $T$-action on $N$ is free and $N$ is orientable. Regarding the existence of a global section of $E$, we note that we have explicit trivializations over the $F_v$, $F_e$, and $F_c$ as $F_v\times T^3/T_v$, $F_e\times T^3/T_e$, and $F_c\times T^3/T_v$. The transition maps are induced by the gluing maps $\psi_{*,*}$ and are precisely the maps induced by the group homomorphisms $\varphi_{*,*}$. But all of these preserve the identity element. Hence the constant local sections at $1$ of these explicit trivializations glue to a global section.
\end{proof}

We continue to modify $N$ further. We start by choosing a single connection path $c$ and remove the corresponding $2$-handle by setting $N_0$ to be the closure of $N\backslash {D_c}$. The handle $D_c$ will be reglued eventually with closer attention to the gluing map.

\begin{rem}\label{rem:zerosection}
Part $(ii)$ of the previous theorem restricts to $N_0$
whose $T$-orbit space is a $D^2$-bundle over $F_0:=F\backslash \mathring{F_c}$. The corresponding circle bundle has orientable total space. The boundary of $N_0/T$ is the union of said circle bundle and the restriction of the disk bundle to $\partial F_0$. We note that the disk bundle has a preferred ``$0$-section'' whose image are those points which map to $F\times\{0\}$ under the map $N\rightarrow F\times I$. We observe that orbits are regular outside of this $0$-section.  Let $\eta\subset N_0/T$ denote the circle which is given by the $0$-section over $\partial F_0$. We note that the action on $\partial N_0$ is free outside of the orbits belonging to $\eta$.
\end{rem}

Let $g$ be the rank of $\pi_1(F_0)$.

\begin{prop}\label{prop:N0->N1}
There is a sequence of attachings of $g$ $T$-equivariant $2$-handles $D^2\times D^2\times T$ along disjoint equivariant embeddings $f_i\colon S^1\times D^2\times T\rightarrow \partial N_0$, $i=1,\ldots,g$, with the following properties
\begin{itemize}
\item The images of the $f_i$ lie over $\mathring{F_0}$.
\item The orbit space  of $N_1=N_0\cup \bigsqcup_{i=1}^g D^2\times D^2\times T$ is homeomorphic to $D^4$.
\item the loop $\eta \subset \partial (N_1/T)\cong S^3$ is an unknot.
\end{itemize}
\end{prop}

As an intermediate step Let $T^2_*$ be the $2$-torus with a small open disk removed and consider $T^2_*\times D^2$. We will make use of the following

\begin{lem}\label{lem:surgery}
There exist disjoint embeddings $f_1,f_2\colon S^1\times D^2\rightarrow \mathring{T}^2_*\times S^1$ such that there is a homeomorphism 
\[T^2_*\times D^2\cup_{f_1,f_2} \bigsqcup_{i=1,2} D^2\times D^2\cong D^2\times D^2\] which for some identification $\partial T^2_*\cong S^1$ fits into the commutative diagram
\[\xymatrix{
(\partial T^2_*)\times D^2\ar[r] & D^2\times D^2\\
S^1\times D^2\ar[u]\ar[ur] &
}\]
\end{lem}

\begin{proof}
We reduce fiber dimension by one and show that there are attaching maps $f_1,f_2\colon S^1\times D^1\rightarrow \mathring{T}^2_*\times D^1$ such that
\[T^2_*\times D^1\cup_{f_1,f_2} \bigsqcup_{i=1,2} D^2\times D^1\cong D^2\times D^1\]
and the homeomorphism identifies $(\partial T_*)\times D^1$ with $S^1\times D^1\subset D^2\times D^1$. To see this visualize $T^2_*\times D^1$ as a thickening of the standard $T^2\subset \mathbb{R}^3$ with a small thickened up disk removed. Now we add a $2$-handle $D^2\times D^1$ which lies on the ``inside'' with respect to the embedding (glued along the meridian of the torus). The result is homeomorphic to the full torus (to see this take the standard full torus in $\mathbb{R}^3$ and, without changing the homeomorphism type, push a hole through the removed disk into the inside). Now this can be turned into $D^3\cong D^2\times D^1$ in the desired fashion by gluing another $2$-handle along the longitude of the torus.

This settles the situation in dimension $3$. To obtain the statement of the lemma we just take the product of all spaces and handles with $D^1$ and and extend the gluing maps $f_i\colon S^1\times D^1\times D^1$ to be the identity on the second $D^1$ factor. Identifying $D^1\times D^1=D^2$ finishes the proof of the lemma.
\end{proof}

\begin{proof}[Proof of \ref{prop:N0->N1}] Assume first that $F_0$ is orientable.
	In this case we can write it as a connected sum $T^2_*\#T^2\#\ldots\# T^2$.
	Now we start with the ``rightmost'' $T^2$ summand (in itself homeomorphic to
	$T^2_*$) over which the $T$-orbit space is a $D^2$-bundle by Theorem
	\ref{thm:diskbudle} and the corresponding circle bundle has orientable total
	space. Consequently said circle bundle is an oriented $S^1$-bundle and hence
	in fact a principal $S^1$-bundle \cite[Prop.\ 6.15]{Morita}. As $H^2(T^2_*)=0$
	such a bundle is trivial. Hence the disk bundle over $T_*^2$ is given by
	$T_*^2\times D^2$. Applying Lemma \ref{lem:surgery} we find maps
	$f_1,f_2\colon S^1\times D^2\rightarrow \mathring{T}_*^2\times S^1$. This maps
	have image in the free orbits of $\partial N_0/T$ since the only nonregular
	orbits are in $\eta$ (see Remark \ref{rem:zerosection}) which corresponds to
	the $0$ section over $\partial T^2_*$ in the ``leftmost'' summand of the
	connected sum decomposition above. Thus the attaching map lifts to an
	equivariant attaching map $S^1\times D^2 \times T\rightarrow \partial N_0$
	(where we use that every principal $T$-bundle over $S^1\times D^2$ is
	trivial). Attaching the equivariant handles $D^2\times D^2\times T$ changes
	the orbit space by attaching the nonequivariant handles as in the Lemma. Note
	that the $D^2$-bundle structure over $\partial T^2_*$ extends to $D^2\times
	D^2$ by the commutative diagram in \ref{lem:surgery}. Orientability of the
	total space of the boundary circle bundle extends as well. Thus we find that
	the new orbit space is again a $D^2$-bundle over a connected sum $T^2_*\#
	T^2\#\ldots\# T^2$ with one  $T^2$ summand less. We continue the procedure
	until we have arrived at $D^2\times D^2$. Note that we can avoid intersections by using small disks
	when attaching the $2$-handles so we may indeed take the $f_i$ to have
	disjoint image during the induction. In the end, the knot $\eta$ which
	corresponds to the boundary $\partial T_*^2\times \{0\}$ is
	$S^1\times\{0\}\subset D^2\times D^2$ by the lemma. This is the unknot.

In the nonorientable case we may write $F_0=\mathbb{R}P^2_*\# T^2\#\ldots\# T^2$, where $\mathbb{R}P^2_*$ is $\mathbb{R}P^2$ with a small open disk removed. We use the same inductive procedure until we have arrived at a $D^2$-bundle over $\mathbb{R}P^2_*$. Since the corresponding circle bundle has orientable total space, there is only one possibility, namely the nonorientable disk bundle over the Möbius band $\mathbb{R}P^2_*$. Identifying $D^2=D^1\times D^1$, this bundle splits as $DV\times D^1$, where $DV$ denotes the disk bundle of the canonical line bundle over $\mathbb{R}P^2$ restricted to $\mathbb{R}P^2_*$. The total space of $DV$ is homeomorphic to $S^1\times D^2$, where the $0$-section of $DV$ over $\partial\mathbb{R}P^2_*$ now corresponds to the $(2,1)$-torus knot in $S^1\times S^1\subset S^1\times D^2$. Hence $\eta$ corresponds to the $(2,1)$-torus knot in $S^1\times S^1\times \{0\}\subset S^1\times D^2\times D^1$.

We now attach a $2$-handle $D^2\times D^2$ along the path $S^1\times\{0\}\times \{1\}$ using the gluing map $\varphi\colon S^1\times D^2\rightarrow S^1\times D^2\times\{1\}$. The resulting space is homeomorphic to $D^4$: to see this note that gluing $D^2$ to $S^1\times D^1$ along the path $S^1\times\{1\}$ gives a $D^2$ and the above construction is just the product of this with another $D^2$.
To conclude the proof it suffices to show that $\eta$ is the unknot in the boundary of the resulting $D^2\times D^2$. We argue in the gluing $S^1\times D^2\times D^1\cup_\varphi D^2\times D^2$ whose boundary is given as the union $S^1\times D^2\times\{-1\}\cup S^1\times S^1\times D^1\cup_\varphi D^2\times S^1$. We may isotope $\eta$ to the $(2,1)$-torus knot in $S^1\times D^2\times\{1\}$ by pushing the $D^1$-coordinate. This is identified with the $(2,1)$-torus knot in $D^2\times S^1$ when viewed in the part of the boundary that comes from the attached handle. Within this part of the boundary it can be isotoped to the $(0,1)$-torus knot in $S^1\times S^1$. Returning to the original space this corresponds to the $(0,1)$-torus knot in $S^1\times D^2\times\{1\}$. Again, we push along the $D^1$-coordinate to end up with the $(0,1)$-torus knot in $S^1\times D^2\times\{-1\}$. The resulting path can be isotoped to a point within this full torus.
\end{proof}

Recall that $D_c$ is the single $2$-handle which was removed from $N$ in order to form $N_0$. Having simplified $N_0$ to the manifold $N_1$ by gluing free $2$-handles away from $D_c$, we will now reglue $D_c$ with a more carefully chosen attaching map.

\begin{prop}\label{prop:N1->N2}
There is an equivariant embedding $\hat{\psi_c}\colon \partial D_c\rightarrow \partial N_1$ such that the gluing
\[N_2=N_1\cup_{\hat{\psi_c}} D_c\]
satisfies $N_2/T\cong S^2\times D^2$ and $\partial N_2$ is equivariantly homeomorphic to the product $S^2\times S^1\times T$.
\end{prop}

\begin{proof}
We observe that there is a commutative diagram

\[\xymatrix{
\partial D_c/T\ar[r]\ar[d] & D_c/T\ar[d]\\
S^1\times D^2\ar[r]& D^2\times D^2
}\]
with the vertical maps being homeomorphisms. To see this recall that $D_c= D^2\times S^1_{\alpha_1}\times S^1_{\alpha_2}\times D^2_{\alpha_3}$ for weights $\alpha_1,\alpha_2,\alpha_3$. Thus this comes down to $S^1_{\alpha_1}\times S^1_{\alpha_2}\times D^2_{\alpha_3}/T\cong D^2$. Let $H=\ker \alpha_1\cap\ker \alpha_2$. Then we have equivariant homeomorphisms
\[S^1_{\alpha_1}\times S^1_{\alpha_2}\times D^2_{\alpha_3}\cong T/H\times D^2_{\alpha_3}\cong T\times_H D_{\alpha_3}\]
where the action on the middle space is the diagonal one and the right hand side action is only on the $T$ factor. In particular the quotient map is the projection $T\times_H D_{\alpha_3}\rightarrow D_{\alpha_3}/H\cong D^2$ (note that $T$ and hence $H$ acts on $D_{\alpha_3}$ via rotations).

Thus we see that attaching the equivariant handle $D_c$ along the original gluing map $\psi_c$ changes the orbit space $N_1/T\cong D^4$ by attaching a $2$-handle along the unknot. The result does of course depend on the framing. We will show that the Dehn twist on $S^1\times D^2\cong \partial D_c/T$ lifts to an equivariant automorphism $\kappa$ of $\partial D_c$. Having shown this we may precompose $\psi$ with powers of $\kappa$ to achieve any framing on the level of orbit spaces. In particular we may achieve that the orbit space of the handle attachment is $S^2\times D^2$.

To define $\kappa$ let $p\colon I\rightarrow T$ be a path from $1$ to an element of $H$ such that $p(x)$ rotates $D_{\alpha_3}/H$ exactly once while $x$ goes from $0$ to $1$. Then consider
\[I\times T\times_H D_{\alpha_3}\rightarrow I\times T\times_H D_{\alpha_3},\quad (x,[t,y])\mapsto (x,[tp(x)^{-1},p(x)y])\]
which induces the desired automorphism $\kappa$ of $S^1\times T\times_H D_{\alpha_3}$ after gluing the endpoints.

The action on $\partial N_2$ as defined above is free. It remains to show that $\partial N_2\rightarrow \partial N_2/T$ admits a section.
We have $\partial N_2/T=(S^3-\im \psi_c/T)\cup (D^2\times S^1_{\alpha_1} \times S^1_{\alpha_2}\times \partial D^2_{\alpha_3})/T$. Both sides of the decomposition are individually homeomorphic to $D^2\times S^1$ so any principal bundle over it admits a section. Fix a section on the left hand side. Via the gluing map it induces a section $s$ of the $T$-action on boundary of the right hand side which we want to extend to the interior. Identifying the quotient on the right with $D^2\times S^1$ as above this comes down to the extension problem

\[\xymatrix{
S^1_1\times S^1_2\times T\ar[r] & D^2_1\times S^1_2\times T\\
S^1_1\times S^1_2\ar[u]^s\ar[r] & D^2_1\times S^1_2\ar@{-->}[u]
}\]
where we have put an extra index to the occurring circles and disks in order to
distinguish the factors in the notation. Denoting by $T^{S^1_2}$ the
corresponding space of continuous maps with the compact open topology, the datum
of a section $s$ is equivalent to a map $\hat{s}\colon S_1^1\rightarrow
T^{S^1_2}$ which we need to extend to $D_1^2$. This requires $\hat{s}$ to be
nullhomotopic which might not be the case but can be ensured by further
modifying the gluing map of the handle \[\psi_c\colon S^1_1\times
S^1_{\alpha_1}\times S^1_{\alpha_2}\times D^2_{\alpha_3}\rightarrow \partial
N_1.\] We do so by additionally multiplying the $S^1_{\alpha_1}\times
S^1_{\alpha_2}\times D^2_{\alpha_3}$ component over a point
$p \in S^1_1$ with $(\hat{s}(p)(1))^{-1}\in T$. Note that as this does not
change what happens on the level of orbit spaces. It does however change how the
sections of the principal $T$-bundles on the gluing surfaces translate. By abuse
of notation, we now have that for any $p\in S^1_1$ the map $\hat{s}(p)\colon
S^1_2\rightarrow T$ sends $1\mapsto 1$. This family of base point preserving
maps lifts to a family of paths with fixed endpoints in the universal cover
$\mathbb{R}^2\rightarrow T$. Uniform straightening of the paths in
$\mathbb{R}^2$ pushes down to $T$ and defines a homotopy from $\hat{s}$ to a
constant map. This gives the desired extension of $\hat{s}$ to $D_1^2$ and thus
we get a section of the principal $T$-bundle $\partial N_2\rightarrow \partial
N_2/T$.
\end{proof}

We have constructed a $T$-manifold with boundary $N_2$ whose orbit space is $S^2\times D^2$. We now wish to cap off the boundary $S^2\times S^1\times T^2$ to obtain a closed $T$-manifold $X$ whose orbit space is $S^4$. As we shall see in the next section, having this orbit space will imply that our constructed manifold is indeed GKM. We note that in case the stabilizers of the action were all connected, we could apply \cite[Theorem 4]{AyzenbergMasuda}.
The most canonical method to close $N_2$ off to a closed manifold is to glue $D^3\times S^1\times T$. However we do have some flexibility here, as there are many embeddings $\phi\colon S^2\times D^2\rightarrow S^4$ (see in particular Section \ref{sec:nonrigid}). Fixing such a $\phi$ we obtain a decomposition
\[
S^2\times D^2\cup Z = S^4,
\]
where $Z$ is the closure of the complement of the image of $\phi$. It is a $4$-dimensional manifold whose boundary is identified with $S^2\times S^1$ via $\phi$. Now we glue $N_2$ and $Z\times T$ along an equivariant identification of their boundaries $S^2\times S^1\times T$ to form a closed $T$-manifold $X_\phi$ with orbit space $X/T\cong S^4$. This is our desired realization.

We note that this is in line with \cite[Corollary 1.3]{AyzenbergMasuda} where the authors prove, generalizing a result of \cite{KarshonTolman} in the symplectic setting, that the orbit space of an equivariantly formal $T^{n-1}$-action with connected stabilizers on a simply-connected closed manifold of dimension $2n$ is homeomorphic to a sphere.

\begin{rem}\label{rem:condition(a)}
We interpret the labels of the graph as elements of $\hom(T,S^1)$. The isotropies occurring in $X$ are given by the trivial group, $\ker \alpha(e)$ for $e\in E(\Gamma)$ as well as $\ker\alpha(e)\cap\ker \alpha(f)$ for adjacent edges $e,f\in E(\Gamma)_v$, $v\in V(\Gamma)$. We note further that by construction any component of the fixed point set of one of these occurring subgroups contains a fixed point. The group $\ker\alpha(e)$ is connected if and only if $\alpha(e)$ is primitive and $\ker\alpha(e)\cap\ker \alpha(f)$ is connected if and only if $\alpha(e),\alpha(f)$ form a $\mathbb{Z}$-basis of $\hom(T,S^1)$. Hence $X$ has connected isotropy groups if and only if any two weights of adjacent edges form a basis.
\end{rem}

\subsection{Smoothness}

The constructed manifold $X$ can be endowed with a smooth structure making the $T$-action smooth. Let $\tilde{N}$ denote the union of $M$ with all connection handles $D_c$ (i.e.\ the original $N$ but one of the gluing maps potentially modified as in Proposition \ref{prop:N1->N2}). We construct a smooth structure on $\tilde{N}$. As $X$ arises from $\tilde{N}$ by attaching free handles, one can then complete this to a smooth structure on $X$ in the following way: the classical nonequivariant technique of smoothing the corners provides a smooth structure on the set of regular orbits $X_{reg}/T$ which extends the smooth structure on $\widetilde{N}_{reg}/T$ coming from $\widetilde{N}_{reg}$. The smooth principal bundle $\widetilde{N}_{reg}\rightarrow \widetilde{N}_{reg}/T$ is smoothly equivalent to the pullback of the universal $T$-bundle along a smooth map $f\colon N_{reg}/T\rightarrow BT$ with values in some finite skeleton $\mathbb{C}P^n\times\mathbb{C}P^n\subset BT=\mathbb{C}P^\infty\times\mathbb{C}P^\infty$. We know that the map $f$ extends to $X_{reg}/T$ as the principal bundle extends. Homotoping this extended $f$ to a smooth map as above identifies $X_{reg}$ as a smooth pullback and extends the smooth structure from $\widetilde{N}$ to all of $X$.

Now for the construction of the smooth structure on $\widetilde{N}$ we make use of the local $T^3$ action on each of the building blocks.
We introduce a chart for every building block $B_v,K_e,D_c$ and we start on the level of the local $T^3$-orbit spaces. For a connection path $c$ we set $\overline{U_v}=\mathring{D}^2\times I\subset \overline{D_c}$. For an edge $e$ set $\overline{U_e}$ to be the union of $\mathring{I}\times (\mathbb{R}^2_{\geq 0}\cap D^2)\subset \overline{K_e}$ together with small open radial extensions of its intersection with $S^1\times I=\partial\overline{D_c}$ into the interior of $\overline{D_c}$ for the two connection paths which run through $e$. Finally, for a vertex $v$ set $\overline{U_v}$ to be the union of $B_v$ with small extensions into the interior of the three adjacent $\overline{K_e}$ as well as of the three adjacent $\overline{D_c}$. Now denote by $U_c,U_e,U_v$ the respective part in $\widetilde{N}$ lying over $\overline{U_c},\overline{U_e},\overline{U_v}$. The standard $T^3$-action on ${B_v}$ extends to $U_v$ by twisting it with the appropriate automorphisms $\varphi_{v,\bullet}\colon T^3\rightarrow T^3$ when entering the other adjacent building blocks (the same holds for $U_e$). The sections over the building blocks are compatible with the gluing maps and hence induce sections on the $\overline{U}_\bullet$ which agree on their intersections. The only exception is the equivariant $2$-handle $D_c$ whose attaching map has been modified in Proposition \ref{prop:N1->N2}. Here the gluing map has been additionally multiplied with certain elements of $T^3$. For the $U_v$ (and analogously the $U_e$) touching $D_c$, the sections can still be extended over $\overline{U_v}$ but not necessarily in a way that they agree with the section defined on $\overline{U_c}$.

 Now we may choose a smooth atlas of the underlying $3$-dimensional manifold with corners in the following way: identify $\overline{U_c},\overline{U_e},\overline{U_v}$ with smooth subset $\overline{V_c},\overline{V_e},\overline{V_v}$ of $\mathbb{R}^3_{\geq 0}$ preserving the face structure and such that the transition functions between the $\overline{V_\bullet}$ are smooth. We denote by $V_\bullet\subset \mathbb{C}^3$ the $T^3$-invariant smooth subspace generated by $\overline{V}_\bullet$. We have sections $\overline{V_\bullet}\rightarrow V_\bullet$ induced by $\mathbb{R}_{\geq 0}^3\rightarrow \mathbb{C}^3$.
Then, using the sections $\overline{U_\bullet}\rightarrow U_\bullet$ and $\overline{V_\bullet}\rightarrow V_\bullet$, the maps $\overline{V_\bullet}\rightarrow \overline{U_\bullet}$ can be extended $T^3$-equivariantly to maps $V_\bullet\rightarrow U_\bullet$. Smoothness of the transition functions can be argued as follows: they are by construction twisted $T^3$-equivariant, i.e.\ they become $T^3$-equivariant after pulling back one of the actions along a suitable automorphism $\varphi_{\bullet,\bullet}$ of $T^3$. Furthermore the transition maps respect the section $\overline{V_\bullet}\rightarrow V_\bullet$ except in the case of the modified $2$-handle $D_c$ above where it additionally multiplies with certain elements of $T^3$. We note that these elements can be chosen to vary smoothly in $\overline{V_c}$. Thus the sections are preserved after additionally composing with a suitable smooth equivariant automorphism of the respective subset of $V_c$. It is shown in \cite[Theorem 4.1]{Wiemeler} (see also the initial sentences of the proof) that equivariant extensions of regular sections as above are smooth. Hence the $V_\bullet$ give a smooth atlas. 
The $T$-action is smooth because the (local) $T^3$-action is.

\section{Topological properties of the realization}\label{sec:cohom}
\subsection{Generalities}
\label{sec:generalities}
In this section we prove that the realization $X = X_\phi$ constructed in the previous section does indeed satisfy the (integer) GKM conditions in case the initial abstract GKM graph $(\Gamma,\alpha)$ satisfies the appropriate necessary conditions --- for any choice of $\phi$.
Let $G\subset X$ denote the $1$-skeleton of the action. By construction this is the graph of $2$-spheres encoded in $(\Gamma,\alpha)$. Also for a connection edge path $c_i\colon e_1,e_2,\ldots$ let $\alpha_1,\alpha_2$ be the weights at $v_1$ (with notation as before) and $H_i=\ker \alpha_1\cap \ker \alpha_2$. Then we obtain an equivariant map $\tilde{c_i}\colon S^1\times T/H_i\rightarrow G$ following the edge path $c_i$ which is unique up to equivariant homotopy (up to orientation of the path) and we choose it such that $S^1\times\{1\}$ gets embedded transversely to the orbits. Denote by $\widetilde{G}$ the space that arises by gluing equivariant $T/H_i$-cells $D^2\times T/H_i$ along all the $\tilde{c_i}$. Denote by $\widetilde{N}$ the union of $M$ with all connection handles, i.e.\ the original $N$ but with the gluing map of a single handle replaced by the modified $\hat{\psi_c}$ from Proposition \ref{prop:N1->N2}. Alternatively one can view it as $N_2$ without the free handles attached during Proposition \ref{prop:N0->N1}.

\begin{prop}\label{prop:retraction}
The inclusion $G\rightarrow M$ is an equivariant homotopy equivalence. It extends to an equivariant homotopy equivalence $\widetilde{G}\rightarrow \widetilde{N}$.
\end{prop}

\begin{proof}
Recall that for each building block $B_v$, $v$ a vertex of $\Gamma$, we have a section $\overline{B_v}\rightarrow B_v$ and a fixed identification $\overline{B_v}\cong F_v\times I$. These combine to a $T^3$-equivariant surjection $F_v\times I\times T^3\rightarrow B_v$ with the nonregular $T^3$-orbits contained in the image of $F_v\times \{0\}\times T^3$. Thus
\[r_t\colon F_v\times I\times T^3\rightarrow F_v\times I\times T^3,\quad (x,s,g)\mapsto (x,ts,g)\]
descends to an equivariant deformation retraction of $B_v$ to the part lying over $F_v\times\{0\}\subset\overline{B_v}$. This retraction can be defined analogously over the $1$- and $2$-handles $K_e$ and $D_c$. Defining the $r_t$ in this way we see that they are in fact compatible with the gluing maps, even for the modified $\hat{\psi_c}$ from Proposition \ref{prop:N1->N2}. Hence we obtain a global deformation retraction of $\widetilde{N}$ onto the part lying over the compact surface $F$. This is exactly $\widetilde{G}$. The retraction is $T$-equivariant as it is $T^3$-equivariant on every building block. Note that the retraction respects $M$. The intersection $\widetilde{G}\cap M$ is just $\widetilde{G}$ with small equivariant $2$-disks removed from the interior of every equivariant $2$-cells. In particular this deformation retracts further onto the $1$-skeleton of the action, which is $G\subset M$.
\end{proof}

\subsection{Fundamental group}

\begin{lem}
The map
\[\pi_1(G_T)\rightarrow \pi_1(G/T)\]
is an isomorphism.
\end{lem}

\begin{proof}
The existence of a fixed point in $G$ implies surjectivity of $\pi_2(G_T)\rightarrow \pi_2(BT)$ as can be seen by comparing with the long exact homotopy sequence of the Borel fibration of the action on the fixed point. Hence $\pi_1(G)\rightarrow \pi_1(G_T)$ is an isomorphism. The Lemma now follows by observing that $\pi_1(G)\rightarrow \pi_1(G/T)$ is an isomorphism.
\end{proof}

\begin{prop}\label{prop:pi1N2}
We have $\pi_1((N_2)_T)=0$.
\end{prop}

\begin{proof}
We prove that $\pi_1((N_2)_T)\rightarrow \pi_1((N_2)/T)$ is an isomorphism. and use $N_2/T=S^2\times D^2$.
As $M$ deformation retracts equivariantly onto $G$ the previous Lemma implies that $\pi_1(M_T)\rightarrow \pi_1(M/T)$ is an isomorphism. Now we show that this property is stable under suitable equivariant $2$-handle attachment.

Consider such a $2$-handle of the form $D^2\times S^1_{\alpha_1}\times S^1_{\alpha_2} \times D_{\alpha_3}= D^2\times T/H\times D^2_{\alpha_3}$, where $H=\ker\alpha_1\cap\ker \alpha_2$, which gets glued to a space $Y$ along an equivariant map $\varphi\colon S^1\times T/H\times D^2_{\alpha_3}\rightarrow Y$. Set $Y'$ to be the union of $Y$ and the equivariant handle. Note that we may equivariantly contract the $D_{\alpha_3}$ factor hence it suffices to study gluings of equivariant $2$-cells. Identifying $(D^2\times T/H)_T\cong D^2\times BH$, we obtain a diagram

\[\xymatrix{
\pi_1(Y_T)\ar[d] & \pi_1( S^1\times BH)\ar[d]\ar[l]\ar[r] & \pi_1(D^2\times BH)\ar[d]\\
\pi_1(Y/T)& \pi_1( S^1)\ar[l]\ar[r] & \pi_1(D^2)
}\]
where by Seifert van Kampen the pushout over the first (resp.\ 2nd) row yields $\pi_1(Y'_T)$ (resp.\ $\pi_1(Y'/T)$). Now we assume that the left hand vertical map is an isomorphism and want to show that the induced map between the pushouts is an isomorphism. To do this it suffices to show that $\pi_1(BH)\rightarrow \pi_1(S^1)\oplus \pi_1(BH)\cong \pi_1(S^1\times BH)\rightarrow \pi_1(Y_T)$ is trivial. This is the map on equivariant cohomology induced by the inclusion of an orbit $T/H\rightarrow Y$. This is not homotopically trivial in general. However in the case of attaching handles to $M$ we note that any orbit inclusion of $T/H$ is equivariantly homotopic to the constant map because every component of $M^H$ contains a fixed point by construction. Now we apply the above discussion to the connection handles, which get glued disjointly to $M$, and subsequently to the handles added in Proposition $\ref{prop:N0->N1}$ for which no problems arise since here $H=0$.
\end{proof}

\begin{prop}
We have $\pi_1(X)=0$.
\end{prop}

\begin{proof}
We compute $\pi_1(X_T)$ as the pushout of

\[\pi_1((N_2)_T)\longleftarrow \pi_1( (S^2\times S^1\times T)_T)\longrightarrow\pi_1((Z\times T)_T)
\]
by applying Seifert van Kampen. The left hand side is trivial by the previous proposition and the right hand side arrow can be identified with $\pi_1(S^2\times S^1)\rightarrow \pi_1(Z)$. Thus, $\pi_1(X_T)$ is the same as the pushout of 
\[\pi_1(S^2\times D^2)\longleftarrow \pi_1( S^2\times S^1)\longrightarrow\pi_1(Z),
\]
i.e., $\pi_1(X_T)=\pi_1(S^4)=0$. Comparing the long exact homotopy sequences of the Borel fibrations of $X$ and a fixed point of $X$ yields surjectivity of $\pi_2(X_T)\rightarrow \pi_2(BT)$. Hence $\pi_1(X)\cong\pi_1(X_T)=0$.
\end{proof}

\subsection{Rational cohomology}

\begin{lem}\label{lem:dimensionstuff}
We have $H^{even}_T(G)\cong H^*_T(\Gamma,\alpha)$. If $H^*(\Gamma,\alpha;\mathbb{Q})$ satisfies (6-dimensional) Poincaré duality, then $H_T^3(G;\mathbb{Q})=\mathbb{Q}$ as well as $H_T^{2n+1}(G;\mathbb{Q})=0$ for $n\geq 2$.
\end{lem}

\begin{proof}
We cover $G$ with the neighbourhood $V$ which consists of a small ball around every fixed point as well as the neighbourhood $E$ which is the interior of the edges. The corresponding Mayer-Vietoris sequence (with integer coefficients) reads
\[0\rightarrow H_T^{even}(G)\rightarrow H_T^{even}(V)\oplus H_T^{even}(E)\rightarrow H_T^{even}(E\sqcup E)\rightarrow H^{odd}_T(G)\rightarrow 0.\]
As the map $H_T^*(E)\rightarrow H_T^*(E\sqcup E)$ is just the diagonal $\Delta$, the chosen orientation for each edge yields an isomorphism $H^*(E\sqcup E)/\im(\Delta)\rightarrow H^*(E)$. Hence the above sequence reduces to an exact sequence
\[0\rightarrow H_T^{even}(G)\rightarrow H_T^{even}(V)\rightarrow H_T^{even}(E)\rightarrow H^{odd}_T(G)\rightarrow 0.\]
This gives the description $H_T^{even}(G)\cong H^*_T(\Gamma,\alpha)$. The rest of the lemma follows via a dimension argument: let $2v$ be the number of vertices in $\Gamma$. By Proposition \ref{prop:eqcohomfreeoverQ}, $H^{even}_T(G;\mathbb{Q})=H^*_T(\Gamma,\alpha;\QQ)$ is free over $R_\QQ$. By Lemma \ref{lem:rankgraphcohom} and Poincaré duality, it has exactly $1,v-1,v-1,1$ generators in degrees $0,2,4,6$. Using that $\dim R_\QQ^{2k}=k+1$, a calculation shows that $\dim H^0_T(G;\mathbb{Q})=1$, $\dim H^2_T(G;\mathbb{Q})=v+1$, and $\dim H^{2n}_T(G;\mathbb{Q})=(2n-1)v$ for $n\geq 2$. Furthermore $\dim H^{2n}_T(V;\mathbb{Q})=(n+1)2v$ and $\dim H^{2n}_T(E;\mathbb{Q})=3v$ (for the last equality note that there are $3v$ edges and each of the corresponding algebras has $R_\QQ$-annihilator generated by the weight in degree $2$).
We compute that $\dim H^3_T(G;\mathbb{Q})=2\cdot 3v-3v-4v+v+1=1$ and $\dim H^{2n+1}_T(G;\mathbb{Q})=3v-(n+1)2v+(2n-1)v=0$ for $n\geq 2$.
\end{proof}

\begin{lem}\label{lem:G->N1}
The map $H^k_T(N_1;\mathbb{Q})\rightarrow H^k_T(G;\mathbb{Q})$ is an isomorphism for $k\geq 2$.
\end{lem}

\begin{proof}
Recall that $G\rightarrow M$ is an equivariant homotopy equivalence and $N_1$ arises out of $M$ by attaching equivariant almost free $2$-handles $D^2\times T\times_H D^2$ for some finite $H\subset T$. Since the action on the handle is almost free we may compute its rational equivariant cohomology via the orbit space, which yields $H^*_T(D^2\times T\times_H D^2;\mathbb{Q})=\mathbb{Q}$ and $H^*_T(S^1\times T\times_H D^2;\mathbb{Q})=H^*(S^1;\mathbb{Q})$. If a $T$-space $Y'$ arises out of a $T$-space $Y$ by equivariant attachment of such a handle a Mayer-Vietoris argument yields $H^k(Y';\mathbb{Q})=H^k(Y;\mathbb{Q})$ for $k\geq 3$ as well as an exact sequence
\[0\rightarrow H^1_T(Y';\mathbb{Q})\rightarrow H^1_T(Y;\mathbb{Q})\rightarrow \mathbb{Q}\rightarrow H^2_T(Y';\mathbb{Q})\rightarrow H^2_T(Y;\mathbb{Q})\rightarrow 0.\]
Thus either $\dim H^1_T(Y;\mathbb{Q})$ decreases by $1$ or $\dim H^2_T(Y;\mathbb{Q})$ increases by $1$. We have already shown that $\pi_1((N_1)_T)=0$ and hence $H_T^1(N_1;\mathbb{Q})=0$. Thus it suffices to check that any handle attached to $M$ on the way to $N_1$ decreases the rank of $H^1_T(M;\QQ)$ which will imply that $H^2_T(M;\QQ)$ remains unchanged, finishing the proof.

As argued in the proof of Proposition \ref{prop:pi1N2} we have an isomorphism $\pi_1(Y_T)\rightarrow \pi_1(Y/T)$ for any intermediate stage $Y$ on the way from $M$ to $N_1$. In particular $H_T^1(Y;\mathbb{Q})\cong H^1(Y/T;\mathbb{Q})$. Hence it suffices to consider orbit spaces of these $Y$ which are all $D^2$-bundles over surfaces with boundary (recall that we have proved this over all building blocks and the gluings are compatible with this structure, cf.\ Theorem \ref{thm:diskbudle}). The orbit space $M/T$ is a $D^2$-bundle over some surface with boundary. Each handle attached on the way to $N_0$ reduces the number of boundary components by $1$ up to a minimum of $1$ component for $N_0$. In particular $\dim H^1$ decreases in each step. When going from $N_0$ to $N_1$ as in Proposition \ref{prop:N0->N1}, we attach pairs of handles where each pair eliminates a $T^2$ summand from the surface. Consequently each handle reduces $\dim H^1$ by $1$.
\end{proof}

\begin{rem}
The reader may wonder why there is no $N_2$ appearing in this section although it was present in the construction of $X$. The reason for this is that we had to introduce $N_2$ in order to deal with finite isotropies and that from a rational standpoint this is a detour when constructing $X$ out of $N_1$.
\end{rem}

\begin{lem}\label{lem:N_1->X}
The map $H_T^k(X;\mathbb{Q})\rightarrow H_T^k(N_1;\mathbb{Q})$ is an isomorphism for $k\neq 3,4$ and there is an exact sequence
\[0\rightarrow H_T^3(X;\mathbb{Q})\rightarrow H_T^3(N_1;\mathbb{Q})\rightarrow \mathbb{Q}\rightarrow H_T^4(X;\mathbb{Q})\rightarrow H_T^4(N_1;\mathbb{Q})\rightarrow 0.\]
\end{lem}

\begin{proof}
Note that $\partial N_1/T=S^3$ and the action on $\partial N_1$ is almost free. Hence $H_T^*(\partial N_1;\mathbb{Q})=H^*(S^3;\mathbb{Q})$. Also $(X-N_1)/T$ is the complement of a $D^4$ in $S^4$ and in particular homologically trivial. Hence $H_T^*(X-N_1;\mathbb{Q})=\mathbb{Q}$. The lemma now follows from the equivariant Mayer-Vietoris sequence of $X=N_1\cup (X-N_1)$.
\end{proof}

\begin{thm}\label{thm:rationalGKM}
If $H^*(\Gamma,\alpha;\mathbb{Q})$ satisfies (6-dimensional) Poincaré duality, then $X$ is equivariantly formal over $\mathbb{Q}$. Hence it is a rational GKM manifold.
\end{thm}

\begin{proof}
We have $H^3_T(N_1;\mathbb{Q})=\mathbb{Q}$ by Lemmas \ref{lem:dimensionstuff} and \ref{lem:G->N1}.
Let $\varphi\colon H_T^3(N_1;\mathbb{Q})\rightarrow \mathbb{Q}$ denote the corresponding map in the exact sequence from Lemma \ref{lem:N_1->X}. If $\varphi\neq 0$ then $H_T^{odd}(X;\mathbb{Q})=0$ and
\[H_T^{even}(X;\mathbb{Q})\rightarrow H_T^{even}(G;\mathbb{Q})=H^*_T(\Gamma,\alpha;\mathbb{Q})\]
is an isomorphism. As $H^*_T(\Gamma,\alpha;\mathbb{Q})$ is a free module by Proposition \ref{prop:eqcohomfreeoverQ}, it follows that $X$ is equivariantly formal.

Now in the case $\varphi=0$ we will derive a contradiction. We obtain
\[\dim H^4_T(X;\mathbb{Q})>\dim H^4_T(N_1;\mathbb{Q})=\dim H^4(\Gamma,\alpha;\mathbb{Q}).\]
Let $n=\dim H^2(X;\mathbb{Q})=\dim H_T^2(X;\mathbb{Q})-2=\dim H^2(\Gamma,\alpha;\mathbb{Q})-2$. As $H^*_T(\Gamma,\alpha;\mathbb{Q})$ has a single generator in degree $0$, we deduce that a minimal $R_\QQ$-generating set of the latter has $n$ elements in degree $2$. by the Poincaré duality assumption there are also $n$ generators in degree $4$ and $R_\QQ$-freeness implies $\dim H^4(\Gamma,\alpha;\mathbb{Q})= 3+ 3n$. On the other hand, using Poincaré duality of $X$ we deduce that $3+3n$ is the total degree $4$ dimension of $R_\QQ\otimes H^*(X;\mathbb{Q})$. Thus the Serre spectral sequence of the Borel fibration of $X$ yields $\dim H_T^4(X;\mathbb{Q})\leq 3+3n=\dim H^4(\Gamma,\alpha;\mathbb{Q})$ which is a contradiction.
\end{proof}

\subsection{Integral cohomology}

The Mayer-Vietoris sequence of $S^4 = S^2\times D^2\cup Z$ implies
\begin{lem}\label{lem:homZ}
We have $H^1(Z)={\mathbb{Z}}$ and $H^2(Z)=H^3(Z)=H^4(Z)=0$. 
\end{lem} Note that the fundamental group of $Z$ depends on the particular choice of $Z$, see Section \ref{sec:nonrigid} below.

\begin{prop}\label{prop:torsionfree}
If $H^3_T(\widetilde{G})$ is torsion free then $H^*(X)$ is torsion free. 
\end{prop}

\begin{proof}
Note that $H^1(X)=0$ and $H^2(X)$ is torsion free since its torsion agrees with the torsion of $H_1(X)=0$ by universal coefficients. Also $H^5(X)\cong H_1(X)=0$ and $H^6(X)=\mathbb{Z}$. Now if $H^3(X)$ were torsion free this would imply torsion freeness of $H_2(X)\cong H^4(X)$ by universal coefficients and Poincaré duality. Thus we may concentrate on degree $3$.

Now if $H^3(X)$ had torsion this would give torsion in the $E_2^{0,3}$ entry of the Serre spectral sequence of the Borel fibration. As $H^2(X)$ is torsion free, the same holds for the $E_2^{2,2}$ entry and hence any torsion in $E_2^{0,3}$ survives until $E_\infty$. It would follow that also $H^3_T(X)$ has torsion thus it suffices to show that the latter is torsion free.

Using Lemma \ref{lem:homZ}, the equivariant Mayer-Vietoris sequence of $X= N_2\cup (Z\times T)$ yields the exact sequence
\[H_T^2(N_2)\rightarrow H_T^2(S^2\times S^1\times T)\rightarrow H_T^3(X)\rightarrow H_T^3(N_2).\]
We claim that the first map in this sequence is surjective. Indeed, consider the diagram
\[\xymatrix{
H_T^2(N_2)\ar[r] & H_T^2(S^2\times S^1\times T)\\
H^2(N_2/T)\ar[u]\ar[r] & H^2(S^2\times S^1)\ar[u]_\cong
}\]
where the right hand vertical arrow is an isomorphism due to freeness of the action. But the bottom horizontal map is induced by $S^2\times S^1\rightarrow S^2\times D^2$ and in particular surjective. This proves the claim.

As a consequence $H_T^3(X)\rightarrow H_T^3(N_2)$ is injective. However $H_T^3(N_2)$ arises from $\widetilde{N}$ (as above Proposition \ref{prop:retraction}) by attaching free handles so $H_T^3(N_2)=H_T^3(\widetilde{N})=H_T^3(\widetilde{G})$. Hence if $H_T^3(\widetilde{G})$ is torsion free, so is $H_T^3(X)$.
\end{proof}

\begin{lem}\label{lem:restriction}
Assume that any torsion element of $H_T^3(G)$ restricts nontrivially to $H_T^3(G^U)$ for some nontrivial subgroup $U\subset T$. Then $H_T^3(\widetilde{G})$ is torsion free.
\end{lem}

\begin{proof}
We have an exact Mayer vietoris sequence associated to covering $\widetilde{G}$ with a thickened $G$ and the interiors of the equivariant $2$-cells. This gives the sequence
\[0\rightarrow H^3_T(\widetilde{G})\rightarrow H^3_T(G)\rightarrow H^3_T\left(\bigsqcup_{i=1}^l S^1\times T/H_i\right)\]
where we have used that $H_T^3(D^2\times T/H_i)=0$ and $H_T^2(D^2\times T/H_i)\rightarrow H_T^2(S^1\times T/H_i)$ is surjective. In order to show torsion freeness of $H_T^3(\widetilde{G})$ it suffices to show that any torsion element $x\in H_T^3(G)$ gets mapped nontrivially under some $\tilde{c_i}$ (as defined in the beginning of Section \ref{sec:generalities}). By assumption $x$ restricts nontrivially to $H_T^3(G^U)$ for some nontrivial $U$. This is a graph of $2$-spheres. Any connected component with less than $2$ edges does not contribute to $H_T^3(G^U)$. In particular $x$ restricts nontrivially to a path component $P$ of $G^U$ with two adjacent $2$-spheres corresponding to edges $e_1,e_2$ with weights $\alpha(e_1),\alpha(e_2)$. But then since $U\in \ker \alpha(e_1)\cap \ker\alpha(e_2)$ it follows that $U\subset \ker \alpha(\nabla_{e_1}(e_2))$. Inductively we see that all spheres belonging the connection path $c_i$ through $e_1,e_2$ all lie in $P$. In fact this is already the entirety of $P$ since no three edges meeting at a vertex can vanish on $U$ due to the effectivity assumption. The same argument shows that $P$ is graph theoretically a simple loop.

Note that $Q:=P\cup_{\tilde{c_i}} D^2\times T/H_i$ with induced $T/H_i$-action is an effective quasitoric manifold. By \cite[Proposition 7.3.23]{ToricTopology}, $H^*(Q)$ is concentrated in even degrees, and consequently also $H_{T/H_i}^*(Q)$. Then the same holds for $H_T^*(Q)=H^*(BT)\otimes_{H_T^*(BT/H_i)} H_{T/H_i}^*(Q)$. By a similar Mayer-Vietoris sequence to above we deduce that the attaching map $\tilde{c_i}^*\colon H_T^3(P)\rightarrow H_T^3(S^1\times T/H_i)$ is injective. Thus $x$ does not vanish under $\tilde{c_i}^*$.
\end{proof}

\begin{lem}\label{lem:depth}
Assume that $H^*_T(\Gamma,\alpha)$ is a free $R$-module and let $x\in H_T^3(G)$ be a nontrivial torsion element. Then $\sqrt{\Ann(x)}$ does not contain $R^+$.
\end{lem}

\begin{proof}
As in Lemma \ref{lem:dimensionstuff} we obtain the exact sequence 
\[0\rightarrow H_T^{even}(G)\rightarrow H_T^{even}(V)\rightarrow H_T^{even}(E)\rightarrow H^{odd}_T(G)\rightarrow 0.\]
Assume $x$ is of order $n$ and let $p$ be a prime factor of $n$. By replacing $x$ with $\frac{n}{p}x$ we may assume for the sake of the lemma that $x$ is of order $p$. Now we localize at the maximal ideal $\mathfrak{p}=(p,R^+)\subset R$. We define $\mathfrak{q}$ as the maximal ideal of $R_\mathfrak{p}$ and consider the $\mathfrak{q}$-depth of the $R_\mathfrak{p}$-modules. Set $M= H_T^{even}(V)_\mathfrak{p}/H_T^{even}(G)_\mathfrak{p}$. By assumption and Lemma \ref{lem:dimensionstuff}, $H_T^{even}(G)_\mathfrak{p}$ is a free $R_\mathfrak{p}$-module and thus $\depth_\mathfrak{q}(H_T^{even}(G)_\mathfrak{p})=3$. The same holds for $H_T^{even}(V)_\mathfrak{p}$ and thus we obtain $\depth_\mathfrak{q}(M)\geq 2$. Now the exact sequence above gives the short exact sequence
\[0\rightarrow M\rightarrow H_T^{even}(E)_\mathfrak{p}\rightarrow H_T^{odd}(G)_\mathfrak{p}\rightarrow 0.\]

The middle term is a direct sum over modules of $\mathfrak{q}$-depth $\geq 1$ so $\depth_\mathfrak{q}(H_T^{even}(E)_\mathfrak{p})\geq 1$. Consequently we obtain $\depth_\mathfrak{q}H_T^{odd}(G)_\mathfrak{p} \geq 1$. By the description of depth via regular sequences it follows that there is an element $\alpha\in\mathfrak{q}$ such that multiplication with $\alpha$ is injective on $H_T^{odd}(G)_\mathfrak{p}$. We may assume the denominator of $\alpha$ to be $1$ thus $\alpha\in \mathfrak{q}\cap R=\mathfrak{p}$. Now $x$ defines a nontrivial element in $H_T^{odd}(G)_\mathfrak{p}$. It follows inductively that $\alpha^i x$ is nonzero for any $i\geq 0$. Observe that $\alpha$ is of the form $\alpha=k\cdot p+f$ for some $f\in R^+,~k\in\mathbb{Z}$ and that $px=0$ implies $\alpha^ix=f^ix$. Thus $\sqrt{\Ann(x)}$ does not contain $f$.
\end{proof}

\begin{lem}\label{lem:subgroupcorrespondence}
Let $C\subset R^2$ be a subgroup such that $R^2/C$ is finite. Then there is a unique finite subgroup $U\subset T$ with $C=\ker(R^2\rightarrow H^*(BU))$. If furthermore $U'\subset T$ is any closed subgroup and $\ker(R^2\rightarrow H^*(BU')) \subset C$, then $U\subset U'$.
\end{lem}

\begin{proof}
Let $n$ be the order of $R^2/C$. Then $R^2\supset C\supset nR^2$ and $C$ induces a subgroup $\overline{C}$ of $\Lambda= R^2/nR^2$. Let $T\supset T_n\cong \mathbb{Z}_n\times \mathbb{Z}_n$ be the subgroup of $n$-torsion elements. Then $BT_n\rightarrow BT$ induces an isomorphism $\Lambda\cong H^2(BT_n)$. We have \[H^2(BT_n)\cong \mathrm{Ext}^1_\mathbb{Z}(H_1(BT_n),\mathbb{Z})\cong \hom(H_1(BT_n),\mathbb{Z}_n)\cong \hom(\pi_1(BT_n),\mathbb{Z}_n)\cong \hom(T_n,\mathbb{Z}_n).\]
The second isomorphism is obtained from the long exact sequence
\[0\rightarrow\hom(T_n,\mathbb{Z}_n)\rightarrow \mathrm{Ext}^1(T_n,\mathbb{Z})\xrightarrow{\cdot n} \mathrm{Ext}^1(T_n,\mathbb{Z})\]
associated to the short exact sequence $0\rightarrow \mathbb{Z}\rightarrow \mathbb{Z}\rightarrow\mathbb{Z}_n\rightarrow 0$. For the remaining identifications we use universal coefficients, the Hurewicz Theorem, and the identification of $\pi_1(BT_n)$ with the deck group of $ET\rightarrow BT_n$ (with respect to some choice of base point in $ET$). Now for some $U\subset T_n$ we get an induced continuous map $BU=ET/U\rightarrow ET/T_n=BT_n$ and the induced square
\[\xymatrix{
H^2(BT_n)\ar[d]\ar[r]^\cong& \hom(T_n,\mathbb{Z}_n)\ar[d]\\
H^2(BU)\ar[r]^\cong& \hom(U,\mathbb{Z}_n)
}\]
commutes by naturality of the isomorphism chain above.
Since the maps \[T_n\supset U\mapsto \ker(\hom(T_n,\mathbb{Z}_n)\rightarrow \hom (U,\mathbb{Z}_n))\quad\text{and}\quad\hom(T_n,\mathbb{Z}_n)\supset U\mapsto \bigcap_{\phi\in U}\ker\phi\] are inverse to one another, it follows that there is a unique group $U\subset T_n$ such that we have $\ker (H^2(BT_n)\rightarrow H^2(BU))=\overline{C}$. But then $C=\ker(H^2(BT)\rightarrow H^2(BU))$. Furthermore since any two finite subgroups $U$ and $U'$ lie in a common subgroup of the form $T_n$ we obtain uniqueness of $U$.

Now let $U'\subset T$ be any closed subgroup such that $D:=\ker(R^2\rightarrow H(BU'))$ is contained in $C$. The group $U'$ is isomorphic to a product of copies of $S^1$ and finite cyclic groups. Thus it follows by looking at individual factors that if we set $U'_n$ to be the $n$-torsion subgroup then the kernel of $H^2(BU')\rightarrow H^2(BU'_n)$ is exactly $n\cdot H^2(BU')$. It follows that $D':=\ker(R^2\rightarrow H^2(BU'_n))=n\cdot R^2+D\subset C$. Now $U'_n$ and $U$ from above are both subgroups of $T_n$ and the prior results on subgroups of $T_n$ imply $U\subset U'_n$. This finishes the proof.
\end{proof}

\begin{prop}\label{prop:freeimpliestorsioncondition}
If $H^*_T(\Gamma,\alpha)$ is a free $R$-module, then $H^*(X)$ is torsion free.
\end{prop}

\begin{proof}
We show that $H_T^3(\widetilde{G})$ is torsion free and use Proposition \ref{prop:torsionfree}.
Combining Lemma \ref{lem:restriction} and Lemma \ref{lem:depth} it suffices to prove the following. An element $x\in H_T^*(G)$ for which $\sqrt{\Ann(x)}$ does not contain $R^+$ restricts nontrivially to $H_T^*(G^U)$ for a nontrivial subgroup $U\subset T$. As $\sqrt{\Ann(x)}$ is the intersection of all prime ideals containing $\Ann(x)$, there is a prime ideal $\mathfrak{p}\subset R$ which contains $\Ann(x)$ but does not contain $R^+$. In particular $\mathfrak{p}\cap R^2$ is strictly smaller than $R^2$ and we choose a finite index strict subgroup $C\subset R^2$ which contains $\mathfrak{p}\cap R^2$.
Denote by $S\subset R$ the multiplicative subset generated by $R^2\backslash C$. We observe that $(R^2\backslash C)\cap \mathfrak{p}=\emptyset$ implies $S\cap\mathfrak{p}=\emptyset$ and hence $S\cap \Ann(x)=\emptyset$. Thus $x$ induces a nontrivial element of $S^{-1}H_T^*(G)$.

By Borel localization \cite[Theorem 3.2.6]{AlldayPuppe} it follows that $x$ maps to a nontrivial element in $H_T^*(G^S)$, where $G^S=\{p\in G~|~ \text{no element of $S$ becomes trivial in $H^*(BT_p)$}\}$ and $T_p$ is the stabilizer of $p$.
Now by Lemma \ref{lem:subgroupcorrespondence} there is a unique (finite) nontrivial subgroup $U\subset T$ such that $C=\ker(H^2(BT)\rightarrow H^2(BU))$. The proof is complete if we can show $G^S\subset G^U$ as then $x$ restricts nontrivially to $H_T^*(G^U)$.
For $p\in G^S$ we observe that $\ker (H^2(BT)\rightarrow H^2(BT_p))$ has to be contained in $C$. Also by Lemma \ref{lem:subgroupcorrespondence} it follows that $U\subset T_p$ and hence $p\in G^U$.
\end{proof}

\begin{thm}\label{thm:mainintegerGKM}
If $H^*_T(\Gamma,\alpha)$ is a free $R$-module and $H^*(\Gamma,\alpha;\QQ)$ satisfies ($6$-dimensional) Poincaré duality, then $X$ is an integer GKM manifold.
\end{thm}

\begin{proof}
This follows from Theorem \ref{thm:rationalGKM} and Proposition \ref{prop:freeimpliestorsioncondition}.
\end{proof}

\section{Rigidity}\label{sec:rigidity}
In this section we prove the rigidity statement that a simply-connected $6$-dimensional integer $T^2$-GKM manifold such that 
\begin{enumerate}
\item[(a)] every closed stratum of a finite isotropy group contains a $T^2$-fixed point. 
\item[(b)] there exists a $T^2$-fixed point in whose vicinity there occur at most two distinct finite nontrivial isotropy groups
\end{enumerate}
is determined, up to equivariant homeomorphism, by their GKM graph, see Theorem \ref{thm:rigidity}.
 We would like to remark that in the special case of connected isotropy groups it can also be proven using the results and ideas of \cite{Ayzenberg}. Concretely, one has to observe that an isomorphism of the GKM graphs of two integer GKM manifolds as above extends to a homeomorphism of the orbit spaces (which are necessarily $S^4$ by \cite[Corollary 1.3]{AyzenbergMasuda}) intertwining the GKM graphs, and apply \cite[Theorem 5.5]{Ayzenberg} or \cite[Corollary 5.8]{Ayzenberg}. Our proof proceeds along the lines of our realization procedure: we first prove that there is an equivariant homeomorphism between any two thickenings of the one-skeleta, extend this equivariant homeomorphism to a thickening of the full nonregular set and then conclude by extending it to their complements. 

\begin{rem}\label{rem:condition(b)} Condition (b) above has an equivalent formulation in terms of the GKM graph. It is equivalent to the fact that there are two adjacent edges whose weights are a $\mathbb{Z}$-basis. Indeed the finite isotropies occurring in any neighbourhood of a fixed point are precisely the intersections $\ker\alpha_1\cap\ker\alpha_2$, $\ker\alpha_2\cap\ker\alpha_3$, and $\ker\alpha_3\cap\ker\alpha_1$, where the $\alpha_i$ are the weights at that fixed point. Clearly if two of the $\alpha_i$ form a $\mathbb{Z}$-basis then the corresponding intersection is empty. Conversely if the three intersections are non-empty, then they must be pairwise distinct for otherwise the triple intersection $\ker\alpha_1\cap\ker\alpha_2\cap\ker\alpha_3$ is non-empty and the action would not be effective.
\end{rem}

\subsection{Neighbourhoods of the one-skeleton}

\begin{lem}\label{lem:equivisotopies} Consider the $T$-space $X=S^1_{\alpha_1} \times D^4_{(\alpha_2,\alpha_3)}$, where the $\alpha_i\colon T\rightarrow S^1$ are pairwise linearly independent. Let $Y$ be the space of self homeomorphisms of $X$ which are $T$-equivariant, respect the $D^4$-fibers over $S^1$, and are $\mathbb{R}$-linear on the fibers. Then $Y$ is connected.
\end{lem}

\begin{proof}
Let $f\in Y$. Then there is a linear transformation of $\mathbb{R}^4$ preserving the unit disk, i.e. $A\in \mathrm{O}(4)$, as well as $t\in S^1$ such that for any $p\in D^4$ we have $f(1,p)=(t, Ap)$. Let $S\subset T$ denote the kernel of $\alpha_1$. This is a one-dimensional, possibly disconnected subgroup of $T$ Then for any $s\in S$, $p\in D^4$ the $T$-equivariance implies $Asp=sAp$. Denoting by $\alpha(S)$ the subgroup that is the image of $S$ under $(\alpha_2,\alpha_3)\colon T\rightarrow S^1\times S^1\subset \mathrm{U}(2)\subset \mathrm{O}(4)$ this means that $A$ lies in the centraliser $C(\alpha(S))\subset \mathrm{O}(4)$ of $\alpha(S)$. Conversely every $(t,A)\in S^1\times C(\alpha(S))$ defines a unique $f\in Y$ by setting $f(s,p)=z\cdot (t(z^{-1}s),A(z^{-1}p))$ for some $z\in T$ with $z\cdot 1=s$. Thus $Y\cong S^1\times C(\alpha(S))$.

It remains to see that $C(\alpha(S))$ is connected. Let $S_0$ be the identity component of $S$.
Note that neither $\alpha_2$ nor $\alpha_3$ vanish on $S_0$ since otherwise they
would be linearly dependent from $\alpha_1$. If $\alpha_2$ and $\alpha_3$ give
rise to isomorphic $2$-dimensional $S_0$-representations (i.e.\
$\alpha_2|_{S_0}=\pm \alpha_3|_{S_0}$), then $C(\alpha(S_0))$ is isomorphic to
$\mathrm{U}(2)$: indeed in case the signs of the weights agree, $\alpha(S_0)$ is
the diagonal circle and we have that it is $\mathrm{U}(2)$. In case they do not
we obtain a conjugate of $\mathrm{U}(2)$ in $\mathrm{O}(4)$. As $S$ is generated
by $S_0$ and a cyclic subgroup, the centralizer $C(\alpha(S))$ is isomorphic to
the centralizer of a single element in $\mathrm{U}(2)$, in particular connected.
If the subrepresentations defined by the $\alpha_i|_{S_0}$
are not isomorphic then by Schur's Lemma they are preserved by $A$.
Automorphisms of the irreducible representations are isomorphic to $S^1$, hence
$C(\alpha(S_0))= T^2$ in this case. Thus also
$C(\alpha(S))=T^2$.
\end{proof}

\begin{prop}\label{prop:thickening} Let $X$ be a smooth orientable $T$-manifold whose one-skeleton is the $T$-space encoded in a GKM graph $(\Gamma,\alpha)$. Then there is a smooth codimension $0$ submanifold $\tilde{M}$ with boundary which contains the one-skeleton of the action and which is equivariantly homeomorphic to the manifold $M$ constructed in Section \ref{sec:realization}. Moreover if $X'$ is another $T$-manifold as above, then we can choose an analogous neighbourhood $\tilde{M}'\subset X'$ of the one-skeleton such that $\tilde{M}'$ is equivariantly diffeomorphic to $\tilde{M}$.
\end{prop}

\begin{proof}
We choose a $T$-invariant Riemannian metric $g$ on $X$ such that for any fixed
point $p$ the exponential map $T_p X\rightarrow X$ is an isometry on a small
neighbourhood $U_p$ of the origin, i.e.\ $g$ is flat around the fixed points.
For a fixed point $p$ corresponding to some vertex $v\in V(\Gamma)$, the
$T$-representation $T_p X$ decomposes as the sum of irreducible $2$-dimensional
subrepresentations $V_i$, $i=1,2,3$. Now choose $\varepsilon$ small enough such
that \[\tilde{B}_v:=\{(u_1,u_2,u_3)\in V_1\oplus V_2\oplus
V_3~|~\|u_1\|,\|u_2\|,\|u_3\|\leq \varepsilon\}\]
is contained in $U_p$. In particular $\exp(\tilde{B_v})$ is a neighbourhood of $p$ isometric to $\tilde{B}_v$.

For an edge $e\in E(\Gamma)$ connecting vertices $v,w$ there is a corresponding
invariant $2$-sphere $S_e$. We consider the normal bundle $NS_e$. For
$\delta>0$, let $D_\delta S_e$ be the disk bundle obtained by fiberwise
restriction to closed $\delta$-balls around the $0$-section. For $\delta$ sufficiently small the exponential map is a diffeomorphism from $D_\delta S_e$ onto a $T$-invariant tubular neighbourhood of $S_e$. If we assume $\delta$ to be small compared to $\varepsilon$, then by choice of the metric $\exp(\tilde{B}_v)\cap \exp (D_\delta S_e)$ is $\exp(D_\delta B_\varepsilon^2(p))$, where $B_\varepsilon^2(p)\subset S_e$ is the $2$-dimensional closed $\varepsilon$-ball $B_\varepsilon^2(p)\subset S_e$ around the fixed point $p\in S_e$ corresponding to $v$ and $D_\delta  B_\varepsilon^2 (p)$ denotes the restriction of the disk bundle to this subset (we will use the analogous notation for restrictions of the disk bundle to any subset).
We use analogous terminology around the fixed point $q$ corresponding to $w$. We define $\tilde{K}_e$ to be the closure of \[D_\delta S_e\backslash (D_\delta B^2_\varepsilon(p)\cup D_\delta B^2_\varepsilon(q)).\]
Now $\exp(\tilde{B}_v)\cap \exp(\tilde{K}_e)$ is $\exp(D_\delta (\partial
B^2_\varepsilon(p))$ and the analogous observation holds around $q$. We do this
for all $v\in V(\Gamma)$, $e\in E(\Gamma)$ and choose $\varepsilon,\delta$
sufficiently small such that intersections of the above type (i.e. vertex and
adjacent edge) are the only nontrivial intersections among all of the
$\exp(\tilde{B}_v)$ and $\exp(\tilde{K}_e)$. Furthermore we observe that the
union of all of the latter give a topological manifold with boundary
$\tilde{M}\subset X$ which is a neighbourhood of the one-skeleton. Note that at this point $\tilde{M}$ is not a smooth submanifold. We will prove first that
$\tilde M$ is equivariantly homeomorphic to $M$ and address questions of smoothness in the end.

We now choose $T$-equivariant homeomorphisms $f\colon B_v\rightarrow \tilde{B}_v$ where we use the terminology from Section \ref{sec:realization}. Let again $v$ be a vertex belonging to an edge $e$ and $p$ be the fixed point associated to $v$. We identify $T_p X\cong \mathbb{C}^3$ isometrically such that each of the $\mathbb{C}$-factors is an irreducible subrepresentation. Then the choice of $f$ as above comes down to choosing an identification \[\mathbb{C}^3\supset \mathbb{R}^3\supset D^3\cap (I^3)\rightarrow [0,\varepsilon]^3 \subset \mathbb{R}^3\subset \mathbb{C}^3\] and extending $T^3$-equivariantly w.r.t.\ the standard $T^3$-action.

Recall that $K_e$ gets glued to $B_v$ along a map $\psi_{e,v}\colon S^1\times D^4\rightarrow \partial B_v$. The definition of $\psi_{e,v}$ involved the choice of a small constant $\delta$ which we assume to be the same as above. From the definition of $\psi_{e,v}$ we see that the homeomorphism $f$ can be adjusted via suitable real stretching factors such that $f\circ \psi_{e,v}$ is --up to switching of coordinates-- the map
\[S^1\times D^4\rightarrow \tilde{B}_v,\quad (t,z_1,z_2)\mapsto (\varepsilon t,
\delta \cdot c_{\varepsilon_2}(z_1)\cdot t^{k_2},\delta\cdot
c_{\varepsilon_3}(z_2)\cdot t^{k_3}).\]
In particular composing this with the exponential map the image agrees with $\exp(D_\delta(\partial B^2_\varepsilon(p))$. Hence there is a unique map $g$ making the diagram
\[\xymatrix{
S^1\times D^4\ar[d]^{\psi_{e,v}}\ar[r]^g & D_\delta(\partial B^2_\varepsilon(p))\ar[d]^{\exp}\\
B_v\ar[r]^{\exp\circ f}& \exp(\tilde{B}_v)
}\]
commute. It is $T$-equivariant, respects the respective fibers over $S^1$ and $\partial B^2_\varepsilon$ and is in fact $\mathbb{R}$-linear on these fibers. The same situation occurs around the other end of the edge near the fixed point $q$ associated to the other vertex $w$ connected to $e$. Thus we obtain a map $g\colon \partial K_e\rightarrow D_\delta(\partial B_\varepsilon^2(p)\cup \partial B_\varepsilon^2(q))$. In order to obtain the desired equivariant homeomorphism $M\rightarrow \tilde{M}$ it suffices to extend $g$ to an equivariant homeomorphism $K_e\rightarrow \tilde{K}_e$.

Observe that $\tilde{K}_e$ is an equivariant disk bundle over $S_e\backslash
(B^2_\varepsilon(p)\cup B^2_\varepsilon(q))$. The latter is equivariantly homeomorphic
to $\partial B_\varepsilon^2(p)\times I$ which in turn deformation retracts
equivariantly onto the subspace $\partial B^2_\varepsilon(p)\times\{0\}$. It
follows that as a linear disk bundle $\tilde{K}_e$ is equivalent to $D_\delta
(\partial B_\varepsilon^2(p))\times I$. We can now identify $D_\delta(\partial
B^2_\varepsilon(p))\cong S^1\times D^4$, e.g.\ using the isomorphism $g$ of linear
disk bundles above. The extension problem for $g$ above then reduces to the question whether we can extend a $T$-equivariant linear bundle automorphism (not necessarily covering the identity) of $S^1\times D^4\times \partial I$ to all of $S^1\times D^4\times I$. This follows from Lemma \ref{lem:equivisotopies}.

Finally, regarding the smoothness, we adjust $\tilde{M}$ by modifying $\tilde{B}_v$ within $T_p X\cong\mathbb{C}^3$ by suitable real stretching. Using the local flatness of the metric in the neighbourhood $\exp(U_p)$ lets us achieve that the new $\exp (\tilde{B}_v)$ connects smoothly to $\exp(\tilde{K}_e)$ and hence $\tilde{M}$ is a smooth submanifold. Given another manifold $X'$ with the same one-skeleton we construct $\tilde{M}'$ in completely analogous fashion as a union of $\exp(\tilde{B}_v)$ and $\exp(\tilde{K}_v)$. We obtain homeomorphism $f'\colon B_v\rightarrow \tilde{B_v'}$ in the same way as we defined $f$ above, where due to the modifications made to the $\tilde{B}_v$, the role of $[0,\epsilon]^3$ is replaced by a suitable smooth submanifold of $\mathbb{R}^3$ with corners. Then $f'\circ f^{-1}$ is smooth. The resulting map on the boundary of $\exp(\tilde{K}_e)$ is smooth and the extension can be carried out smoothly. Thus $\tilde{M}$ is diffeomorphic to $\tilde{M}'$.
\end{proof}

\begin{lem}\label{lem:isotropymanifolds}
Let $X$ be a $6$-dimensional integer $T$-GKM manifold such that any finite
isotropy stratum contains a $T$-fixed point. Then for
any finite group $H\subset T$ the fixed point set $X^H$ is a disjoint union of
(non-effective) $4$-dimensional quasitoric $T$-manifolds, invariant $2$-spheres and isolated points.
\end{lem}
\begin{proof}
Every component $N$ of $X^H$ where the principal isotropy has positive dimension
is either a fixed point or an invariant $2$-sphere. In case the principal
isotropy on $N$ is discrete then by assumption the finite
isotropy stratum $N$ contains a $T$-fixed point. It follows from our global
effectivity assumption that the discrete principal isotropy on $N$ is the
intersection of two weights at this fixed point. Hence $N$ is $4$-dimensional.
In particular $N$ is also a component of $X^{H'}$ for any nontrivial subgroup
$H'\subset H$. Hence it suffices to prove the case where $H\cong \mathbb{Z}_p$
for a prime $p$.

As a first step we show that the $T$-action on $X^H$ is equivariantly formal
over $\mathbb{Q}$. In order to show this it suffices to prove that
$\dim_\mathbb{Q} H^*(X^H;\mathbb{Q})=\dim_\mathbb{Q} H^*(X^T;\mathbb{Q})$, where
we note that  $\dim_\mathbb{Q} H^*(X^H;\mathbb{Q})\geq\dim_\mathbb{Q}
H^*(X^T;\mathbb{Q})$ is automatic (see e.g.\ \cite[Proposition 9.6]{GZsurvey}).
Now since the Serre spectral sequence of $X_T\rightarrow BT$ collapses
with $\mathbb{Z}$-coefficients and hence also with
$\mathbb{Z}_p$-coefficients, the same holds for the Borel fibration of the
restricted $H$-action which is the pullback of the above map along
$BH\rightarrow BT$. Hence \[\dim_{\ZZ_p} H^*(X;\mathbb{Z}_p)=
\mathrm{rk}_{H^*(BH;\ZZ_p)} H_H^*(X;\ZZ_p)= \mathrm{rk}_{H^*(BH;\ZZ_p)}
H_H^*(X^H;\ZZ_p)=\dim_{\ZZ_p} H^*(X^H;\mathbb{Z}_p)\]
where the first equality uses the aforementioned collapsing of $X_H\rightarrow BH$ and the second equality follows from Borel localization (cf.\ \cite[Theorem 3.2.6]{AlldayPuppe}). Furthermore the universal coefficient theorem implies $\dim_\QQ H^*(X^H;\QQ)\leq\dim_{\ZZ_p} H^*(X^H;\ZZ_p)$. It follows that
\[\dim_\QQ H^*(X^H;\QQ)\leq \dim_{\ZZ_p} H^*(X;\mathbb{Z}_p)=\dim_\QQ H^*(X;\QQ)=\dim_\QQ H^*(X^T;\QQ).\]
Thus $T$ indeed acts rationally equivariantly formal on $X^H$. Let $N$ be a
$4$-dimensional component of $X^H$. Then the $T$-action on $N$ is rationally
equivariantly formal as well. Since any closed stratum of an
occurring finite isotropy group is 4-dimensional, there is only a single
occurring finite isotropy type $U\supset H$ on $N$. In particular as a
$T/U$-manifold $N$ is rationally equivariantly formal and locally standard of
complexity $0$ (cf.\ \cite[Definition 7.1.1]{ToricTopology}).
The locally standard condition implies that $N/T$ is a surface $\Sigma$ with
boundary $\partial \Sigma$. The part of $N$ lying over $\partial \Sigma$
consists of circles of $2$-spheres in the one-skeleton. The remaining orbits are
regular. It remains to show that $\Sigma$ is a disk (and in
particular that it has only one boundary component). But as $T/U$ acts
rationally equivariantly formally we have $H_{T/U}^{odd}(N;\mathbb{Q})=0$ and a
Mayer-Vietoris argument shows that
$H^1(\Sigma;\mathbb{Q})=H^1_{T/U}(N;\mathbb{Q})=0$. The disc is the only surface
with non-trivial boundary satisfying this condition.
\end{proof}

\begin{lem}\label{lem:rededges}
Let $(\Gamma,\alpha)$ be a $3$-valent GKM graph, $\nabla$ a compatible connection, and $c$ a connection path. Then we can order the connection paths $c_1,\ldots,c_k$ with $c=c_k$ such that there are $e_1,\ldots,e_{k-1}\in E(\Gamma)$ with the property that for $1\leq i\leq k-1$ the edge $e_i$ is contained in $c_i$ but not in $c_j$ for $j<i$.
\end{lem}
\begin{proof}
We observe first that any edge is contained in precisely two connection paths. We consider the dual graph $\Gamma^*$ whose vertices are the connection paths and whose edges are given by $E(\Gamma)$ where an edge $e$ connects the two connection paths it is contained in. Now choose a maximal subtree $\cT\subset \Gamma^*$ and orient the edges such that they point away from $c$. We order the connection paths as $c_1,\ldots,c_k$ in a way such that moving along an oriented edge in $T$ reduces the index. Now for every $c_i$, $i\leq k-1$, choose $e_i$ to be the unique edge pointing into $c_i$.
\end{proof}

\begin{prop}\label{prop:twohandlethickening}
Let $X_1$ and $X_2$ be two $6$-dimensional integer $T$-GKM manifolds with the
property that any finite isotropy stratum contains a $T$-fixed point. We assume
that they have the same GKM graph $(\Gamma,\alpha)$. Assume furthermore that
there are two adjacent weights in $\Gamma$ which form a basis. Then there are
codimension $0$ $T$-invariant submanifolds $Y_1\subset X_1$, $Y_2\subset X_2$
which contain all singular orbits and such that $Y_1$ and $Y_2$ are
equivariantly homeomorphic.
\end{prop}

\begin{proof}
We start by considering smooth neighbourhoods of the one-skeleta $M_1\subset X_1$ and $M_2\subset X_2$ as in Proposition \ref{prop:thickening} as well as a diffeomorphism $\varphi\colon M_1\rightarrow M_2$. Now the $M_i$ already contain all singular orbits of positive dimensional isotropy. Thus we must extend the neighbourhoods to the orbits of nontrivial finite isotropy. By Lemma \ref{lem:isotropymanifolds} these are exactly the regular strata of (noneffective) quasitoric $T$-manifolds embedded into the $X_i$ whose individual $1$-skeleta are circles of two-spheres corresponding to a connection path.

We consider a connection path $c$ along which the intersection of the kernels of the weights is some nontrivial finite group $H\subset T$ and let $Q_1\subset X_1$ be the associated quasitoric $4$-manifold with principal isotropy $T/H$. Then $\varphi$ maps a neighbourhood of the $1$-skeleton of $Q_1$ onto a neighbourhood of a quasitoric $4$-manifold $Q_2\subset X_2$. Abstractly we know that $Q_1$ and $Q_2$ are equivariantly diffeomorphic. We wish to extend $\varphi$ to an equivariant diffeomorphism $Q_1\cong Q_2$. We note that $B_i:=Q_i\backslash \mathring{M_i}$ is equivariantly homeomorphic to ${D^2}\times T/H$. Thus we have to extend a given smooth equivariant diffeomorphism $h\colon S^1\times T/H\rightarrow S^1\times T/H$ to an equivariant diffeomorphism of $D^2\times T/H$. We may assume $h$ is of the form $(p,t)\mapsto (p,\eta(p)t)$ where $\eta\colon S^1\rightarrow T/H$. Then $h$ extends if and only if $\eta$ is nullhomotopic. To achieve this we may modify $\varphi\colon  M_1\cong M_2$ as follows. Choose an edge $e$ that is part of the connection path of $Q_1$ (the choice of $e$ will be elaborated on below, for now it is arbitrary). Then there is a building block associated to $e$ in $M_1\cong M$ (cf.\ Proposition \ref{prop:thickening}) which is diffeomorphic to $K_e= I\times S^1_{\alpha_1}\times D^4_{\alpha_2,\alpha_3}$ where $\alpha_1,\alpha_2$ are weights belonging to edges in $c$. The intersection of $\partial B_1\cong S^1\times T/H\subset Q_1$ with this $K_e$ is precisely
$I\times S^1_{\alpha_1}\times S^1_{\alpha_2}\times \{0\}\subset K_e$. Now we
precompose $\varphi$ with the automorphism of $M_1$ which is the identity
outside $K_e$ and which multiplies the $\alpha_1$ and $\alpha_2$ component of
$K_e$ with the inverse of $\eta$ parametrized over the $I$-component. Note that
this is well defined as $\ker\alpha_1\cap\ker\alpha_2=H$ and the
$\alpha_3$-component remains fixed. Hence we may assume that $\varphi$ extends
over $B_1$ to an equivariant diffeomorphism $Q_1\cong Q_2$.

Next we wish to extend the map between normal bundles $NQ_i$. Note that
$\varphi$ induces a map between the normal bundles $N(Q_1\backslash
B_1)\rightarrow N(Q_2\backslash B_2)$. For convenience we may assume $\varphi$
extends to a small collar around and $M_1$ and hence the map between normal
bundles $NQ_1|_{\partial B_1}\rightarrow NQ_2|_{\partial B_2}$ is defined. We
observe that the inclusion of $B_1$ is equivariantly homotopic to the map onto a
single $T/H$ orbit which we may take inside $K_e$. Hence $NB_1$ is equivariantly
isomorphic to the diagonal action on $B_1\times \mathbb{C}_{\alpha_3}$.
Similarly $NB_2=B_2\times \mathbb{C}_{\alpha_3}$. Identifying $B_1=D^2\times
T/H$ and identifying this further with $B_2$ via $\varphi$, the question of
extending $\varphi_*\colon NB_1|_{\partial B_1}\rightarrow NB_2|_{\partial B_2}$
to the interior is equivalent to extending an automorphism of $S^1\times
T/H\times \mathbb{C}_{\alpha_3}$ to one of $D^2\times
T/H\times\mathbb{C}_{\alpha_3}$, where automorphisms are supposed to cover the
identity of $D^2\times T/H$. Choosing compatible orientations and Riemannian
metrics on the $M_i$ we may furthermore assume that the fiberwise linear
transformations of $\CC_{\alpha_3}$ lie in $\mathrm{U}(1)$. Hence, using the
equivariance, such an automorphism of $S^1\times T/H\times\mathbb{C}_{\alpha_3}$
is of the form $(p,[t],v)\mapsto (p,[t],\eta(p)v)$ where $\eta\colon
S^1\rightarrow \mathrm{U}(1)$. The map $\varphi_*$ extends to $NB_1$ if and only
if $\eta$ is nullhomotopic. This can be achieved similarly to before by
modifying $\varphi$ by precomposing with an equivariant automorphism of $M_1$
which is supported on $K_e=I\times S^1_{\alpha_1}\times
D^4_{\alpha_2,\alpha_3}$. Concretely set $(t,v,w,z)\mapsto
(t,v,w,\eta(t)^{-1}z)$ where we interpret $\eta$ as a loop. Note that this does
not change $\varphi$ on $Q_1$. For this modified $\varphi$ the map of normal
bundles extends.

We now choose Riemannian metrics on $X_1$ and $X_2$ such that $\varphi$ is an
isometry on $M_1$. A small closed neighbourhood of the normal bundles of the
$Q_i$ maps to a closed tubular neighbourhood $T_i$ around the $Q_i$. The map
$\varphi_*$ defines an extension of $\varphi$ to a homeomorphism $M_1\cup
T_1\rightarrow M_2\cup T_2$. With respect to the right choice
of metric and neighbourhood the thickening $M_i\cup T_i$ of the one-skeleton and
$Q_i$ will be a topological manifold. To be precise we set $S=Q_i\cap\partial
M_i$ and decompose $T{X_i}|_S=TS\oplus E\oplus F$ where $TS\oplus E=T\partial
M_i|_S$ and $TQ_i=TS\oplus F$. Now choose a $T$-invariant Riemannian metric $g$
in a tube around $S$ such that the above decomposition is orthogonal. We want to
ensure geodesics in $S$ in the direction of $E$ stay in $\partial M_i$. To do
this we use an argument from \cite[Lemma 6.8]{MilnorhCobord}: we choose a
$T$-equivariant involution $A$ around $\partial M_i$ (locally in a tube, send
every vector to its negative) whose fixed point set is $\partial M_i$ and then
replace $g$ by $g'=g+A^*g$. Then extend $g'$ to all of $X$ without modifying it
near $S$. With respect to this metric $\partial M_i$ is totally geodesic near
$S$ and $T_i\cap \partial M_i$ is the image of a neighbourhood of the
$0$-section of $E=NQ_1|_S$. In particular $M_i\cup T_i$ is equivariantly
homeomorphic to a gluing $M_i\cup B_i\times D^2_{\alpha_3}$ along an equivariant
embedding $\partial B_i\times D^2_{\alpha_3}\rightarrow \partial B_i$.

To finish the proof we do this inductively for all connection paths. The only potential problem in the procedure is that above we needed to modify $\varphi$ on $K_e$, where $e$ was an arbitrarily chosen edge in the connection path $c$ belonging to $Q_1$. If however $e$ is not part of a previously considered connection path these modifications can be carried out analogously. In order to ensure this, recall that by assumption there is a connection path $c$ along which two consecutive weights $\alpha_1,\alpha_2$ satisfy $\ker\alpha_1\cap \ker\alpha_2=\{1\}$. In particular there is no quasitoric $4$-manifold of nontrivial isotropy which belongs to $c$. Now choose an order $c_1,\ldots,c_k$ of the connection paths and edges $e_i$ in $c_i$, $i\leq k-1$, as in Lemma \ref{lem:rededges}, where $c=c_k$. When constructing an extension of $\varphi$ onto a thickening of the quasitoric manifold belonging to $c_i$ then we choose the edge $e_i$ in order to carry out the modifications of $\varphi$ on $M_1$. In this way we extend $\varphi$ to a thickening containing all singular orbits.
\end{proof}

\begin{rem}\label{rem:shapeofthickening} For later reference
	we note that in the above proof $B_i\times D^2_{\alpha_3}$ is equivariantly
homeomorphic to the equivariant $2$-handle $D_c$ from Section
\ref{sec:2-handles}. In particular the manifolds $Y_i$ above arise from the
manifolds $M\cong M_i$ constructed in Proposition \ref{prop:thickening} (with
$M$ as in Section \ref{sec:realization}) by gluing in the equivariant
$2$-handles $D_c$ along all connection paths $c$ along which the isotropy (i.e.\
the intersection of the kernels of two consecutive weights) is nontrivial. We
have not verified that the attaching maps agree with the ones used in Section
\ref{sec:realization} (which is the case) but it turns out to not be necessary.
We will however need the following fact about the locally defined $T^3$-action
on the $D_c$.

Let $e$ be an edge in a connection path of nontrivial isotropy. Recall from Lemma \ref{lem:T3action} that the $T^3$-action on $K_e$ extends to $M_{c}$ where the latter is the union of all $B_\bullet$, $K_\bullet$ through which $c$ runs. We note that on the manifold $M\cong M_i\subset X_i$ as constructed in Propostition \ref{prop:thickening} the locally defined $T^3$-action is smooth. Hence in the notation used in the construction of the $Y_i$ in the above Proposition \ref{prop:twohandlethickening}, $T^3$ acts on the normal bundle $NQ_i|_{Q_i\cap M_i}$ (where $Q_i$ is the quasitoric fixed point manifold associated to $c$). It follows from Lemma \ref{lem:extendactionon2handle} below that this action extends over all of $NQ_i$. We may choose the metric $g$ in the definition of the tube $T_i$ around $Q_i$
to be $T^3$-invariant on $M_c$. Then the $T^3$-action on $M_{c}$ extends to $T_i$. In particular when viewing the $Y_i$ as a gluing of the $M_i$ and two-handles $D_c$ then the locally defined $T^3$-action extends to the $D_c$.
\end{rem}

\begin{lem}\label{lem:extendactionon2handle}
Let $\alpha\colon T\rightarrow S^1$ be a weight, $H\subset T$ discrete and consider the equivariant $\mathbb{C}$-bundle $D^2\times T/H\times \mathbb{C}_\alpha$ with the diagonal $T$-action. Furthermore assume there is an effective $S^1$-action on $S^1\times T/H\times\mathbb{C}_\alpha$ which commutes with the $T$-action, acts linearly on the $\mathbb{C}$-fibers and covers the identity of $S^1\times T/H$. Then this $S^1$-action extends in this way to all of $D^2\times T/H\times \mathbb{C}_\alpha$.
\end{lem}
\begin{proof}
A $T$-equivariant linear bundle automorphism
over $\{p\}\times T/H$ is of the form $(p,[t],v)\mapsto
(p,[t],tAt^{-1} v)$ where $A\in \mathrm{GL}(2,\mathbb{R})$ is a linear map
commuting with the $H$-action on $\mathbb{C}_{\alpha}$. Let $G\subset
\mathrm{GL}(2,\mathbb{R})$ denote the subgroup of these matrices (either
$G=\mathrm{GL}(2,\mathbb{R})$ or $G=\mathbb{R}^+\cdot \mathrm{U}(1)$). In
particular the action $\{p\}\times T/H\times
\mathbb{C}_{\alpha}$ is determined by the action on
$\{p\}\times\{[1]\} \times \mathbb{C}_{\alpha}$ which in turn is given by an
injective homomorphism $S^1\rightarrow G$. Since maximal compact subgroups are
conjugate, all circles in $G$ are conjugate to the standard circle
${\mathrm{U}}(1)$. We thus find a map $\eta\colon S^1\rightarrow G$ such that
$z\cdot (p,[1],v)=(p,[1],\eta(p) z\eta(p)^{-1}v)$ where multiplication with $z$
on the right is to be understood as the standard linear transformation on
$\mathbb{C}$. Concretely, for every $p$ we have a circle in
$G$, with Lie algebra a line in the Lie algebra ${\mathfrak{g}}$ of $G$. This
defines a curve that lies in a single orbit of the $G$-action by conjugation on
the space of lines in $\mathfrak{g}$. Hence, the map $\eta$ can be chosen
continuous everywhere except at $1\in S^1$, such that the difference when
approaching $1$ from both sides lies in the centralizer
$\mathrm{GL}(1,\CC) = {\mathbb{R}}^+\cdot {\mathrm{U}}(1)$
of ${\mathrm{U}}(1)$ in $\mathrm{GL}(2,\mathbb{R})$. Thus, by multiplying $\eta$
with a suitable curve in the connected group ${\mathbb{R}}^+\cdot
{\mathrm{U}}(1)$, we can assume that it is continuous. Then we may multiply
$\eta$ from the right with a further curve in ${\mathrm{U}}(1)$ in order to make
it nullhomotopic, at the same time keeping valid the above description of the
action in terms of $\eta$. Now this description allows to extend it to all of
$D^2\times T/H\times \mathbb{C}_{\alpha}$.
\end{proof}

We define boundary connected sum of two oriented topological manifolds $M,N$ with boundary. We choose an embedded disk in each boundary, one orientation preserving and the other one orientation reversing, and such that the boundary spheres are bicollared (cf.\ Definition \ref{defn:bicollared}). Now we glue $M$ and $N$ along these embeddings. We denote the resulting space by $M\#^\flat N$. It is a manifold with boundary with the latter given by the usual connected sum $\partial M\#\partial N$. The construction is reviewed in some more detail in Appendix \ref{sec:topologicalshiat}. The reason we work in the topological category is that our orbit spaces do not naturally come with a smooth structure (even if they are usually topological manifolds). The homeomorphism type of $M\#^\flat N$ does not depend on the chosen disks (cf.\ Definition \ref{defn:bicollared}, Corollary \ref{prop:consumunique}). In particular it makes sense to write $\#^\flat_p M$ for the $p$-fold boundary connected sum.

\begin{lem}\label{lem:sphere-twist}
Let $X$ be a $6$-dimensional integer GKM $T$-manifold
with the property that any finite isotropy stratum contains
a $T$-fixed point, with GKM graph $(\Gamma,\alpha)$. We
assume that there are two adjacent weights in $\Gamma$ which form a basis. Let
$Y\subset X$ be the codimension $0$ submanifold constructed in the proof of
Proposition \ref{prop:twohandlethickening}. Then $Y/T\cong \#_p^{\flat} D^3
\times S^1$. With respect to this identification for every summand of the
boundary $\#_p S^2 \times S^1$ there is a sphere of the form $S^2\times\{p\}$
such that a twist around that sphere is induced by a $T$-equivariant
automorphism of $Y$.
\end{lem}

The sphere twists in the above lemma are defined as follows (cf.~\cite[p.\ 79]{Laudenbach}): Let $W^3$ be a closed and orientable 3-manifold and $S \subset W^3$ an
embedded $2$-sphere. Choose a tubular neighbourhood of $S$ in $W^3$ which is
diffeomorphic to $S^2 \times I$. Let $\gamma \colon I \to \textrm{SO}(3)$ be
a loop which represents the generator of $\pi_1(\rm{SO}(3))\cong\mathbb{Z}_2$
with basepoint the identity. It is up to homotopy given by rotation around an axis. The homeomorphism 
\[
	I\times S^2 \to I\times S^2 , \quad (t,p) \mapsto (t,\gamma(t) \cdot p)
\]
clearly extends by the identity to a homeomorphism of $W^3$, which we denote by
$H_S(\gamma)$. It follows that its isotopy class in $\pi_0(\mathrm{Homeo(W))}$ is well-defined
with respect to the homotopy class of $\gamma$ as well as with respect to the
homotopy class of
$S$ in $W^3$  (cf.~\cite{MR314054}).

\begin{proof}
 By construction (see Remark
\ref{rem:shapeofthickening}), $Y$ is equivariantly homeomorphic to the manifold
$M$ from Section \ref{sec:realization} with equivariant $2$-handles $D_c$
attached along all connection paths $c$ with nontrivial isotropy.

Let $c_1,\ldots,c_l$ denote those connection paths such that
the kernels of two adjacent weights have nontrivial intersection. By assumption
there is at least one connection path which is not of this form, so if we choose
one of those as the connection path with highest index in Lemma
\ref{lem:rededges}, it follows that we can choose edges $e_i$ in $c_i$,
$i=1,\ldots,l$, such that $e_i$ is not in $c_j$ for $j<i$. Hence $e_i$ is part
of a closed loop after removing $e_j$ for all $j>i$. It follows that there is a
maximal tree $\cT\subset \Gamma$ which contains none of the $e_i$.

We start by considering the union over the $B_v/T$ and $K_e/T$ in $\cT$. It
follows from Lemma \ref{lem:orbitspaceisS1bundle} and the preceding discussion
that for $v\in V(\Gamma)$ and $e\in E(\Gamma)$ we have $B_v/T\cong D^4$ and
$K_e/T\cong I\times D^3$. These get glued together at $3$-disks in their
boundary. One checks that all boundary spheres of attaching
disks are bicollared. By Corollary \ref{prop:consumunique}, the result of
gluing the $B_v/T$ and the $K_e/T$ only depends on the orientation of the
attaching maps. When considering the union over the tree
$\cT$ the resulting orbit space is homeomorphic to $D^4$ independently of the
orientation of the attaching maps. Now we glue the remaining edges $I\times
D^3$ to the boundary. This time the orientations of the attaching maps are
relevant, and there are in fact two possible results of an attaching of $K_e/T$
(only two since $K_e$ has an orientation reversing automorphism). One of them
results in a nonorientable boundary while the other one is the same as taking
the boundary connected sum with $D^3\times S^1$. However recall that we already
know that $\partial M/T$ is orientable by Proposition \ref{prop:orientability}.
Thus it has to be the latter which proves that $M/T$ is a boundary connected sum
of copies of $D^3 \times S^1$.
 To arrive at $Y$ we add the $D_{c_i}$ for $i=1,\ldots,l$,
	which cancels $l$ of the $D^3\times S^1$-factors. One way to see this is to
	instead of adding all $K_e$ in the previous step, we leave out
	$e_1,\ldots,e_l$ and then add $K_{e_i}\cup D_{c_i}$ inductively for
	$i=1,\ldots,l$. The union $K_{e_i}/T\cup D_{c_i}/T$ is a $4$-ball and it is
	attached along a $3$-ball in its boundary. Hence adding these does not change
	the homeomorphism type.

We now come to the statement about the sphere twists. Let
$M'$ denote the union of all $B_v$ and $K_e$ except for $e=e_1,\ldots,e_l$
(i.e., all those in the tree $\cT$, and possibly some more). Thus
$M'/T=\#_p^{\flat} D^3 \times S^1$ as above and $Y/T\cong M'/T$ as it arises by
gluing $4$-balls along $3$-balls in the boundary. Note first that the summands
in the connected sum $\partial M'/T=\#_p S^2 \times S^1$ correspond to edges
$K_e/T$, $e\neq e_1,\ldots, e_l$, glued after the completion of the tree.
In what follows we fix one such summand, i.e., one such edge
$e$. The corresponding sphere $S^2\times\{p\}$ as in the statement of the lemma
is given by the sphere $\{1/2\}\times S^2\subset I\times D^3\cong K_e/T$. Its
sphere twist (up to isotopy) is given by the automorphism which rotates the
$S^2$ factor around an axis once, parametrized over the first component of
$[1/2-\varepsilon,1/2+\varepsilon]\times S^2\subset K_e/T$ and is the identity
everywhere else. We prove first that this lifts to an
equivariant automorphism of $M'$.

Consider $T_e\subset T^3$ (as in Section \ref{sec:realization}) and a complementary circle $S$ such that $T_e\times S=T^3$. Recall that $K_e= I\times S^1\times D^4$ with the standard $T^3$-action on $S^1\times D^4$ and $T$ acts via the subtorus $T\cong T_e\subset T^3$. In Lemma \ref{lem:orbitspaceisS1bundle} we have written $K_e/T$ as $\overline{K}_e\times T^3/T_e$ with the $T^3/T_e$ fibers collapsed over $F_e$. Thus via the canonical map $S\cong T^3/T_e$ the circle $S$ acts on $K_e/T$ by rotating the $T^3/T_e$ factor at unit speed. Thus when identifying $K_e/T\cong I\times D^3$ we see that $S$ acts as standard rotation on the $D^3$-factor. Now let $\gamma\colon [1/2-\varepsilon,1/2+\varepsilon]\rightarrow S$ be a single loop around $S$ based at $1$. Then we define an automorphism $\varphi$ of $M'$ by sending $(t,p)\mapsto (t,\gamma(t)\cdot p)$ for $(t,p)\in [1/2-\varepsilon,1/2+\varepsilon]\times(S^1\times D^4)\subset K_e$ and taking $\varphi$ to be the identity everywhere else. The $S$-action commutes with the $T$-action hence $\varphi$ is indeed $T$-equivariant. By what we have argued before, it induces the desired sphere twist on $M'/T$.

It now suffices to argue that this equivariant automorphism
of $M'$ extends to automorphisms of all $D_{c_i}\cup K_{e_i}$ and thus to an
equivariant automorphism of $Y$, and that this extension is still isotopic to
the sphere twist. We proceed inductively. If $c_1$ does not run through $e$
then there is nothing to check. In case it does, the
$T^3$-action on $K_e$ extends along the connection path, i.e.\ to $M_{c_1}$
(cf. Lemma \ref{lem:T3action}), and to the tube $D_{c_1}$ as explained in Remark
\ref{rem:shapeofthickening}.

Recall the projection $K_e\rightarrow F_e$ to the compact
surface introduced in Section \ref{sec:realization}. The automorphism $\varphi$
on $M'$ can be understood (up to isotopy) as multiplying all points lying over
some $p\in F_e$ by some $\psi (p)\in T^3$ for some function $\psi\colon
F_e\rightarrow T^3$ which maps to the unit near the ends of $F_e$ intersecting
with the neighbouring $F_v$ for vertices $v$ adjacent to $e$. The map
$K_\bullet\rightarrow F_\bullet$ extends to a map $K_e\cup D_{c_1}\cup
K_{e_1}\rightarrow F_e\cup F_{c_1}\cup F_{e_1}$. Thus to extend $\varphi$ to
$D_{c_1}\cup K_{e_1}$ it suffices to extend the map $\psi$ to $F_e\cup
F_{c_1}\cup F_{e_1}$ such that it maps to the identity of $T^3$ near the parts
of $\partial F_{e_1}$ which connect to other $F_v$ for vertices $v$ adjacent to
$e,e_1$ as well as those parts of $\partial F_{c_1}$ which are not glued to
either $F_e$ or $F_{e_i}$. There is no obstruction to doing so: one can take the
support of the extended $\psi$ to be a (quadrangular) stripe going from  $F_e$
through the interior of $F_{c_i}$ to $F_{e_1}$ such that $\psi$ maps to the unit
on the two opposing sides of the quadrangle travelling through the interior of
$F_{c_1}$.

\begin{figure}[h]
	\centering
	\includegraphics{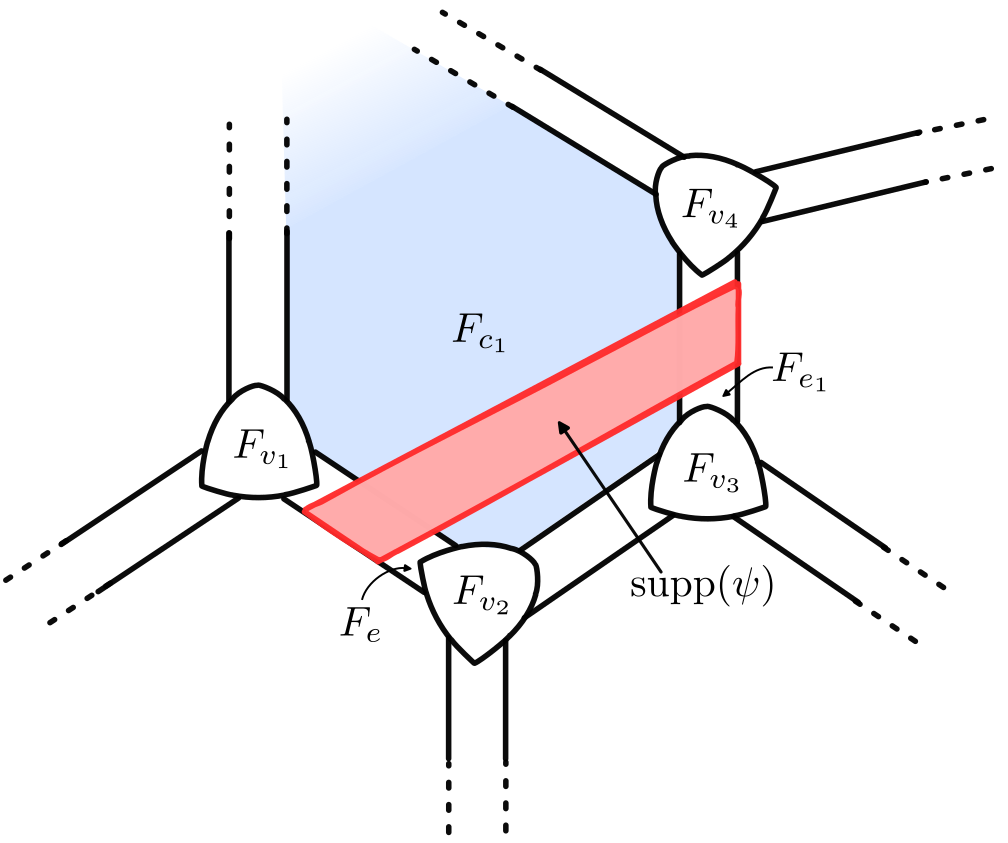}
\end{figure}

This solves the problem of extending $\varphi$ to $D_{c_1}\cup K_{e_1}$. We show that the extension is still isotopic to the sphere twist. There is a homeomorphism $h\colon M_1':= M'/T\cup (D_{c_1}\cup K_{e_1})/T\cong M'/T\cong \#_p^{\flat} D^3 \times S^1$ which absorbs the added $4$-ball $(D_{c_1}\cup K_{e_1})/T$ into neighbourhoods of their attaching $3$-balls in $\#_p^{\flat} D^3 \times S^1$. Then the automorphism $\overline{\varphi}$ of $\#_p^{\flat} D^3 \times S^1$ induced by the equivariant map $\varphi\colon M_1'\rightarrow M_1'$ will be isotopic to the one induced by $\varphi|_{M'}$ and thus to the desired sphere twist. To see this argue as follows: the attaching $D^3$ has a neighbourhood in $\#_p^{\flat} D^3 \times S^1$ homeomorphic to $D^3\times [-1,0]$ with the attaching $D^3$ corresponding to $D^3\times\{0\}$. The attached $(D_{c_1}\cup K_{e_1})/T$ can be identified with $D^3\times [0,1]$ with the attaching $D^3$ corresponding to $D^3\times \{0\}$. The automorphism $\overline{\varphi}$ in question respects the attaching disk $D^3\times\{0\}$. Now observe that on the attached $D^3\times [0,1]$ we may first isotope $\overline{\varphi}$ to the map $\overline{\varphi}|_{D^3}\times \mathrm{id}_{[0,1]}$. Now the desired isotopy is given by $h_t\circ \overline{\varphi}\circ h_t^{-1}\colon \#_p^{\flat} D^3 \times S^1\rightarrow \#_p^{\flat} D^3 \times S^1$ where $h_t$  corresponds to the linear stretching map $D^3\times [-1,t]\rightarrow D^3\times [-1,0]$ and is the identity everywhere else on $\#_p^{\flat} D^3 \times S^1$. 

Now proceed inductively, adding the remaining $D_{c_i}\cup K_{e_i}$ until having arrived at $Y$. One proceeds analogously, the only difference being that $c_i$, $i\geq 1$ might now run through multiple edges where $\varphi$ is not the identity, i.e. through $e$ and $e_j$ for $j<i$. The argument still works if we just take individual disjoint stripes connecting each of these edges to $F_{c_i}$ and extend $\psi$ in the above fashion for each one.
\end{proof}

\subsection{Extending the equivariant homeomorphism}

\begin{prop}\label{prop:equiv-homeo} Let $X_1$ and $X_2$ be closed $T$-manifolds with $H^2(X_i/T)=0$ and let $M_1\subset X_1$ and $M_2\subset X_2$ be compact codimension $0$ submanifolds with boundary, whose interiors contain the nonregular stratum of the $X_i$. Assume that there is a retract $r\colon M_1/T\rightarrow W$ onto a subspace $W\subset \partial M_1/T$ such that $r$ is a homotopy equivalence and that furthermore there is a commutative diagram
\[\xymatrix{M_1\ar[d]\ar[r]^{\varphi} & M_2\ar[d]\\ X_1/T\ar[r]^\psi & X_2/T
}\]
where $\varphi$ is an equivariant homeomorphism, $\psi$ a nonequivariant homeomorphism and the vertical maps are the canonical ones. Then there is an equivariant homeomorphism $X_1\cong X_2$.

\begin{proof}
Let $C_i$ be the closure of the complements of the $M_i$ in the $X_i$. The map $\psi$ restricts to a homeomorphism $C_1/T\rightarrow C_2/T$. The action over these orbit spaces is free hence we have principal bundles $C_i\rightarrow C_i/T$. The pullback bundle $\psi^*(C_2)$ is isomorphic to $C_2$. It is isomorphic to $C_1$ if and only if their defining classes in $H^2(C_1/T)$ agree. Note however that their restrictions to $H^2(\partial M_1)$ agree since the bundles $C_1|_{\partial M_1/T}$ and $C_2|_{\partial M_2/T}$ are isomorphic via $\varphi$. From the Mayer-Vietoris sequence
\[H^2(X_1/T)\rightarrow H^2(M_1/T)\oplus H^2(C_1/T)\rightarrow H^2(\partial M_1/T)\]
we obtain injectivity of the right hand arrow. This implies the classes agree already in $H^2(C_1/T)$. Hence abstractly $C_1\cong \psi^* C_2\cong C_2$.

In order to obtain an equivariant homeomorphism $M_1\rightarrow M_2$ this needs to be compatible with the map $\varphi$ on $\partial M_1$. However, depending on the map $\varphi$, this might or might not be the case. From the commutative diagram
\[\xymatrix{
C_1|_{\partial M_1/T}\ar[d]\ar@/^1.3pc/[rr]^{\varphi} & C_1\ar[r]^\cong\ar[d] & C_2\ar[d]\\
 \partial M_1/T \ar[r]& C_1/T \ar[r]^\psi & C_2/T
}\]
we see the question is equivalent to asking whether the automorphism $\phi$ of $C_1|_{\partial M_1/T}$ obtained by composing $\varphi$ with the inverse of the right hand isomorphism does extend to all of $C_1$. Such an automorphism is defined by a map $f_\phi\colon \partial M_1/T\rightarrow T$ with the property that $\phi(p)=f_\phi(\pi(p))\cdot p$, where $\pi$ denotes the projection of the bundle. The automorphism $\phi$ extends if and only if $f_\phi\colon \partial M_1/T\rightarrow T$ extends to $C_1/T$. It is sufficient that $f_\phi$ is nullhomotopic since in that case we may homotope $f_\phi$ to a constant map on a collar neighbourhood of $\partial M_1/T$ in $C_1/T$ and extend it constantly to the rest of $C_1/T$. In order for such a map to be nullhomotopic it is sufficient that its restriction to $W$ is nullhomotopic. This will not in general be the case for $\varphi$ but there is a choice of equivariant homeomorphism $\varphi'\colon M_1\rightarrow M_2$ for which this holds. Namely starting with a given $\varphi$ we set $\varphi'=\varphi\circ \phi'$ where $\phi'$ is the equivariant automorphism of $M_1$ which we define as $p\mapsto f_\phi(r(\pi(p)))^{-1}\cdot p$. Extending the notation from above we see that the associated map $f_{\phi\circ \phi'}\colon \partial M_1/T\rightarrow T$ is constant on $W$. Hence by the previous arguments $\varphi'|_{\partial M_1}$ extends to an isomorphism $C_1\rightarrow C_2$ of principal bundles which yields an equivariant homeomorphism $X_1\rightarrow X_2$.
\end{proof}

\end{prop}

\subsection{The rigidity theorem}

\begin{thm}\label{thm:rigidity}
Let $X_1$, $X_2$ be two smooth simply-connected
$6$-dimensional integer $T$-GKM manifolds with the same GKM graph
$(\Gamma,\alpha)$. Assume further that there are two adjacent weights in
$\Gamma$ that form a basis, and that every finite isotropy stratum in $X_1$ and
$X_2$ contains a $T$-fixed point. Then $X_1$ and $X_2$ are $T$-equivariantly
homeomorphic.
\end{thm}

Before starting the proof, we mention the following application. In the mid 90's, Tolman \cite{Tolman} found the first example of a simply-connected compact manifold with an Hamiltonian torus action with finite fixed point set that does not admit an invariant K\"ahler structure. Shortly after, Woodward \cite{Woodward} modified her example so that it even extends to a multiplicity-free $\U(2)$-action. In \cite{GKZdim6} we proved that both spaces are nonequivariantly diffeomorphic to Eschenburg's twisted flag manifold $\SU(3)//T^2$, by showing that all these spaces are simply-connected and that their natural GKM actions have identical GKM graph. In view of the above theorem, the same ingredients imply
\begin{cor}\label{cor:tolwood}
Tolman's and Woodward's examples are equivariantly homeomorphic to Eschen\-burg's twisted flag manifold.
\end{cor}

The strategy to prove Theorem \ref{thm:rigidity} is as follows: by
Proposition \ref{prop:twohandlethickening} there are
codimension $0$ submanifolds $Y_i\subset X_i$ which are $T$-equivariantly
homeomorphic and such that the action is free outside of the
$Y_i$. Furthermore by \cite[Theorem
1]{AyzenbergMasuda} we know that $X_i/T$ is a rational homology $4$-sphere. As
the $X_i$ are simply-connected, \cite[Corollary II.6.3]{BredonIntro} implies
that the orbit spaces $X_i/T$ are simply-connected as well. The universal
coefficient theorem and Poincaré duality imply that $X_i/T$ is actually an
integral homology sphere and thus, by Freedman's solution of the
four-dimensional Poincar\'e conjecture \cite{Freedman}, already homeomorphic to
$S^4$. Furthermore by Lemma \ref{lem:sphere-twist}
we have $Y_i/T=\#_p^{\flat} D^3 \times S^1$ which deformation retracts onto a
wedge of circles contained in $\partial Y_i/T$ where the orbits are free. We
are hence in the situation of Proposition \ref{prop:equiv-homeo} with the only
remaining question being whether for a given equivariant identification
$\varphi$ below, one can choose the nonequivariant homeomorphism $\psi$ such
that the diagram

\[\xymatrix{Y_1\ar[d]\ar[r]^{\varphi} & Y_2\ar[d]\\ X_1/T\ar[r]^\psi & X_2/T
}\]
commutes. Or equivalently we ask whether a map on $Y_1/T$ prescribed by some equivariant homeomorphism $\varphi$ extends to a homeomorphism over the complements of the $Y_i/T$.

\begin{lem}\label{lem:complementconsum}
The complement $V_i$ of $Y_i/T\setminus \partial Y_i/T$ (for $i=1,2$) in $X_i/T$
is homeomorphic to the boundary connected sum 
\[ 
	\#_p^{\flat} S^2 \times D^2.
\]
\end{lem}
\begin{proof}
$V_i$ is a
topological $4$-manifold with boundary $\partial V_i = \partial Y_i$. Using the first part of Lemma \ref{lem:sphere-twist} and Seifert--Van Kampen for the decomposition $S^4\cong M_i/T\cup V_i$, we obtain the pushout of
\[\pi_1(\#_p^{\flat} D^3 \times S^1)\leftarrow \pi_1(\#_p S^2 \times S^1)\rightarrow \pi_1(V_i)\]
is $0$. Since the left hand map is an isomorphism it follows that $\pi_1(V_i)=0$. Furthermore from the decomposition $S^4=Y_i/T\cup V_i$ and its Mayer-Vietoris sequence we obtain $H^2(V_i)\cong H^2(\partial V_i)$. Since the latter is $\#_p S^2\times S^1$ we thus obtain $H^2(V_i)=H^1(\partial V_i)$. Now it follows from \cite{Boyer} that there is only one choice for a manifold $V_i$ bounding the given $\partial V_i$ and satisfying these properties. Indeed if $(H_2(V_i),L)$ is the intersection form then it sits in an exact sequence (cf.\ \cite[Definition 1.1]{Boyer})
\[0\rightarrow H_2(\partial V_i)\rightarrow H_2(V_i)\xrightarrow{\mathrm{ad}(L)}  (H_2(V_i))^{*}\rightarrow H_1(\partial V_i)\rightarrow 0.\]
For rank reasons the right hand arrow is an isomorphism, which implies $\mathrm{ad}(L)=0$ so $L$ is trivial. The uniqueness of $V_i$ is then provided by \cite[Remark 5.3]{Boyer}.
\end{proof}

Thus the proof of Theorem \ref{thm:rigidity} reduces to extending a given homeomorphism $\#_p S^2 \times S^1\rightarrow \#_p S^2 \times S^1$ to a map $\#_p^{\flat} S^2 \times D^2\rightarrow \#_p^{\flat} S^2 \times D^2$. We may assume this map to be orientation preserving as there is an orientation reversing automorphism of $\#_p^{\flat} S^2 \times D^2$. This only depends on the isotopy type of the first map:

\begin{lem}\label{lem: homeo enters the game}
Let $\tilde \psi \colon \partial V_1 \to \partial V_2$ be a homeomorphism isotopic
to $\psi_0$. Then $\psi_0$ extends to a homeomorpism of $V_1$ and $V_2$ if
$\tilde \varphi$ extends.
\end{lem}

\begin{proof}
	Let $H \colon \partial V_1  \times  I \to \partial V_2$ be an isotopy between $\psi_0$ and $\tilde \psi$, i.e., $H( \cdot , 0) =
	\psi_0$, $H( \cdot , 1) = \tilde \psi$ and for every $t \in
	I$ the map $H( \cdot , t)$ is a homeomorphism. Let $U_i$ be a collar
	neighbourhood of $\partial V_i$, i.e. $U_i \cong \partial V_i \times (-\tfrac{1}{2},
	1]$. The map
	\[
		\partial V_i \times [0,
		1]  \to \partial V_i \times [0,
	1], \quad (p,t) \mapsto (H(p,t),t)
	\]
	is clearly a homeomorphism. The complement $\tilde V_i$ of $\partial V_i \times (0,1]$ in
	$V_i$ is diffeomorphic to $V_i$ itself. Furthermore the above map induces a
	homeomorphism $\tilde \psi \colon \partial \tilde V_1 \to \partial \tilde V_2$ which extends by assumption to a homeomorphism $\tilde V_1 \to \tilde V_2$. Hence $\psi_0$ extends to a homeomorphism of $V_1 \to
	V_2$.
\end{proof}

Thus we need to understand which elements of $\pi_0 (\mathrm{Homeo}(\#_p S^2
\times S^1))$ (the group of path components of the homeomorphism group of $\#_p
S^2 \times S^1$) have a representative which extends to a homeomorphism of $V_1
\to V_2$. Denote by $\Gamma_+$ the path components of the orientation preserving
homeomorphisms. Laudenbach shows in \cite[Th\'{e}or\`{e}me 4.3]{Laudenbach} that
$\Gamma_+$ fits in an exact sequence
\[
	1 \longrightarrow (\ZZ_2)^p \longrightarrow \Gamma_+ \longrightarrow
	\textrm{Out}(*_p \ZZ) \longrightarrow 1
\]
where $\mathrm{Out}(*_p \ZZ)$ is the outer automorphism group and the right hand map is given via the action on $\pi_1(\#_p S^2
{\times S^1})\cong *_p \ZZ$. It is furthermore shown (cf.\ \cite[p.79]{Laudenbach}) that generators of the kernel of the map $\Gamma_+ \to
\textrm{Out}(*_p \ZZ)$ are given by sphere twists $H_{S^2 \times
\{\mathrm{pt}\}}(\gamma)$ for $S^2 \times \{\mathrm{pt}\}$ in $\#_p
S^2 \times S^1$.

Now with regards to the proof of the main theorem, according to Lemma \ref{lem:sphere-twist} we can change the automorphism of $\#_p S^2
\times S^1$ that we wish to extend by these sphere twists if we compose the initial $\varphi\colon Y_1\rightarrow Y_2$ with a suitable equivariant automorphism on some $Y_i$. Thus the proof of Theorem \ref{thm:rigidity} reduces to the proof of the following non-equivariant statement.

\begin{prop}
Every element in $\mathrm{Out}(\pi_1(\#_p
S^2 \times S^1))$ is induced by a homeomorphism of $\#_p
S^2 \times S^1$ which extends to $\#_p^{\flat} S^2 \times D^2$.
\end{prop}

This is proved in the next section.

\subsection{Extending homeomorphisms in dimension 4}
\label{sec:extendinghomeos}

\begin{rem}
In general not every homeomorphism of $\#_p S^2 \times S^1$ can be extended to
$\#_p^{\flat} S^2 \times D^2$. For $p=1$ take for example the clutching map $S^2
\times S^1 \to S^2 \times S^1$ of the non-trivial $S^2$ bundle over $S^2$. If
this map could be extended then the non-trivial $S^2$-bundle over $S^2$ would be
diffeomorphic to $S^2 \times S^2$ which is not the case.
\end{rem}

Let $x_1, \ldots x_p$ be generators of $\ast_p \ZZ$. Then from \cite[p.\
131]{Magnus} a generating set of $\mathrm{Out}(\ast_p \ZZ)$ is induced by the maps
($k, m = 1, \ldots, p$)
\begin{enumerate}[label=(\alph*)]
	\item $\sigma(x_k) = x_k^{-1}$, $\sigma(x_l) = x_l$ for $l \neq k$ 
	\item $\pi(x_k)=x_{m}$, $\pi(x_{m})=x_{k}$ and
		$\pi(x_l) =x_l$ for $l \neq {k}, {m}$
	\item $\rho(x_{k}) = x_{k}x_{m}$, $\rho(x_l) =x_l$ for $l \neq k$.
\end{enumerate}

In \cite[p.\ 82]{Laudenbach} Laudenbach describes elements in $\Gamma_+$ which are
mapped to the above generators. We would like to give an alternative
description of these elements, which will allow us to prove that they extend over 
$\#_p^{\flat} S^2 \times D^2$.

We take the $4$-ball $D^4\subset \mathbb{R}^4$ to which we add handles embedded in $\mathbb{R}^4$, each of which is homeomorphic to $I\times D^3$ and glued to $\partial D^4$ along the ends, such that the result is an embedding of $\#_p^{\flat} D^3 \times S^1$ into $\mathbb{R}^4$. Furthermore we view $\mathbb{R}^4$ as a subset of $S^4$ via the one-point compactification. The closure of the complement of $\#_p^{\flat} D^3 \times S^1\subset S^4$ is homeomorphic to $\#_p^{\flat} S^2 \times D^2$ (see e.g.\ Lemma \ref{lem:complementconsum}) giving rise to the decomposition
\[
	S^4 = \#_p^{\flat} S^2 \times D^2 \cup \#_p^{\flat} D^3 \times S^1
\]
Both sides of the decomposition intersect in the common boundary $C:= \#_p S^2 \times S^1$.\\

The strategy is as follows: We choose a base point in $S^3$ away from the handles and consider the fixed generating set $x_1,\ldots,x_p\in\pi_1(C)$ where each path corresponds to one of the handles and runs through it once while remaining in $S^3$ for the rest of the time. We note that this property does characterize the $x_i$ uniquely (up to sign) as removing the handles from $C$ results in a simply-connected space.

Our main tool for the construction is the
\emph{flip map} $\varphi \colon \RR^4
\to \RR^4$, defined by $\varphi(t_1,t_2,t_3,t_4) = (t_1,t_2,t_3,-t_4)$, which extends to $S^4$.
If the handles are placed symmetrically around the equator $t_4=0$ and the base point in $C$ is chosen with $t_4=0$, then $\varphi$ preserves $C$ and induces a map on $\pi_1(C)$. Furthermore $\varphi$ respects the above decomposition of $S^4$ and in particular restricts to $\#_p^{\flat} S^2 \times D^2$ as desired. The effect on $\pi_1(C)$ depends on the placement of the handles. The desired automorphisms will be of the form $f^{-1}\circ\varphi\circ f$ where $f$ is a diffeomorphism of $S^4$ moving the handles into a specific placement. We have the starting position of the handles be located within individual small $4$-balls in $\mathbb{R}^4$ away from each other. These small balls may be moved independently to desired positions while preserving $S^3$. Having placed these balls invariant under the flip map one can go on to ensure that the new placement of the handles is preserved by $\varphi$ as well. Via ambient isotopies the starting handle placement can be mapped to the new one via an ambient diffeomorphism $f\colon S^4\rightarrow S^4$. We may assume that $f$ fixes the base point on the equator. The map $f^{-1}\circ\varphi\circ f$ now respects $C$ and extends to $\#_p^{\flat} S^2 \times D^2$. When determining what it does on the $x_i$, note that the images of the $f_*(x_i)$ are characterized uniquely in the same way as the original $x_i$ with respect to the new handle placement.

The above construction produces orientation reversing automorphisms, although we are interested in orientation preserving ones. Thus as a first step we 
need to define a diffeomorphism $\omega \colon C \to C$ which reverses the
orientation and acts by the identity on the $\pi_1(C)$. This is achieved by placing the handles ``tangent'' to the equator as in the figure below.

		\begin{figure}[h]
			\centering
			\includegraphics{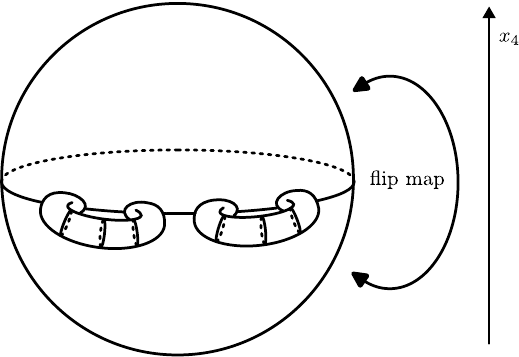}
			\caption{The placement of the discs for $\omega$}
			\label{fig:omega_flip} 
  	\end{figure}

We continue with a case by case study on the generators of
$\mathrm{Out}(\pi_1(C))$.
\begin{enumerate}[label = (\alph*)]
	\item To realize $\sigma$ we place one handle ``orthogonal'' to the equator by putting one boundary disk on each hemisphere of $S^3$ one handle, such
		that they are mapped to each other under the flip map. The remaining discs
		are placed along the equator as above.

		\begin{figure}[h]
			\centering
			\includegraphics{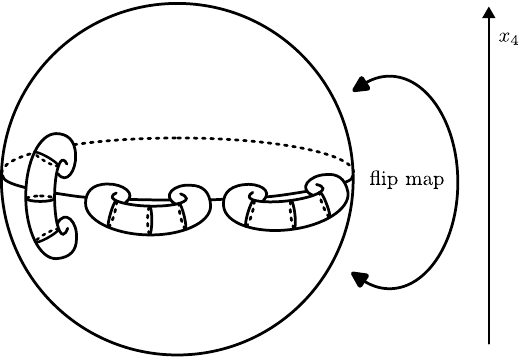}
			\caption{Diffeomorphism representing $\sigma$}
			\label{fig:case_a} 
		\end{figure} 
		
		This induces a diffeomorphism on $\colon C \to C$ which is
		orientation reversing and represents $\sigma$. Composing with $\omega$ yields the desired map.

	\item For this case we put one handle in each hemisphere such that they are mapped to one another by
		the flip map.
		The remaining handles are placed along the equator as before.

		 		\begin{figure}[H]
		 			\centering
		 			\includegraphics{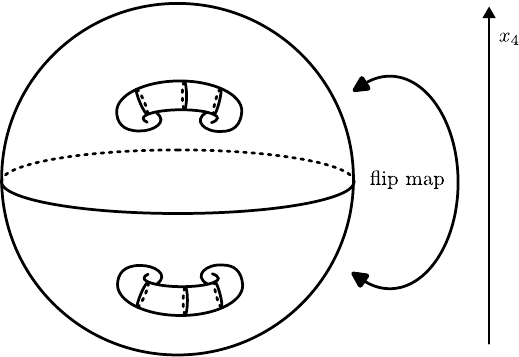}
		 			\caption{Diffeomorphism representing $\pi$}
		 			\label{fig:case_b}
		   	\end{figure}

		The map represents $\pi$ in $\mathrm{Out}(\pi_1(C))$. Furthermore
		concatenating with $\omega$ produces an orientation preserving map,
		which still represents $\pi$.

	\item To realize $\rho$ we consider the symmetric placement of the discs as
		pictured in figure \ref{fig:case_c}. Note that in addition to the previous setup this configuration requires us to bring two handles close to each other and slide one on top of the other.
				\begin{figure}[H]
					\centering
					\includegraphics{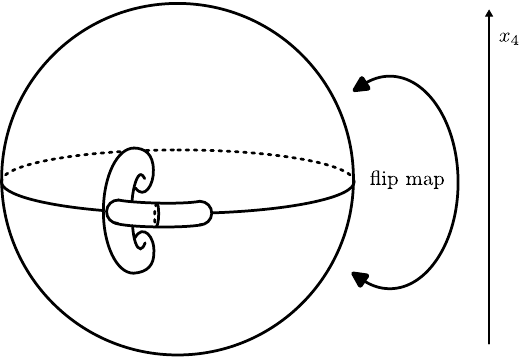}
					\caption{A diffeomorphism representing $\rho$.}
					\label{fig:case_c} 
		  	\end{figure}
	 With the same arguments as in the cases (a) and (b) we see, that the flip map 	 induces a diffeomorphism $h$ of $C$ which extends to $S^4$. 	 
	 It remains to check what $h$ does on $\pi_1(C)$. Consider Figure \ref{fig:paths_case_c},
	 the blue curve represents $x_k$ (or rather its image after moving the base setup to this new handle placement) and goes on the tube on top of
	 the other. The red curve goes through the tube which is reflected
	 under the flip map.
	 		\begin{figure}[H]
	 			\centering
	 			\includegraphics{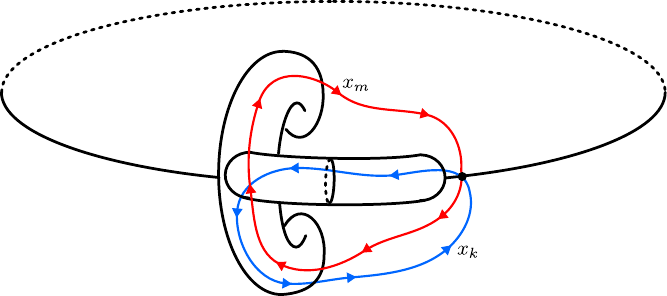}
	 			\caption{Generators under the diffeomorphism $h$}
	 			\label{fig:paths_case_c} 
	   	\end{figure}

We observe that $h_*(x_k)=x_k x_m$, $h_*(x_m)=x_m^{-1}$ and $h_*(x_l)=x_l$ for $l\neq k,m$. Now consider $g \colon C \to C$ from case
	(a) such that $g_*(x_m)= x_m^{-1}$ and $g_*(x_l) = x_l$ ($l \neq m$). Then $(h\circ g)_*(x_k)
 = x_k x_m$ and $(h \circ g)_*(x_l)=x_l$ for $l \neq k$. Hence $h\circ g$
represents $\rho$. 
\end{enumerate}

\subsection{Equivariant non-rigidity: knotted spheres}\label{sec:nonrigid}
In this and the next section we prove that the rigidity statement is false in general if one of the conditions on the finite isotropies is violated.

\begin{thm}\label{thm:nonrigidity}
There are two $6$-dimensional integer GKM $T^2$-manifolds which are not equivariantly homeomorphic but have the same GKM graph. The counterexamples can be chosen such that both satisfy one of the following conditions
\begin{enumerate}
\item[(a)] every closed stratum of a finite isotropy group contains a
	$T^2$-fixed point,
\item[(b)] there exists a $T^2$-fixed point in whose vicinity there occur at most two distinct finite nontrivial isotropy groups.
\end{enumerate}
\end{thm}

In this section we construct examples satisfying condition
(a) and violating (b) (cf.\ Proposition \ref{prop:differentembeddings}) while
examples satisfying (b) and violating (a) are constructed in the next section
(cf.\ Proposition \ref{prop:nonstandard})

The counterexample we construct in this section will have the GKM graph of the GKM $T$-action on $S^2\times S^2\times S^2$ which arises as the pullback of the standard componentwise $T^3$-action pulled back along the homeomorphism $T\rightarrow T^3$, $(s,t)\mapsto (s,st^2,st^5)$. The corresponding GKM graph is given by the cube

\begin{center}
\begin{tikzpicture}
\draw[very thick, red] (0,0) -- ++(2,0);
\draw[very thick, red] (0,2) -- ++(2,0);
\draw[very thick, red] (1,1) -- ++(2,0);
\draw[very thick, red] (1,3) -- ++(2,0);

\draw[very thick, blue] (0,0) -- ++(0,2);
\draw[very thick, blue] (2,0) -- ++(0,2);
\draw[very thick, blue] (1,1) -- ++(0,2);
\draw[very thick, blue] (3,1) -- ++(0,2);

\draw[very thick, cyan] (0,0) -- ++(1,1);
\draw[very thick, cyan] (2,0) -- ++(1,1);
\draw[very thick, cyan] (2,2) -- ++(1,1);
\draw[very thick, cyan] (0,2) -- ++(1,1);

  \node at (0,0)[circle,fill,inner sep=2pt]{};
  \node at (0,2)[circle,fill,inner sep=2pt]{};
  \node at (2,0)[circle,fill,inner sep=2pt]{};
  \node at (2,2)[circle,fill,inner sep=2pt]{};
  
   \node at (1,1)[circle,fill,inner sep=2pt]{};
  \node at (1,3)[circle,fill,inner sep=2pt]{};
  \node at (3,1)[circle,fill,inner sep=2pt]{};
  \node at (3,3)[circle,fill,inner sep=2pt]{};
  
  \node at (1,-0.5)[red]{(1,0)};
   \node at (-0.7,1)[blue]{(1,2)};
    \node at (3.3,0.5)[cyan]{(1,5)};

\end{tikzpicture}
\end{center}
where equally coloured (colinearly depicted) edges have the same label. The connection respects the sides of the cube such that the connection paths (cf.\ Section \ref{sec:2-handles}) are the paths around the faces of the cube. We could now follow the realization process outlined in Section \ref{sec:realization} until we arrive at the manifold $N_2$ (see Proposition \ref{prop:N1->N2}). We do however give a more direct construction of $N_2$ in order to avoid having the reader go through the whole construction process in order to check certain properties which we will claim below.

Consider the $T$-space $X:=S^2\times S^2\times S^2$ above and $T\subset T^3$ via the above embedding. The orbit space of $S^2$ under the standard $S^1$-action is the interval $I$, hence $X/T^3$ is the cube $I^3$. A section of the orbit projection $S^2\rightarrow S^2/S^1\cong I$ gives rise to a section $s\colon I^3\rightarrow X$, and we obtain a $T^3$-equivariant surjection $ I^3\times T^3\rightarrow X$. This induces a map $\varphi\colon I^3\times T^3/T\rightarrow X/T$. For $p\in \partial I^3$ the subgroup $T$ operates transitively on the $T^3$-orbit of $s(p)$. Thus $\varphi$ maps $\{p\}\times T^3/T$ to a single point. On the other hand for $p$ in the interior of $I^3$ the $T^3$-orbit of $s(p)$ is free and $\varphi$ is injective on $\{p\}\times T^3/T$. Thus there is a homeomorphism

\[\psi\colon (I^3\times T^3/T)/_\sim\longrightarrow X/T\] where $\sim$ collapses $\{p\}\times T^3/T$ whenever $p\in\partial I^3$.

Now we choose a small closed collar neighbourhood of $\partial I^3$ isomorphic to $S^2\times [0,\varepsilon]$ with $S^2\times \{0\}$ corresponding to $\partial I^3$. Let $Y$ be the preimage of this collar neighbourhood under the projection $X\rightarrow I^3$. Then $Y$ is the desired manifold satisfying the properties of Proposition \ref{prop:N1->N2}: indeed $\psi$ induces $Y/T\cong (S^2\times [0,\varepsilon]\times T^3/T)/_\sim$ where the $T^3/T$ part gets collapsed over $S^2\times\{0\}$. Identifying $T^3/T$ with a circle this gives $Y/T\cong S^2\times D^2$. Furthermore $\partial Y$ is the part of $Y$ lying over $S^2\times \{\varepsilon\}$. Using a splitting $T^3=S^1\times T$ we see that $\partial Y$ is equivariantly homeomorphic to $S^2\times S^1\times T$.

\begin{lem}\label{lem:singularstratum}
With respect to the above identification $Y/T\cong S^2\times D^2$ the set of nonregular orbits is precisely $S^2\times\{0\}$
\end{lem}

\begin{proof}
We have already seen that the $T^3$- and hence also the $T$-orbits in $X$ are free over the interior of $\partial I^3$. It remains to show that every $T$-orbit lying over $\partial I^3$ is indeed nonregular. Those orbits consist of those points where at least one of the three components in $S^2\times S^2\times S^2$ is a pole $P$ of the corresponding $S^2$ and hence remains fixed under the action. Thus it suffices to show that the action is not effective on the subsets of the form $\{P\}\times S^2\times S^2$ (up to switching of the factors). This is equivalent to the condition that the intersection of the kernels of any two of the weights in the above GKM graph is nontrivial. To see this one checks that any pair within the weight vectors $(1,0),(1,2),(1,5)$ is not a basis of $\mathbb{Z}^2$.
\end{proof}

Now we finish the construction as in Section \ref{sec:realization}: choose some topological embedding $\phi\colon S^2\times D^2\rightarrow S^4$. Denoting by $Z$ the closure of the complement of the image of $\varphi$ we may now glue $Z\times T$ equivariantly to $Y$ along $\partial Y\cong S^2\times S^1\times T$ to obtain a manifold $X_\phi$ with orbit space homeomorphic to $S^4$. By the results of Section \ref{sec:cohom}, $X_\phi$ is an integer GKM manifold. We observe that there is a commutative diagram

\[\xymatrix{Y\ar[r]\ar[d] & X_\phi\ar[d]\\ S^2\times D^2\ar[r]^{\phi} &S^4}\]
where the horizontal maps are orbit projections.

\begin{prop}\label{prop:differentembeddings}
Let $\phi_1,\phi_2\colon S^2\times D^2\rightarrow S^4$ be two topological embeddings such that the complements of $S^2\times\{0\}$ in $S^4$ are not homeomorphic. Then the manifolds $X_{\phi_1}$ and $X_{\phi_2}$ as constructed above are not equivariantly homeomorphic.
\end{prop}

\begin{proof}
By Lemma \ref{lem:singularstratum} the image of the nonregular orbits in $X_{\phi_i}/T\cong S^4$ is precisely the image of $S^2\times\{0\}\rightarrow S^2\times D^2\xrightarrow{\phi_i} S^4$.
An equivariant homeomorphism $f\colon X_{\phi_1}\rightarrow X_{\phi_2}$ maps the nonregular orbits onto the nonregular orbits. In particular it would restrict to a homeomorphism between $S^4\backslash \phi_1(S^2\times\{0\})$ and $S^4\backslash \phi_2(S^2\times\{0\})$. By assumption this cannot happen.
\end{proof}

There are knotted $2$-spheres in $S^4$, first constructed in \cite{Artin} by taking an arc in the real half space $\{(x_1,x_2,x_3)\in\mathbb{R}^3~|~ x_3\geq 0\}$ and rotating it around the plane $x_3=x_4=0$ in $\mathbb{R}^4$. We note that the resulting embeddings of $S^2$ extend to embeddings of $S^2\times D^2$ by considering a tube around the original arc in the $3$-dimensional half space. The author of \cite{Artin} computes the fundamental group of the complements of these spheres, which turn out to be different for suitable choices of arcs. Thus embeddings as in Proposition \ref{prop:differentembeddings} exist, which proves Theorem \ref{thm:nonrigidity}.

\begin{rem}\label{rem:smoothuniqueness}
Although the $T^2$-manifolds in the statement of Theorem \ref{thm:rigidity} are smooth, the rigidity is only obtained up to equivariant homeomorphism. Furthermore we even know the underlying manifolds to be diffeomorphic by \cite{GKZdim6}. Therefore we want to comment on the possibility of an equivariant diffeomorphism classification. Recall that manifolds $X_i$, $i=1,2$, as in Theorem \ref{thm:rigidity} are the equivariant gluing of a thickening $M$ of the GKM graph as well as some free $T$-space $C_i$ with orbit space $C_i/T=\#^\flat_p S^2\times D^2$. The proof constructs an equivariant homeomorphism
$C_1\rightarrow C_2$ compatible with the gluing data and naturally one would want to strengthen this to an equivariant diffeomorphism. In that case it would follow that $C_1/T$ and $C_2/T$ are diffeomorphic with respect to the smooth structures induced from the $X_i$.
However to our knowledge it is presently not known whether smooth structures on $\#^\flat_p S^2\times D^2$ are even unique. In case $C_1/T$ and $C_2/T$ are not diffeomorphic no equivariant diffeomorphism as above can exist.

We note that if exotic smooth structures as above exist, then they will occur in
our equivariant setting by the following construction: we start with a smooth
GKM manifold $X$ and $M,C\subset X$ as above. Now choose any smooth structure on
$C/T$. we can give $C$ a compatible smooth structure by smoothing the map
$C/T\rightarrow BT$. In this way $C\rightarrow C/T$ becomes a smooth principal
bundle. However since the smooth structure on $\partial C/T\cong\#_p S^2\times
S^1\cong \partial M/T$ is unique, the result will always glue together smoothly:
there is a diffeomorphism $ \partial C/T\cong \partial M/T$ homotopic to the original
gluing data. It follows via smoothing of homotopies of the maps to $BT$ that
the, a priori only topologically equivalent, principal bundles $\partial M$ and
$\partial C$ may be identified smoothly. Hence the result is a smooth GKM
manifold whose smooth structure induces the (possibly exotic) initial smooth
structure on $C/T$.
\end{rem}

\subsection{Equivariant non-rigidity: intersection of finite isotropy strata}

We consider the example of $T=T^2$ acting on $X:=S^2\times S^2\times S^2$ from
the previous section, although we change the action to be the pullback of the
standard $T^3$-action along $(s,t)\mapsto(s,st^2,st^3)$. The weights are
$\alpha=(1,0)$, $\beta=(1,2)$, and $\gamma=(1,3)$. We note that
$\ker\alpha\cap\ker \beta=1\times\mathbb{Z}_2=:H_1$, $\ker\alpha\cap \ker
\gamma=1\times \mathbb{Z}_3=:H_2$ and $\ker\beta\cap\ker\gamma$ is trivial. In
particular the action is free over two of the faces of the cube and
hence satisfies condition (b) from Theorem
\ref{thm:nonrigidity}.

Now in the cube $I^3=X/T^3$ consider a tube $\overline{A}\cong D^2\times I$,
where $D^2\times\{0\}$ is embedded in a face over which the $T$-isotropy is
$H_1$ while $D^2\times\{1\}$ is embedded in a  face which corresponds to
$T$-isotropy $H_2$ and the rest of the tube lies in the interior of the cube.
Let $A=\pi^{-1}(\overline{A})$ where $\pi$ is the map $X\rightarrow X/T$. Then
using the map $\psi$ defined in the previous section we obtain a homeomorphism
$A/T=(D^2\times I\times T^3/T)/\sim$ where $\sim$ collapses the $T^3/T$-factor
over $D^2\times\{0,1\}$. Hence $A/T\cong D^2\times S^2$. The
subset in $D^2\times S^2$ corresponding
to nonregular orbits is of the form $D^2\times\{N,S\}$ for $N,S\in S^2$. Now
choose a small open disk $C\subset S^2$ away from $\{N,S\}$.
We set $B\subset X$ to be the union of all points lying over
$D^2\times (S^2\backslash C)$. This is a $T$-manifold with
boundary and $B/T\cong D^2\times D^2$ where the singular orbits are represented
by $B^{H_1}/T=D^2\times\{N\}$ and $B^{H_2}=D^2\times \{S\}$ for $N,S$ in the
interior of $D^2$.

We wish to modify $X$ by cutting out $B$ and gluing in a modified version in which the $H_1$- and the $H_2$-stratum intersect without violating the integer GKM condition. We collect the needed properties of this modification in the following

\begin{lem}\label{lem:nonstandardball}
There is a $T$-manifold $B'$ with boundary such that
\begin{enumerate}
\item $\partial B'$ is equivariantly homeomorphic to $\partial B$ but $B'^{H_1}$ and $B'^{H_2}$ intersect in two orbits with isotropy $H_1\cdot H_2$.
\item There is a surjection $\pi_1(B_T)\rightarrow \pi_1(B'_T)$ which commutes with the maps from $\pi_1(\partial B_T')\cong \pi_1(\partial B_T)$.
\item There is an isomorphism $H^*_T(B')\cong H^*_T(B)$ commuting with the maps to $H^*_T(\partial B')\cong H^*_T(\partial B)$.
\end{enumerate}
\end{lem}

\begin{proof}
Before we come to the construction of $B'$ we construct
certain tubes around the nonregular orbits in $\partial B$. These nonregular
orbits consist of those points lying over $\partial \overline{A}\cap \partial
I^3$. Their $T^3$-orbit space consists of of two disjoint circles
over which the isotropy is constant $H_1$ or $H_2$. We consider an equivariant
tube $U_1\subset \partial B$ around the circle of singular orbits
with isotropy $H_1$. Recall that via a section of
$X\rightarrow X/T^3\cong I^3$ we may write $X=I^3\times T^3/\sim$, where $\sim$
collapses certain isotropy groups in the $T^3$ factor. With respect to this
description we may understand $U_1$ as $S^1\times [0,\epsilon]\times T^3/\sim$
where $S^1\times [0,\epsilon]$ is understood as the corresponding subset in $
\overline{A}\cong D^2\times [0,1]$ such that $S^1\times\{0\}$ is the
intersection with a face of the cube over which the third circle of $T^3$ (i.e.\
the one corresponding to the weight $\gamma$) acts trivially. Since the third
circle factor gets collapsed over $S^1\times\{0\}$ we may rewrite $U_1\cong
S^1\times D^2_\gamma\times S^1_\alpha\times S^1_\beta$. Similarly a tube around
the circle of singular $T$-orbits with isotropy $H_2$ is of the form $U_2\cong
S^1\times D^2_\beta\times S^1_\alpha\times S^1_\gamma$.

We now turn to the construction of $B'$. Consider the group
$H=H_1\cdot H_2$ and its diagonal action on $D^2_\beta\times
D^2_\gamma$, i.e.\ $H_1$ acts only on $D^2_\gamma$ and $H_2$
only on $D^2_\beta$. Associated to this we obtain the $T$-space
$Y:=(D^2_\beta\times D^2_\gamma)\times_H T$ where we have divided by the
diagonal $H$-action and the $T$-action on $Y$ is understood as acting only on
the rightmost $T$-summand. The orbit space of the singular stratum consists of two disks $Y^{H_1}$
and $Y^{H_2}$ which intersect in the center. However $\partial Y$ is different
from $\partial B$ as the circles of singular orbits in $\partial Y/T\cong S^3$
are linked. To overcome this problem we consider two copies of $Y$ and glue
them.

Let $p\colon Y\rightarrow Y/T\cong D^2_\beta/ H_2\times D^2_\gamma/H_1\cong D^2\times D^2$ denote the projection. Now in $\partial Y/T=S^1\times D^2\cup D^2\times S^1$ we consider the subspace \[C=S^1_{[-\frac{\pi}{2},\frac{\pi}{2}]}\times D^2\cup D^2\times S^1_{[-\frac{\pi}{2},\frac{\pi}{2}]}\]
where $S^1_{[a,b]}=\{e^{it}~|~t\in [a,b]\}$. One checks that $C$ is a topological $3$-ball. Now we define $B':=Y\cup_{p^{-1}(C)} Y$ as two copies of $Y$ glued along $p^{-1}(C)$. We prove that it satisfies properties $(i),(ii)$, and $(iii)$ from Lemma \ref{lem:nonstandardball}.

\begin{figure}[H]
	\centering
	\includegraphics{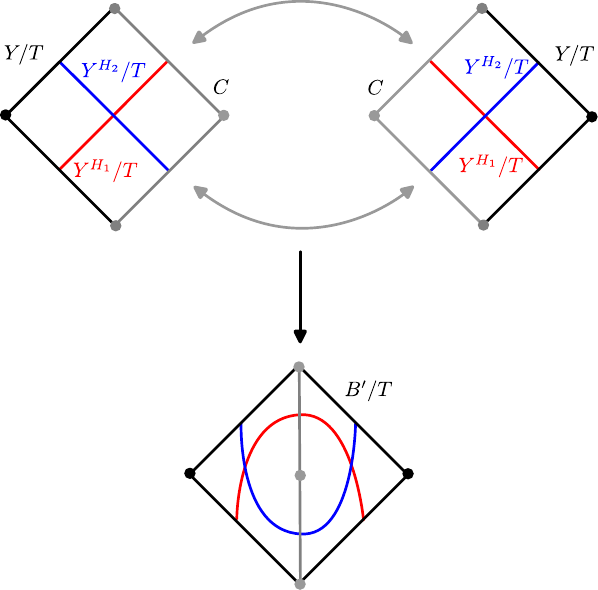}
  \caption{The glueing of the double of $Y$ along $p^{-1}(C)$ in the $T$-orbit
	space (schematically).}
\end{figure}

With regards to $(i)$ we note first that $B'/T$ is the boundary connected sum of
two $4$-balls and hence itself a $4$-ball. Thus $\partial B'/T\cong S^3\cong
\partial B/T$. The strategy to prove $\partial B\cong
\partial B'$ is to find equivariant tubes $U_1'$, $U_2'$ around the singular
orbits in $\partial B'$, equivariant homeomorphisms $U_i\cong U_i'$ and then
extend this to a global homeomorphism by using Proposition
\ref{prop:equiv-homeo}.

A neighbourhood of $(\partial Y)^{H_1}$ in $\partial Y$ is given by $(S^1_\beta\times D^2_{\gamma,\epsilon})\times_H T$, where $D^2_{\gamma,\epsilon}$ refers to the subdisc of $D^2_\gamma$ of radius $\epsilon$ for some small $\epsilon>0$.
Its image under $p$ in $Y/T\cong D^2\times D^2$ is $S^1\times D^2_\epsilon$.
However if we view this neighborhood as a subspace in one of the copies of $Y$ in $B'$, then only $p^{-1}(S^1_{[\frac{\pi}{2},\frac{3\pi}{2}]}\times D^2_\epsilon)$ lies in $\partial B'$ since its rest lies in the interior of $p^{-1} (C)$. Let $J\subset S^1_\beta$ be an interval such that $J$ maps homeomorphically onto $S^1_{[\frac{\pi}{2},\frac{3\pi}{2}]}$ under $S^1_\beta/H_2\cong S^1$. Then we have an equivariant homeomorphism
\[p^{-1}(S^1_{[\frac{\pi}{2},\frac{3\pi}{2}]}\times D_\epsilon)\cong J\times (D^2_{\gamma,\epsilon}\times_{H_1} T)\]
induced by the canonical inclusion $J\times (D^2_{\gamma,\epsilon}\times_{H_1} T)\rightarrow S^1_\beta\times D^2_{\gamma,\epsilon}\times_H T$.

We consider the double of this in $\partial B'$: the subspace $\partial J\times
(D^2_{\gamma,\epsilon}\times_{H_1} T)$ lying over the points $\{\pm i\}\times
D^2_\epsilon\subset S^1\times D^2_\epsilon$ glues with itself via the identity
and thus yields a tube $U_1'$ around $(\partial B')^{H_1}$ in $\partial B'$ with
$U_1'\cong S^1\times (D^2_{\gamma,\epsilon}\times_{H_1} T)$. We note that this
is equivariantly homeomorphic to $U_1$ above via the map
\[S^1\times (D^2_{\gamma,\epsilon}\times_{H_1} T)\rightarrow S^1\times (D^2_\gamma\times S^1_\alpha\times S^1_\beta)\]
obtained by $(x,y,t)\mapsto (x,\epsilon^{-1}
\overline{y}\gamma(t),\alpha(t),\beta(t))$. Complex
conjugation in the $y$-coordinate is needed as
$D^2_{\gamma,\epsilon}\times_{H_1} T$ is the quotient of the diagonal left
$H_1$-action.

The next step is to observe that the circles in the center of $U_1/T$ and
$U_1'/T$ are unknots in $\partial B/T\cong S^3\cong \partial B'/T$. We start
with $U_1'$ and the knot $\eta_1'=S^1\times \{0\}\subset
U_1'/T\cong S^1\times D^2_{\epsilon}$. We consider
\[W=S^1_{[\frac{\pi}{2},\frac{3\pi}{2}]}\times [0,1]\subset S^1\times D^2\subset
\partial(D^2\times D^2)\cong \partial Y/T.\]
Now observe that \[W\cap C= \{i\}\times [0,1]\cup S^1_{[\frac{\pi}{2},\frac{3\pi}{2}]}\times\{1\}\cup \{-i\}\times[0,1].\]
Let $D'=W\cup W\subset \partial B'/T$ be the double of $W$. It consists of two copies of $W$ glued along the subspace above. We observe that $D'$ is an embedded $2$-disk bounding $\eta'$ whose interior is contained in the regular orbits.

Similarly we construct a disk $D\subset \partial B/T$ which bounds
$\eta_1=S^1\times \{0\}\subset S^1\times (D^2_\gamma \times S^1_\alpha\times
S^1_\beta)/T=U_1/T$. Recall the initial homeomorphism $B/T= D^2\times D^2$ under
which the singular orbits are $D^2\times \{N,S\}$ for $N,S\in \mathring{D^2}$.
Under this identification $\eta_1= S^1\times\{N\}$. Let
$P \subset D^2$ be a path from $N$ to a point $P(1)\in S^1$ which does not go
through $S$. Then a disk $D$ in $\partial B/T$ bounding $\eta_1$ is given by
$S^1\times P\cup D^2\times P(1)$.

Now in completely analogous fashion we build an equivariant tube $U_2'\subset \partial B'$ around the circle $\eta_2'$ of singular orbits $(\partial B')^{H_2}$ and consider the analogous equivariant homeomorphism $U_2'\cong U_2$. In the same way as before we find a disk in $\partial B'/T$ bounding $\eta_2'$ and we note that it is disjoint from the previously constructed $D'$ which bounds $\eta_1'$. Hence $\eta_1',\eta_2'$ are two unlinked unknots. The same applies to the knots $\eta_1,\eta_2$ of singular orbits in $\partial B/T$.

We wish to extend the equivariant homeomorphism $\varphi\colon U_1'\cup
U_2'\rightarrow U_1\cup U_2$ to an equivariant homeomorphism $\partial B'\cong
\partial B$ with the help of Proposition \ref{prop:equiv-homeo}. As $\partial
B/T\cong S^3\cong \partial B'/T$, we have $H^2(\partial B/T)= H^2(\partial
B'/T)=0$. Moreover, regarding the requirements of the
proposition we check that $U_1/T \cong S^1\times D^2_{\epsilon}$ deformation
retracts onto $S^1\times \{{\mathrm{pt}}\}\subset S^1\times\partial
D^2_\epsilon$, and analogous observations hold for $U_2/T$ as well as
$U_i'/T$. The only requirement that needs investigation is the existence of a
homeomorphism $\psi\colon \partial B'/T\rightarrow \partial B/T$ which extends
the map induced by $\varphi$. One way to see this is the following: there is a
smooth structure on $\partial B'/T$ with respect to which the maps $S^1\times
D^2\cong U_i'/T\rightarrow \partial B'/T$ are smooth. Note however that it needs
to be constructed directly on $\partial B'/T\cong D^4$ and is not (everywhere)
induced from a smooth structure on $\partial B'$ as the action is not free. The
same holds for $S^1\times D^2\cong U_i/T\subset \partial B/T$. Thus the question
comes down to whether two pairs of smooth tubular neighbourhoods around two
unlinked smooth unknots can be converted into one another by means of an ambient
diffeomorphism $f$. It is well known that there is a diffeomorphism mapping the
respective knots onto one another and the uniqueness of smooth tubular
neighbourhoods implies that $f$ as above exists up to potentially applying Dehn
twists to $S^1\times D^2\cong U_i/T\rightarrow \partial B/T$. But Dehn twists
lift to equivariant maps of the $U_i$ (see e.g.\ the map $\kappa$ from
Proposition \ref{prop:N1->N2}) hence we find the desired extension $\psi$ after
potentially modifying $\varphi$ via equivariant Dehn twists. This finishes the
proof of part $(i)$ of Lemma \ref{lem:nonstandardball}.

For the proof of $(ii)$ we first compute $\pi_1(B_T')$. Recall that $B'$ is a
gluing of two copies of $Y$ along $p^{-1}(C)$. The space $Y$ is equivariantly
homotopy equivalent to its orbit $\{0\}^2\times_H T\cong T/H$. Hence the
inclusion of this orbit induces a homotopy equivalence $BH\rightarrow Y_T$. It
remains to study the space $p^{-1}(C)$ which is itself a gluing.
Let $J'\subset S^1_\beta$ be an interval which maps
	homeomorphically to $S^1_{[-\frac{\pi}{2},\frac{\pi}{2}]}$ under
	$S^1_\beta/H_2\cong S^1$.
We observe
\[A_1:=p^{-1}(S^1_{[-\frac{\pi}{2},\frac{\pi}{2}]}\times D^2)\cong J'\times (D^2_\gamma\times_{H_1} T)\simeq T/H_1\]
where the first homeomorphism is induced by the canonical inclusion $J'\times
(D^2_\gamma\times_{H_1} T)\rightarrow (S^1_\gamma\times D^2_\beta)\times_H T$
and the right hand equivariant homotopy equivalence is the inclusion of any
$T/H_1$ orbit. Similarly $A_2:=p^{-1}(D^2\times
S^1_{[-\frac{\pi}{2},\frac{\pi}{2}]})\simeq T/H_2$. We note that the
intersection $S^1_{[-\frac{\pi}{2},\frac{\pi}{2}]}\times D^2\cap D^2 \times
S^1_{[-\frac{\pi}{2},\frac{\pi}{2}]}$ is contractible and consists of free
$T$-orbits. Hence $(A_1\cap A_2)_T$ is contractible. Thus $\pi_1(p^{-1}(C)_T)$
is the free product $\pi_1(BH_1)\ast \pi(BH_2)$ with the map being induced by
two orbit inclusions. We note that the inclusion of a $T/H_i$-orbit into $B'$ is
equivariantly homotopic to the projection $T/H_i\rightarrow T/H$ onto the
$T/H$-orbit in one of the two copies of $Y$. Hence by Seifert-Van Kampen
$\pi_1(B'_T)$ is isomorphic to the pushout of
\[\pi_1(BH)\leftarrow \pi_1(BH_1)\ast\pi_1(BH_2)\rightarrow \pi_1(BH).\]
By our choice of groups $\pi_1(BH_1)\cong \mathbb{Z}_2$, $\pi_1(BH_2)\cong \mathbb{Z}_3$, and $\pi_1(BH)\cong \mathbb{Z}_6$ with the maps between them being the obvious inclusions. Hence the above pushout is just isomorphic to a single copy of $\pi_1(BH)$ and $\pi_1(BH)\cong \pi_1(B'_T)$ is induced by the inclusion of any of the two $T/H$-orbits.

Now observe that $\pi_1(BH)\cong \pi_1(BH_1)\oplus \pi_1(BH_2)$ via the obvious maps $\pi_1(BH_i)\rightarrow \pi_1(BH)$. Furthermore the inclusion of a combination of any $T/H_1$- and $T/H_2$-orbit into $B'$ factors -up to equivariant homotopy- through the $T/H$-orbits  and thus induces an isomorphism $\pi_1(BH_1)\oplus \pi_1(BH_2) \rightarrow\pi_1(B'_T)$.

We aim for a similar description of $\pi_1(B_T)$. It is convinient to be a little more specific on the construction of $B$. Recall that $A=D^2\times I\times T^3/\sim$ and that $B$ arises from $A$ by removing a small $D^2\times D^2\times T$ from $A$, away from the singular stratum. Explicitly we can do the following: write $T^3=S\times T$ for a circle $S$ complementary to the image of $T$ in $T^3$. Then set $B=A\backslash (D^2\times (\frac{1}{3},\frac{2}{3})\times S_{(-\epsilon,\epsilon)}\times T)$, where $S_{(-\epsilon,\epsilon)}\subset S$ is a small connected neighbourhood. Now cover $B=B^+\cup B^-$ where $B^+$ (resp.\ $B^-$) consists of those points in $B$ where the $I$-component is $>\frac{1}{3}$ (resp. $<\frac{2}{3}$). Then $B^+$ (resp.\ $B^-$) deformation retracts equivariantly onto any $T/H_1$-orbit (resp.\ $T/H_2$-orbit). The intersection $(B^+\cap B^-)_T$ is contractible. Hence we obtain that the inclusion of any two orbits of this type induces an isomorphism

\[\pi_1(BH_1)\ast \pi_1(BH_2)\rightarrow \pi_1(B_T).\]
Including two orbits of this kind into the boundary of $\partial B\cong\partial B'$ and considering the inverse of the resulting isomorphisms as above leads to the commutative diagram

\[\xymatrix{
\pi_1(BH_1)\ast \pi_1(BH_2)\ar[d] & \ar[l]_(0.3){\cong} \pi_1(B_T) &\ar[l]\pi_1(\partial B_T)\ar[d]^\cong\\
\pi_1(BH_1)\oplus \pi_1(BH_2) & \ar[l]_(0.3){\cong} \pi_1(B_T')& \ar[l] \pi_1(\partial B_T').
}\]
The homomorphism $\pi_1(B_T)\to \pi_1(B_T')$ making this diagram commutative proves $(ii)$.

The proof of part $(iii)$ is very similar. We note that $H^*(BH_1)\cong \mathbb{Z}[t]/(2t)$, $H^*(BH_2)\cong \mathbb{Z}[t]/(3t)$, and $H^*(BH)\cong \mathbb{Z}[t]/(6t)$ with $t$ in degree $2$. The maps $H^*(BH)\rightarrow H^*(BH_i)$ are the obvious projection maps. Now the preceding discussion of $B'$ with the Mayer-Viertoris sequence applied to the identical decompositions as in the computation of $\pi_1(B'_T)$ shows that $H^*(B_T')$ is isomorphic to $H^*(BH)$ with the isomorphism being induced by the inclusion of a $T/H$-orbit. Now as before we may consider any $T/H_1$- and $T/H_2$-orbit in $B'$. The inclusion of these orbits factors (up to homotopy) through the inclusion of the $T/H$-orbit and hence the resulting map factors as \[H^*(B'_T)\rightarrow H^*(BH)\rightarrow H^*(BH_1)\times H^*(BH_2).\] The latter map is an isomorphism in positive degrees as one can see degreewise from the previous algebraic description of the maps using $\mathbb{Z}_6\cong \mathbb{Z}_2\times\mathbb{Z}_3$.

If we can show that an analogous description is valid for $H^*(B_T)$ then this proves part $(iii)$ of Lemma \ref{lem:nonstandardball}.
The Mayer-Vietoris sequence applied to the previously discussed decomposition $B=B^-\cup B^+$ yields an isomorphism
\[H^*(B_T)\rightarrow H^*(BH_1)\times H^*(BH_2)\]
in positive degrees induced by including any $T/H_1$- and $T/H_2$-orbit. Choosing those orbits in the boundary of $\partial B\cong\partial B'$ these maps factor through the boundary and we obtain a commutative diagram
\[\xymatrix{
H^*(BH_1)\times H^*(BH_2) \ar[r]^(0.7){\cong} \ar[rd]^\cong & H_T^*(B)\ar[r] &H_T^*(\partial B)\ar[d]^\cong\\
 & H_T^*(B')\ar[r]&H^*_T(\partial B').
}\]
The isomorphism $H^*_T(B) \to H^*_T(B')$ making this diagram  commutative proves part $(iii)$.
\end{proof}

\begin{prop}\label{prop:nonstandard}
The $T$-manifold which arises from $X$ by removing $B$ and gluing in $B'$ along $\partial B'\cong \partial B$ is a simply-connected integer GKM manifold. It has the same GKM graph as $X$ but unlike $X$ does not satisfy condition (a) from Theorem \ref{thm:nonrigidity}.
\end{prop}

\begin{proof}
The modified manifold $X'$ violates condition (a) since the $H$-fixed points (with $H$ as above) consist of two isolated $T/H$-orbits.

Let $C$ be the complement of $B$ in $X$. Then by Lemma \ref{lem:nonstandardball} there is a commutative diagram
\[\xymatrix{
\pi_1(B_T)\ar[d]& \pi_1(\partial B_T)\ar[d]\ar[r]\ar[l] & \pi_1(C_T)\ar[d]\\
\pi_1(B_T')&\pi_1(\partial B_T')\ar[r]\ar[l] &\pi_1(C_T)
}\]
in which all vertical maps are surjective and the pushout of the first row yields $\pi_1(X_T)$ while the pushout of the second row is $\pi_1(X'_T)$. Hence there is a surjection $0=\pi_1(X_T)\rightarrow \pi_1(X'_T)$. Since $X'$ has a fixed point we obtain also $\pi_1(X)=0$.

In the Mayer-Vietoris sequence of $X=C\cup B$ we observe that $H^*_T(\partial B)$ is $R$-torsion. Since $X$ is integer GKM $H_T^*(X)$ is torsion-free. Hence the connecting homomorphism (a map of $R$-modules) is trivial.  We obtain a commutative diagram
\[\xymatrix{
0 \ar[r] & H^*_T(X) \ar[r] & H^*_T(C) \oplus H^*_T(B) \ar[r]\ar[d]^{\cong} & H^*_T(\partial B) \ar[r]\ar[d]^\cong & 0 \\
&H^*_T(X') \ar[r] & H^*_T(C)\oplus H^*_T(B') \ar[r] & H^*_T(\partial B')
}\]
in which the upper row is short exact, the vertical isomorphism in the middle is given by the map in part $(iii)$ of Lemma \ref{lem:nonstandardball}, and the bottom row is part of the Mayer-Vietoris sequence of $X'=C\cup B'$. It follows that $H^*_T(C)\oplus H^*_T(B')\to H^*_T(\partial B')$ is surjective, hence the bottom row, after adding zeros on the left and right, is short exact as well.  It follows that $H_T^*(X')\cong H_T^*(X)$ and $X$ is indeed an integer GKM $T$-manifold.
\end{proof}

\appendix

\section{Boundary connected sum of topological manifolds}\label{sec:topologicalshiat}

Throughout the paper we encounter various gluing constructions along disks and
we want to argue that these do not depend on the choice of disk. As our orbit
spaces are not naturally endowed with a smooth structure, it is in some way the
easiest solution to work in the topological category. There are however a few
subtleties that arise because topological disks can be rather ''wild``
\footnote{For example one could embed the \emph{Alexander
horned ball} into $S^3$, but there cannot be an orientation preserving homeomorphism of $S^3$ mapping
this ball to one of the hemispheres, since the complement of the Alexander
horned ball is not simply-connected.}. Thus actually proving that working in the topological category does not cause us any problems does involve rather heavy tools. As references on these matters are scarce, we will discuss the needed material below, although it is certainly well-known to experts.

\begin{defn}\label{defn:bicollared}
An embedded sphere $S^{n-1}\rightarrow M$ in an $n$-dimensional topological manifold is called \emph{bicollared} if it has a tubular neighbourhood homeomorphic to $S^{n-1}\times \mathbb{R}$ where $S^{n-1}$ corresponds to the $0$-section $S^{n-1}\times\{0\}$.
\end{defn}

We fix an orientation for $D^n$. Let $M,N$ be $n$-dimensional oriented topological manifolds with boundary. Let $\phi\colon D^{n-1}\rightarrow \partial M$, $\psi\colon D^{n-1}\rightarrow \partial N$ be topological embeddings whose boundary spheres are bicollared. Assume furthermore that $\phi$ preserves orientations and that $\psi$ is orientation reversing. The \emph{boundary connected sum} $M\#^\flat N$ is given by gluing along the chosen disks i.e.\ by identifying  $\phi(x)\sim\psi(x)$ for $x\in D^n$. The result is a topological manifold with boundary where the boundary is given by the usual connected sum $\partial M\#\partial N$ along the embedded disks $\phi,\psi$.

We want to argue that this gives a well defined space independent from the
choices involved. In the smooth category this is a consequence of the \emph{disk
theorem} of Palais \cite{MR117741} and Cerf \cite{MR140120}.
We will use the following topological variant:

\begin{prop}\label{prop:diskscooter}
Let $M$ be a connected oriented topological manifold and $\phi,\phi'\colon D^n\rightarrow M$ be two disks preserving the orientation whose boundary spheres are bicollared. Then there is a homeomorphism $h\colon M\rightarrow M$ with compact support which is isotopic to the identity such that the following diagram commutes:
\[\xymatrix{ & D^n\ar[dl]_{\phi}\ar[dr]^{\phi'}& \\M\ar[rr]^h & & M
}\]
\end{prop}

\begin{rem}\label{rem:multipledisks}
Note that for $n\geq 2$ the result immediately generalizes to embeddings of finite disjoint unions of $D^n$: to see this assume we have disjoint embeddings $\phi_1,\ldots,\phi_k\colon D^n\rightarrow M$ with bicollared boundary spheres. Then removing $\phi_i(D^n)$ for $i=1,\ldots,k-1$ leaves us with a topological manifold which is still connected. Now given some other $\phi'$ as above whose orientation matches that of $\phi_k$ we may apply Proposition \ref{prop:diskscooter}. The resulting homeomorphism of $M\backslash \left(\bigcup_i \phi_i(D^n)\right)$ has compact support and hence extends to all of $M$ via the identity on the $\phi_i(D^n)$. Now apply this inductively.
\end{rem}

Before coming to the proof we note that indeed this has the following immediate consequence:

\begin{cor}\label{prop:consumunique}
Let $M,N$ be two $(n+1)$-dimensional oriented manifolds with boundary. Let \[\phi,\phi'\colon \bigsqcup_{i=1}^k D^n\rightarrow \partial M\quad \text{ and }\quad\psi,\psi'\colon \bigsqcup_{i=1}^k D^n\rightarrow \partial N\]
be embeddings of the disjoint union of $k$ $n$-disks, with bicollared boundary spheres and assume $n\geq 2$ if $k\geq 2$.
Assume further that on each disk $\phi,\phi'$ as well as $\psi,\psi'$ map to the same path component of the respective boundaries and are either both orientation preserving or orientation reversing. Then the pushouts $M\cup_{\phi,\psi} N$ and $M\cup_{\phi',\psi'} N$ are homeomorphic.
In particular the homeomorphism type of the boundary connected sum does not depend on the choice of disks.
\end{cor}

\begin{proof}
It follows from Proposition \ref{prop:diskscooter} and Remark \ref{rem:multipledisks} that there is homeomorphism of $\partial M$ isotopic to the identity which transforms $\phi$ into $\phi'$. Using a collar neighbourhood of $\partial M$ (see \cite{Brown} for the existence in the topological category) and the isotopy we may extend this to a homeomorphism on all of $M$, setting it to be the identity away from the boundary. The same can be done for $N$ and the two homeomorphisms together define the desired homeomorphism on the pushouts.
\end{proof}

We will need the following statement for the proof of Proposition \ref{prop:diskscooter}.

\begin{prop}\label{prop:homeomorphismsofsphere}\begin{enumerate}
\item
An orientation preserving homeomorphism of $S^n$ is isotopic to the identity.
\item An orientation preserving homeomorphism of $D^n$ is isotopic to the
	identity. 
\end{enumerate}
\end{prop}
\begin{proof} We begin with the proof of $(i)$ by following the arguments of\
	\cite{39959}. It suffices to prove this for a homeomorphism $\psi\colon
	S^n\rightarrow S^n$ fixing a point $p$ as we can move $p$ anywhere using
	homeomorphisms isotopic to the identity. Restricting $\psi$ to
	$S^n\backslash\{p\}\cong \mathbb{R}^n$, we apply the Stable Homeomorphism
	Conjecture (This conjecture was proven by Kirby \cite{MR242165} for
	$n \geq 5$, by Quinn \cite{MR679069} for $n=4$, by Moise \cite{MR48805} for
	$n=3$ and by Rad\'{o} for $n=2$ \cite{MR1544659}): It lets us write
$\psi|_{\mathbb{R}^n}$ as a finite composition of stable homeomorphisms, i.e.\
homeomorphisms which are the identity on a nonempty open subset. Any of these
induces a homeomorphism on the one point compactification $S^n$ hence we obtain
a factorization of $\psi$. But any homeomorphism of $S^n$ which is the identity
on some open subset $U$ is isotopic to the identity. Indeed, by removing a small
open disk from $U$ this reduces to the question whether a homeomorphism $h$ of
$D^n$ which is the identity on the boundary can be isotoped to the identity.
This can be achieved using the Alexander trick, i.e.\ via the isotopy with $H_t(x)=x$ for $\|x\|\geq t$ and $H_t(x)=t\cdot h(x/t)$ for $\|x\|\leq t$.

Part $(ii)$ follows from part $(i)$ in the following way: given an orientation
preserving homeomorphism $\psi\colon D^n\rightarrow D^n$ we know that its
restriction to the boundary is isotopic to the identity. Thus the same holds for
the radial extension of this restriction when considered as a homeomorphism of
$D^n$. Composing with its inverse reduces the proof to the case that
$\psi|_{S^{n-1}}$ is the identity. Applying the Alexander trick again yields the result.
\end{proof}

\begin{proof}[Proof of Proposition \ref{prop:diskscooter}]
We follow the proof outlined in \cite{121635}, see also \cite{TopManifolds}.
Given two points in an Euclidean neighbourhood we find a compactly supported
ambient isotopy mapping one onto the other. These extend to ambient isotopies on
$M$ via the identity. Choose a path between the centers (i.e.\ the images of the
origin) of $\phi$ and $\phi'$. Then we may move one of the centers along the
path by concatenating finitely many ambient isotopies of the above type. Hence
we may assume $\phi(0)=\phi'(0)$.

Using that the boundary spheres are bicollared, we extend $\phi,\phi'$ to embeddings $\tilde{\phi},\tilde{\phi}'\colon \mathbb{R}^n\rightarrow M$ onto open neighbourhoods by gluing the outside of the collar $S^{n-1}\times [0,\infty)$ to the boundary of $D^n$ and mapping it onto the tubular neighbourhoods. Using the coordinates of $\tilde{\phi}$ we may shrink $\phi(D^n)$ to be arbitrarily close to $\phi(0)$ using a compactly supported ambient isotopy which extends to $M$ via the identity. Hence we may assume $\phi(D^n)$ is contained in the interior of $\phi'(D^n)$.

It follows from the Annulus Theorem (note that this theorem is equivalent to the
Stable Homeomorphism Conjecture by \cite{MR158384}) that the closure of $\phi'(D^n)\backslash \phi(D^n)$ is homeomorphic to $S^{n-1}\times [0,1]$. In fact, gluing respective sides of the collars around $\phi(S^{n-1})$ and $\phi'(S^{n-1})$ we may extend this to a homeomorphsim of $S^{n-1}\times \mathbb{R}$ onto an open neighbourhood of $\phi'(D^n)\backslash \phi(D^n)$ with $S^{n-1}\times\mathbb{R}_-$ belonging to $\phi(D^n)$. Within this neighbourhood we may find a compactly supported ambient isotopy pushing $S^{n-1}\times [0,1]$ onto $S^{n-1}\times[-1,0]$. Thus we may assume that $\phi(D^n)=\phi'(D^n)$.

Now $\phi^{-1}\circ\phi'$ is an orientation preserving homeomorphism of $D^n$. By Proposition \ref{prop:homeomorphismsofsphere} it is isotopic to the identity. Using the outside collar of $\phi(D^n)$ this isotopy extends to a compactly supported ambient isotopy on all of $M$ starting at the identity. Hence we find a compactly supported homeomorphism of $M$ transforming $\phi$ into $\phi'$ which is isotopic to the identity.
\end{proof}

\bibliographystyle{acm}
\bibliography{GKMrealization}
\end{document}